\documentclass[a4paper]{article}
\usepackage{geometry}
\usepackage{geometry}

\usepackage[utf8]{inputenc}
\usepackage[utf8]{inputenc}
\usepackage{csquotes}
\usepackage{tikz-cd}
\usepackage{graphicx}
\usepackage{subcaption}
\usepackage{amsmath} 
\usepackage{amsmath}
\usepackage[english]{babel}
\usepackage{marvosym}
\usepackage{amssymb}
\usepackage{graphicx}
\usepackage[section]{placeins}
\usepackage{hyperref}
\usepackage{mathtools}
\usepackage[utf8]{inputenc}
\hypersetup{
 colorlinks,
 linkcolor={red!50!black},
 citecolor={blue!50!black},
 urlcolor={blue!80!black}
}

\usepackage[backend=biber,style=alphabetic,sorting=nyt]{biblatex}
\usepackage{hyperref}
\usepackage[english]{babel}
\usepackage{amsthm}
\usepackage{color, soul}
\newtheorem{theorem}{Theorem}[section]

\newtheorem{proposition}[theorem]{Proposition}
\newtheorem{corollary}[theorem]{Corollary}

\newtheorem{assumption}[theorem]{Assumption}

\newtheorem{definition}[theorem]{Definition}
\newtheorem{convention}[theorem]{Convention}
\newtheorem{procedure}[theorem]{Procedure}
\makeatletter
\def\th@remark{%
  \thm@headfont{\bfseries}%
  \normalfont 
  \thm@preskip\topsep \divide\thm@preskip\tw@
  \thm@postskip\thm@preskip
}
\makeatother

\theoremstyle{remark}
\newtheorem{remark}[theorem]{Remark}

\usepackage{mathtools}
\DeclarePairedDelimiter\ceil{\lceil}{\rceil}
\DeclarePairedDelimiter\floor{\lfloor}{\rfloor}

\usepackage{float}
\usepackage{mathtools}
\usepackage{amsmath,amscd,amssymb,amsfonts,amsthm,hyperref}

\usepackage{hyperref}
\usepackage[all]{xy}

\title{Index Calculations}
\newcommand{\mcal}{\mathcal}
\newcommand{\bb}{\mathbb}
\newcommand{\ep}{\epsilon}
\newcommand{\db}{\bar{\partial}}
\newcommand{\p}{\partial}
\newcommand{\dt}{\delta}
\newcommand{\la}{\langle} 
\newcommand{\ra}{\rangle }

\newcommand{\cas}[1]{#1^{\text{\Lightning}}}

\usepackage{geometry}
\usepackage{fullpage}
\title{Computing Embedded Contact Homology in Morse-Bott Settings}
\author{Yuan Yao}

\begin{document}
\maketitle
\begin{abstract}
Given a contact three manifold $Y$ with a nondegenerate contact form $\lambda$, and an almost complex structure $J$ compatible with $\lambda$, its embedded contact homology $ECH(Y,\lambda)$ is defined (\cite{bn}) and only depends on the contact structure. In this paper we explain how to compute ECH for Morse-Bott contact forms whose Reeb orbits appear in $S^1$ families, assuming the almost complex structure $J$ can be chosen to satisfy certain transversality conditions (this is the case for instance for boundaries of concave or convex toric domains, or if all the curves of ECH index one have genus zero).
We define the ECH chain complex for a Morse-Bott contact form via an enumeration of ECH index one cascades. We prove using gluing results from \cite{Yaocas} that this chain complex computes the ECH of the contact manifold.  This paper and \cite{Yaocas} fill in some technical foundations for previous calculations in the literature (\cite{choi2016combinatorial}, \cite{ECHT3}). 
\end{abstract}

\tableofcontents
\section{Introduction}
\subsection{Embedded contact homology}
In this article we develop some tools to compute the embedded contact homology (ECH) of contact 3-manifolds in Morse-Bott settings.

ECH is a Floer theory defined for a pair $(Y,\lambda)$, where $Y$ is a three dimensional contact manifold with nondegenerate contact form $\lambda$ (for an introduction see \cite{bn}). The ECH chain complex is generated by orbit sets of the form $\alpha = \{(\gamma_i,m_i)\}$. Here $\gamma_i$ are distinct simply covered Reeb orbits of $\lambda$; and the $m_i$ is a positive integer which we call the multiplicity of $\gamma_i$. To describe the differential, consider the symplectization $(\bb{R}\times Y, d(e^s \lambda))$ of $Y$ with almost complex structure $J$. Here $s$ denotes the variable in the $\bb{R}$ direction; and $J$ is a generic $\lambda$-compatible almost complex structure (see Definition \ref{compatibleJ}). The differential of ECH, which we write as $\p$, is defined by counting holomorphic currents of ECH index $I=1$ in the symplectization. More precisely, the coefficient
$\langle \p \alpha,\beta \rangle $
is defined by counts of $J$-holomorphic currents
that approach $\alpha$ as $s\rightarrow \infty$ and $\beta$ as $s\rightarrow -\infty$, where convergence to $\alpha,\beta$ is in the sense of currents.
The resulting homology, which we write as $ECH(Y,\xi)$, is an invariant of the contact structure $\xi = \textup{ker} \lambda$.
See Section \ref{ECH review} below for a more precise review of ECH and the ECH index.

In part due to its gauge theoretic origin, ECH has had spectacular applications to understanding symplectic problems and dynamics in low dimensions; for instance sharp symplectic embedding obstructions of four dimensional symplectic ellipsoids (\cite{Mcduffemb}), closing lemmas for Reeb flows on contact 3-manifolds (\cite{irie}), the Arnold chord conjecture (\cite{arnoldchord1,arnoldchord2}), and quantitative refinements of the Weinstein conjecture \cite{1Reeb2}. Several computations (e.g. \cite{ECHT3, choi2016combinatorial,lebow}) and applications (e.g. \cite{beyondech}) of ECH have assumed results from its Morse-Bott version, which we develop in detail in this paper.

\subsection{Morse-Bott theory}
The original definition of ECH requires we use non-degenerate contact forms. However, in practice many contact forms we encounter carry Morse-Bott degeneracies, for which the Reeb orbits are no longer isolated but instead show up in families with weaker non-degeneracy conditions imposed (for a more precise description, see Definition 3.2 in \cite{oh_wang_2018}). Although all Morse-Bott contact forms can be perturbed to non-degenerate ones, it is often useful to be able to compute ECH directly in the Morse-Bott setting, where often the enumeration of $J$-holomorphic curves is easier.

For ECH, since we only consider 3-manifolds, the two Morse-Bott cases are either when the Reeb orbits come in a two dimensional family, or come in one dimensional families. For the first case it then follows that the entire contact manifold is foliated by periodic Reeb orbits. ECH with this kind of Morse-Bott degeneracy has been computed in many cases by \cite{nelson2020embedded}, see also \cite{farris}.

The other case is when Reeb orbits show up in one dimensional $S^1$ families, i.e. we see tori foliated by Reeb orbits. We shall call these tori Morse-Bott tori. It is with this case we concern ourselves in this paper (for a description of what the contact form looks like, see Proposition \ref{prop_locform}). Examples of this include boundaries of toric domains, and torus bundles over the circle see \cite{Hermann, intoconcave, danconvex, lebow}.

For now we consider $(Y^3,\lambda)$ a contact 3-manifold where $\lambda$ is a Morse-Bott contact form all of whose Reeb orbits appear in $S^1$ families. Later for the case of boundary of convex or concave toric domains (Sections \ref{section concave},\ref{section convex}) we allow the case of both nondegenerate Reeb orbits and $S^1$ families of Reeb orbits. We consider the symplectization with a generic $\lambda$ compatible almost complex structure $J$ (see Definition \ref{compatibleJ})
\[
(\bb{R} \times Y^3, d(e^s \lambda)).
\]
Following the recipe described in \cite{BourPhd}, to compute ECH in the Morse-Bott setting we shall count holomorphic cascades of ECH index one. The philosophy behind this is as follows: given $\lambda$, a Morse-Bott contact form with Reeb orbits in Morse-Bott tori, we can perturb
\[
\lambda \longrightarrow \lambda_\dt
\]
where $\lambda_\dt$ with $\dt>0$ is a nondegenerate contact form up to a certain action level $L>>0$. This perturbation requires the following information. For each circle of orbits parameterized by $S^1$, choose a Morse function $f$ on $S^1$ with two critical points. The effect of this perturbation is so that each Morse-Bott torus splits into two nondegenerate Reeb orbits (corresponding to the critical points of $f$): one is an elliptic orbit and the other is a hyperbolic orbit. We also need to perturb the $\lambda$-compatible almost complex structure on the symplectization into a $\lambda_\dt$ compatible almost complex structure, $J_\dt$. Since $\lambda_\dt$ is nondegenerate up to action $L$, we can define the ECH chain complex up to action $L$ in this case by counting ECH index one $J_\dt$-holomorphic curves. The idea is to take $\dt \rightarrow 0$ and see what these ECH index one holomorphic curves degenerate into.

By a compactness theorem in \cite{SFT} (see also \cite{BourPhd, Yaocas}), such $J_\dt$-holomorphic curves degenerate into $J$-holomorphic cascades. For a definition of $J$-holomorphic cascade, see \cite{Yaocas}. Roughly speaking, a $J$-holomorphic cascade, which we shall write as $\cas{u}$, consists of a sequence of $J$-holomorphic curves $\{u^1,..,u^n\}$ that have ends on Morse-Bott tori. We think of the curves $u^i$ as living on different levels, with $u^i$ one level above $u^{i+1}$. Between adjacent levels there is the data of a single number $T_i\in [0,\infty]$ described as follows.  Suppose a positive end of $u^{i+1}$ is asymptotic to a simply covered Reeb orbit $\gamma$ with multiplicity $n$. This $\gamma$ corresponds to a point on $S^1 $ (the $S^1$ that parameterizes the family of Morse-Bott Reeb orbits). Then if we follow the upwards gradient flow of $f$ for time $T_i$ starting at the point corresponding to the Reeb orbit $\gamma$, we arrive at a Reeb orbit $\tilde{\gamma}$, and a negative end of $u^i$ is asymptotic to $\tilde{\gamma}$ with the same multiplicity $n$. We assume all positive ends of $u^{i+1}$ and negative ends of $u^{i}$ are matched up in this way. For an illustration of a cascade\footnote{This figure and the accompanying explanations are taken from Figure 1 in \cite{Yaocas}.}, see Figure 1. 

\begin{figure}[h]
\centering
\includegraphics[width=.4\linewidth]{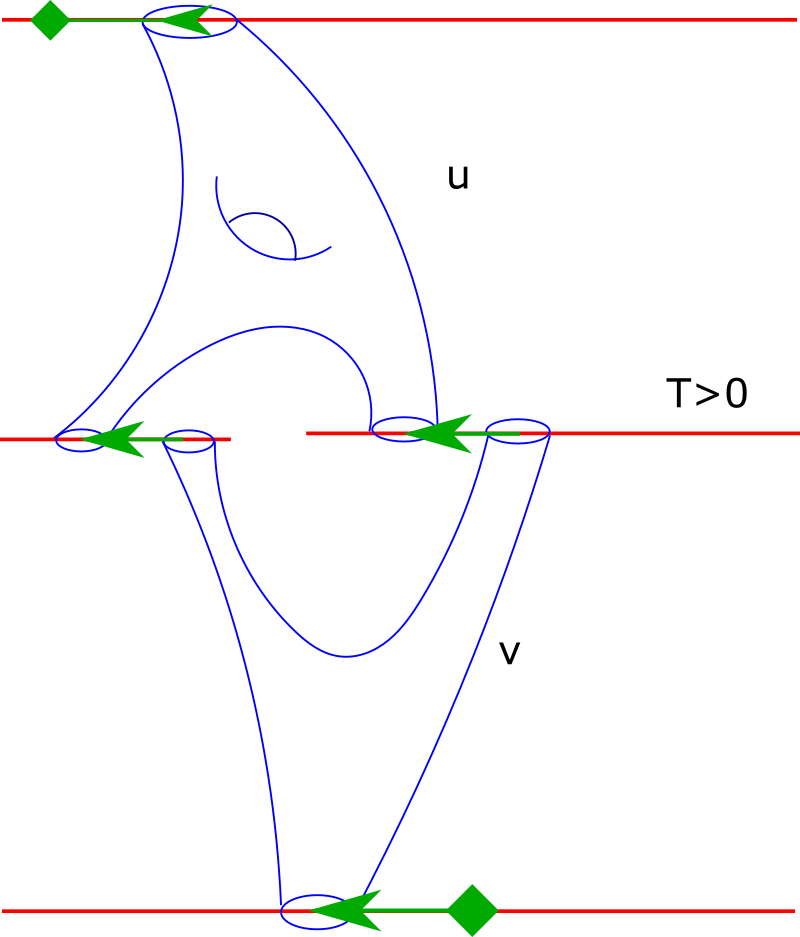}
\caption{A schematic picture of a cascade: the cascade $\cas{u}$ consists of two levels, $u$ and $v$. Horizontal lines correspond to Morse-Bott tori. Moving in the horizontal direction along these horizontal lines corresponds to moving to different Reeb orbits in the same $S^1$ family. Arrows correspond to gradient flows, and diamonds correspond to critical points of Morse functions on $S^1$ families of Reeb orbits. Between the holomorphic curves $u$ and $v$, there is a single parameter $T$ that tells us how long positive ends of $v$ must follow the gradient flow to meet a negative end of $u$.}
\label{fig:cascade}
\end{figure}
\subsection{Main results}
The Morse-Bott ECH chain complex which we write as $(C_*^{MB},\p_{MB})$ (see section \ref{section:computing using cascades}) can be described as follows. Its generators are collections of Morse-Bott tori, equipped with a multiplicity and additional data, which we write as $\alpha = \{(\mcal{T}_j,\pm, m_j)\}$. Here $\mcal{T}_j$ denotes a Morse-Bott torus; we call $m_j$ the multiplicity; and a choice of $+$ or $-$. See Section \ref{MBT as ECH} for a description. Suppose we can choose a $\lambda$ compatible almost complex structure $J$ which is ``good'' (see definition \ref{def:good j}), meaning certain transversality conditions (Definition \ref{def:transversality conditions}) are satisfied. The differential in the Morse-Bott chain complex $\p_{MB}$ counts ECH index one cascades between Morse-Bott ECH generators. The ECH index of a cascade is described in Section \ref{Section:ECH index}. We describe what it means for an cascade to be asymptotic to a Morse-Bott ECH generator in Section \ref{MBT as ECH}. For a description of what ECH index one cascades look like, see Corollary \ref{conditions on currents}, Prop. \ref{nice cascades}.
We prove that
\begin{theorem} \label{maintheorem_intro}
Let $\lambda$ be a Morse-Bott contact form on the contact 3-manifold $Y$ whose Reeb orbits all appear in $S^1$ families. Assuming the almost complex structure $J$ is good (see Definition \ref{def:good j}), the homology of the Morse-Bott ECH chain complex computes the ECH of the contact manifold $ECH(Y,\xi)$.
\end{theorem}
A slightly more precise version of this theorem that we prove is Theorem \ref{theorem:cobordism in general}.

We next find some instances there is enough transversality to compute ECH using the Morse-Bott chain complex.

\begin{theorem} \label{thm:list of transversality conditions}
Let $\lambda$ be a Morse-Bott contact form on the contact 3-manifold $Y$ whose Reeb orbits all appear in $S^1$ families.
We can choose a generic $J$ so that
\begin{itemize}
    \item Every reduced cascade (See Definition \ref{def:reduced cascade}) of $\leq 3$ levels is transversely cut out (see Definition \ref{def:transversality conditions}).
    \item Every reduced cascade where all of the (nontrivial) $J$-holomorphic components of the reduced cascade (in all of its levels) are distinct up to translation in the symplectization direction is transversely cut out (see Definition \ref{def:transversality conditions}).
\end{itemize}
If we can show through some other means that we can choose a small perturbation of $J$ to $J_\dt$ satisfying conditions of Theorem \ref{generic path J} so that for small enough $\dt$, all ECH index one curves degenerate into cascades whose reduced version satisfy either of the above conditions, then consider the Morse-Bott ECH chain complex $(C_*^{MB},\p_{MB})$ as described more precisely in Section \ref{section:computing using cascades}. For the differential $\p_{MB}$, if we restrict to ``good'' cascades (see Sections \ref{Section:ECH index}, \ref{section:computing using cascades} for the notion of ``good'') of ECH index one whose reduced versions are of the above form, the differential is well defined and the chain complex $(C_*^{MB},\p_{MB})$ computes $ECH(Y,\xi)$.
\end{theorem}

For a discussion how these conditions arise and a proof of this theorem, see the Appendix.
This list is by no means exhaustive. We expect there are many other situations where transversality can be achieved; the particulars will depend on the specific details of the contact manifold for which we are computing the ECH chain complex. In particular, for the case relevant for boundaries of convex and concave toric domains, we have the following:

\begin{theorem}
Let $\lambda$ be a contact form on the contact 3-manifold $Y$ whose Reeb orbits apppear either in Morse-Bott $S^1$ families or are non-degenerate. Let $\dt>0$, and $\lambda_\dt$ be the nondegenerate perturbation of $\lambda$ that perturbs each $S^1$ family of Reeb orbits into two nondegenerate ones. If for $\dt>0$ small enough, all $J_\dt$ holomorphic curves of ECH index one in $\bb{R} \times Y^3$ have genus zero, then the embedded contact homology of $Y$ can be computed from the Morse-Bott chain complex $(C_*^{MB,tree},\p_{MB}^{tree})$ (see Section \ref{section:ECH index one curves of genus zero}) using an enumeration of tree like cascades.
\end{theorem}

To be more precise, for the above theorem we need to use a slightly different description of cascades which we call ``tree like'' cascades, which is explained in Sections \ref{section:ECH index one curves of genus zero}, \ref{section concave}, \ref{section convex}.
Consequently, we can prove
\begin{theorem}
For boundaries of concave toric domains or convex toric domains, in the nondegenerate case after a choice of generic almost complex structure all curves of ECH index one have genus zero. Therefore the ECH of boundaries of concave/convex toric domains can be computed using the Morse-Bott ECH chain complex $(C_*^{MB,tree},\p_{MB}^{tree})$, via counts of tree-like ECH index one cascades.
\end{theorem}
For a definition of convex and concave toric domains, see Sections \ref{section concave}, \ref{section convex}.

We mention some previous computations of ECH that have assumed Morse-Bott theory of the flavour we develop in this paper, notably in \cite{ECHT3} for the case of $T^3$, and \cite{choi2016combinatorial} for certain toric contact 3-manifolds, and \cite{lebow} for the case of $T^2$ bundles over $S^1$. This paper and the gluing paper \cite{Yaocas} fill in the foundations for these results.

\begin{remark}
The above theorems say for genus zero curves we have all the transversality we need by simply restricting to cascades of ECH index one and choosing a generic $J$; however this result is not strict, there could well be other scenarios where transversality can be achieved. For instance we expect with some more care we can show the moduli space of cascades of ECH index one and genus one can be shown to be transverse. For discussion of general difficulties see the Appendix.
\end{remark}

\subsection{Some technical details}

For ECH in the nondegenerate setting (see \cite{bn}), as we review in Section \ref{ECH review}, the Fredholm index of a somewhere injective curve is bounded from above by its ECH index. Further, the ECH index is superadditive under unions of $J$-holomorphic curves in symplectizations. Using the fact that after choosing a generic almost complex structure, all somewhere injective curves are transversely cut out, it follows that by restricting to only ECH index one curves we do not need to consider multiply covered nontrivial curves. With this, one defines the ECH differential in the nondegenerate setting via counts of ECH index one $J$-holomorphic curves.

Parts of the above story continue to hold in the case of cascades if we assume can choose $J$ to be good (Definition \ref{def:good j}), as we explain below.

We first note that the notion of an ECH index continues to make sense for cascades, as we explain in Section \ref{Section:ECH index}. The case of cascades, however, is more complicated, in two directions.
\begin{itemize}
    \item During the degeneration process for $\lambda_\dt$ as $\dt\rightarrow 0$, simple curves may degenerate into cascades that have multiply covered components;
    \item For generic $J$, and even if we restrict to cascades all of whose curves are somewhere injective, the cascade need not be transversely cut out. 
\end{itemize}
The second bullet point is the most problematic. This happens because by requiring there is a single parameter between adjacent levels, we are imposing restrictions on the evaluation maps on the ends of the curves in a cascade. Hence a cascade lives in a fiber product, which need not be transversely cut out even if we restrict to only somewhere injective curves. For an explanation of this, see the Appendix. 

However, if we take as an \emph{assumption} that $J$ is good (which isn't always possible, it will depend on the specific contact manifold), then all cascades built out of somewhere injective curves that we consider are transversely cut out. Then we can address the first bullet point by using a version of the ECH index inequality for cascades .

To explain the ECH index inequality for cascades, consider the following. Given a cascade, we can pass to a reduced cascade, which means we replace all multiply covered curves with the underlying simple curves. See Section \ref{degenerations} for a precise description of this process. The reduced cascade also lives in a fiber product because of the conditions we imposed on its ends. By the assumption that $J$ is good (and consequently transversality assumptions in Definition \ref{def:transversality conditions} are satisfied), the reduced cascade is transversely cut out. To each reduced cascade we can associate to it a virtual dimension, which is the dimension of the moduli space of curves that lies in the same configuration as the reduced cascade. We prove that the ECH index of the cascade bounds the Fredholm index of the reduced cascade from above; and that equality holds only if the original cascade had no multiply covered components (and is well behaved in various ways, see Section \ref{Section:ECH index}).

In \cite{Yaocas}, we proved a correspondence theorem between certain cascades and $J$-holomorphic curves.

\begin{theorem}[\cite{Yaocas}]
Given a ``transverse and rigid'' (see Definition 3.4 in \cite{Yaocas}) height one $J$-holomorphic cascade $\cas{u}$ , it can be glued to a rigid $J_\dt$-holomorphic curve $u_\dt$ for $\dt>0$ sufficiently small. The construction is unique in the following sense: if $\{\dt_n\}$ is a sequence of numbers that converge to zero as $n\rightarrow \infty$, and $\{u'_{\dt_n}\}$ is sequence of $J_{\dt_n}$-holomorphic curves converging to $\cas{u}$, then for large enough $n$, the curves $u_{\dt_n}'$ agree with $u_{\dt_n}$ up to translation in the symplectization direction.
\end{theorem}

In this paper, using index calculations, we show that if $J$ is good (some instances of which are described in Theorems \ref{thm:list of transversality conditions}), then essentially all ECH index one cascades are transverse and rigid\footnote{Technically we need to restrict ourselves to \emph{good} ECH index one cascades. This is a fairly minor point, but see Proposition 5.32 and surrounding discussion.}. Thus the gluing theorem above is then used to show the Morse-Bott chain complex computes $ECH(Y,\lambda)$. In the cases where we use ``tree like'' cascades, for instance for boundaries of convex or concave toric domains, the definitions are slightly different, but essentially the same story holds true and we can always choose a generic $J$ so that the Morse-Bott chain complex computes $ECH(Y,\lambda)$.

Finally in the Appendix we explain why the usual techniques for achieving transversality fails for cascades. \\

\textbf{Acknowledgements}
I would like to thank my advisor Michael Hutchings for his consistent help and support throughout this project. I would like to acknowledge the support of the Natural Sciences and Engineering Research Council of Canada (NSERC), PGSD3-532405-2019.
Cette recherche a été financée par le Conseil de recherches en sciences naturelles et en génie du Canada (CRSNG), PGSD3-532405-2019.

\section{ECH review} \label{ECH review}
For a thorough introduction to ECH see \cite{bn}. We will summarize much of the material from \cite{bn} and \cite{Hutchings2002} for convenience of the reader.

Let $(Y^3,\lambda)$ be a contact 3 manifold with nondegenerate contact form $\lambda$. The generator of ECH are collections $\Theta$, where each $\Theta$ is a set of Reeb orbits with multiplicities
\[
\Theta:= \{ (\gamma_i,m_i) | \gamma_i \, \text{are pairwise distinct simply covered Reeb orbits},\, m_i \in \bb{Z}_+\}.
\]
We require $m_i=1$ if $\gamma_i$ is a hyperbolic orbit. 
Then the chain for ECH are just
\[
C_*(\lambda') := \bigoplus_{\Theta_i} \bb{Z}_2 \la \Theta_i \ra.
\]
\begin{remark}
There is a decomposition of ECH according to homology class of $\Theta_i$ in $H_1(Y)$. ECH can also be defined using $\bb{Z}$ coefficients. We will not address these issues here.
\end{remark}
Let $\alpha, \beta$ be ECH generators. Consider the symplectization of $Y$, defined as the symplectic manifold $(\bb{R} \times Y, \omega := d(e^a \lambda))$, where $a$ denotes the $\bb{R}$ coordinate. Equip it with a generic $\lambda$ compatible almost complex structure $J$. By compatible we mean the following
\begin{definition}\label{compatibleJ}
Let $\lambda$ be a contact form (not necessarily nondegenerate) on a contact 3-manifold. Let $J$ be a almost complex structure on the symplectization $(\bb{R} \times Y, \omega := d(e^a \lambda))$. We say $J$ is compatible with $\lambda$ if
\begin{enumerate}
    \item $J$ is invariant in the $\bb{R}$ direction;
    \item Let $R$ denote the Reeb vector field, then $J\partial_s =R$;
    \item Let $\xi$ denote the contact structure, then $J\xi =\xi$ and $d\lambda(\cdot, J\cdot)$ defines a metric on $\xi$. 
\end{enumerate}
\end{definition}

Then the coefficient $\la\p \alpha, \beta \ra$ is defined by

\begin{equation}
\la\p \alpha, \beta \ra
:=   \left\lbrace  
  \begin{tabular}{@{}l@{}}
 $\bb{Z}_2$\, \textup{count of holomorphic currents}\, $\mcal{C}$\, \textup{of ECH index} \,$I=1$,\\
\textup{so that as}  $s\rightarrow +\infty, \, \mcal{C}$ \,\text{approaches}\, $\alpha$  \textup{as a current, and as} $s\rightarrow -\infty$, \\
 $\mcal{C}$\, \text{approaches} $\beta$ \,\text{as a current}.
   \end{tabular}
  \right\rbrace
 \end{equation}
A holomorphic current $\mcal{C}$ is by definition a collection $\{(C_i,m_i)\}$ where each $C_i$ is a somewhere injective $J$ holomorphic curve and $m_i \in \bb{Z}_{>0}$ accounts for the multiplicity of this curve. The ECH index $I$ of a holomorphic curve $C$ (or more generally a relative 2 homology class in $H_2(\alpha,\beta,Y)$, see section below for a definition) is defined by
\begin{equation}
    I(C) := Q_\tau(C) +c_\tau(C) +CZ^I(C) 
\end{equation}
where $Q_\tau(C)$ is the relative intersection number, $c_\tau(C)$ is the relative Chern class, and $CZ$ is a sum of Conley Zehnder indices used in ECH. We will review these terms in the upcoming subsections.
\subsection{Relative first Chern class}
Let $\alpha,\beta$ be orbit sets. We define the relative homology group $H_2(\alpha,\beta,Y)$ to be the set of 2-chains $\Sigma$ with
\[
\p \Sigma = \alpha -\beta
\]
modulo boundary of 3 chains. This is an affine space over $H_2(Y)$, and each $J$ holomorphic curve defines a relative homology class.

We fix trivializations $\tau$ of the contact structure $\xi$ over each Reeb orbit in $Y$. We then define the relative first Chern class $c_\tau$ with respect to this choice of trivialization. For a given homology class in $H_2(\alpha,\beta,Y)$, choose a representative $Z\in H_2(\alpha,\beta,Y)$ that is embedded near its boundaries $\alpha,\beta$. We assume $Z$ is a smooth surface. Let $\iota: Z \rightarrow Y$ be the inclusion. Then consider the bundle $\iota^*\xi$ over $Z$. Let $\psi$ be a section of this bundle that is constant with respect to the trivialization $\tau$ near each of the Reeb orbits, and perturb $\psi$ so that all of its zeroes are transverse. Then $c_\tau(Z)$ is defined to be the algebraic count of zeroes of $\psi$. See \cite{bn} for a more thorough explanation and that this is well defined.
\subsection{Writhe}
Let $C$ be a somewhere injective $J$ holomorphic curve in the symplectization of $Y$, $(\bb{R}\times Y,d(e^a \lambda))$ (with generic
$\lambda$-compatible complex structure $J$) that is asymptotic to $\alpha$ as $s\rightarrow +\infty$ and $\beta$ as $s\rightarrow -\infty$. For simplicity we focus on $s\rightarrow +\infty$ end. It is known (see for example \cite{siefring}) that for $s$ sufficiently large, $C\cap \{s\}\times Y$ is a union of embedded curves near each orbit of $\alpha$. For each orbit $\gamma_i$ of $\alpha$, the curves $C\cap \{s\}\times Y$ forms a braid $\xi_i^+$. We use the trivialization $\tau$ to identify the braids $\xi_i^+$ with braids in $S^1 \times D^2$. We can define the writhe of $\xi_i^+$ by identifying $S^1\times D^2$ with an annulus times an interval, projecting $\xi_i^+$ to the annulus, and counting crossings with signs. The same sign convention is clearly explained in \cite{hutching_revisited}.

Then given a somewhere injective $J$-holomorphic curve $C$ that is not the trivial cylinder, with braids $\zeta_i^+$ associated to the $i$-th Reeb orbit it approaches as $s\rightarrow +\infty$ and braids $\zeta_j^-$ associated to the $j$th Reeb orbit it approaches as $s\rightarrow -\infty$ we define its writhe to be
\[
w_\tau(C) := \sum_i w_\tau(\zeta_i^+) - \sum_j w_\tau (\zeta_j^-).
\]
We also recall the writhe of the braid $\zeta_i^+$ can be bounded by expressions in terms of the Conley-Zehnder indices.
\begin{proposition}
Let $C$ be a somewhere injective holomorphic curve that is not a trivial cylinder which is asymptotic to $\gamma_i$ with total multiplicity $n_i$. Suppose there are $k_i$ distinct ends of $C$ that are asymptotic to $\gamma_i$, with covering multiplicities $q_i^j$. Then the writhe associated to the braid $\zeta_i^+$ corresponding to Reeb orbit $\gamma_i$ is bounded above by
\begin{equation}
    w_\tau (\zeta_i^+) \leq \sum _j^ {n_i} CZ(\gamma_i^j) - \sum_j^{k_i}CZ(\gamma_i^{q_i^j})
\end{equation}
A similar bound holds for braids at $s\rightarrow -\infty$ with signs reversed.
\end{proposition}
We will derive an analogue of this bound for the Morse-Bott case. For now we recall another definition:
\begin{definition}
Let $C$ be a somewhere injective $J$-holomorphic curve that is not a trivial cylinder. For each $\gamma_i$ that $C$ is asymptotic to as $s\rightarrow +\infty$, form the sum $CZ^I(\gamma_i): = \sum _{j=1}^ {n_i} CZ(\gamma_i^j)$ as above, and for each $\gamma_i'$ that $C$ is asymptotic to as $s\rightarrow -\infty$, we form an analogous sum, then we define
\begin{equation}
    CZ^I(C) : = \sum_{\substack{\gamma_i, \\C \,
   \text{is asymptotic to}\, \gamma_i,\\
   \text{as } \, s\rightarrow +\infty}}CZ^I(\gamma_i) - \sum_{\substack{\gamma_i', \\C \,
   \text{is asymptotic to} \, \gamma_i',\\
   \text{as } \, s\rightarrow -\infty}}CZ^I(\gamma_i').
\end{equation}
\end{definition}
This is the Conley-Zehnder index term that appears in the definition of ECH index.
\subsection{Relative adjunction formula}
In this section we review the relative adjunction formula (see \cite{bn, Hutchings2002}). We first review the notion of relative intersection pairing, which is a map depending on the trivialization $\tau$:
\[
Q_\tau: H_2(\alpha,\beta,Y) \times H_2(\alpha,\beta,Y) \rightarrow \bb{Z}
\]
as follows. Let $S$ and $S'$ be surfaces representing relative homology classes in $H_2(\alpha,\beta,Y)$. If we identify $\bb{R} \times Y$ with $ (-1,1)\times Y \subset [-1,1] \times Y$, then we have by definition
\[
\p S =\p S' = \sum_i m_i \{1\} \times \alpha_i -  \sum_i n_i \{-1\} \times \beta_i 
\]

We make the following requirements on the representatives $S$ and $S'$:
\begin{enumerate}
    \item The projections to $Y$ of the intersections of $S$ and $S'$ with $(1-\ep,1]\times Y$ and $[0,\ep) \times Y$ are embeddings.
    \item Each end of $S$ or $S'$ covers Reeb orbits $\alpha_i$ (resp $\beta_i$) with multiplicity $1$.
    \item The image of $S$ (after projecting to $Y$ in a neighborhood $S^1 \times D^2$ of $\alpha_i$ determined by the trivialization $\tau$) do not intersect, and do not rotate with respect to the chosen trivialization $\tau$ as one goes around $\alpha_i$. Further, the image of different ends of $S$ approaching $\alpha_i$ lie on distinct rays in a neighborhood of $\alpha_i$. More concretely using trivialization $\tau$ to identify a neighborhood of $\alpha_i$ with $S^1\times \bb{R}^2$, ends of $S$ approach $\alpha_i$ along different rays in $\bb{R}^2$.  We make a similar requirement for $\beta_i$. We make a similar requirement for $S'$.
    \item All interior intersections between $S$ and $S'$ are transverse.
\end{enumerate}
Representatives satisfying all of the above conditions are called $\tau$ -representatives in \cite{Hutchings2002}, which is a definition we will adopt. 
Then given $\tau$ representatives as listed above, $Q_\tau(S,S')$ is defined to be the algebraic count of intersections between $S$ and $S'$.

We are now ready to state the relative adjunction formula, see also \cite{Hutchings2002}. 
\begin{proposition}
If $C$ is a somewhere injective $J$ holomorphic curve,
\begin{equation}
    c_\tau(C)=\chi(C)+Q_\tau(C)+w_\tau(C) -2\dt (C)
\end{equation}
where $\dt(C)\geq 0$ is defined to be an algebraic count of singularities of $C$. Each singularity is positive due to the fact $C$ is $J$-holomorphic.
\end{proposition}
\subsection{ECH index inequality}
We have now defined all of the terms that appear in the ECH index inequality. We compare this with the Fredholm index.
Let $C$ be a somewhere injective $J$-holomorphic  curve, let $Ind(C)$ denote the Fredhom index of $C$, which in this case is given by
\[
-\chi(C) + 2c_\tau(C) + CZ^{Ind}(C).
\]
Here $CZ^{Ind}(C)$ is defined as follows. If $C$ is positively asymptotic to $\gamma$ with $k$ ends, each of multiplicity $q_k$, then the contribution to $CZ^{Ind}(C)$ from $\gamma$ is given by $\sum_k CZ(\gamma^{q_k})$. Similarly if $C$ is asymptotic to $\gamma$ at the negative ends, then its contribution to $CZ^{Ind}(C)$ is $-\sum_k CZ(\gamma^{q_k})$.
\begin{theorem}
Let $C$ denote a somewhere injective $J$-holomorphic curve as above,
then we have the following inequality
\begin{equation}
    Ind(C) \leq I(C) -2\dt (C).
\end{equation}
\end{theorem}
An immediate corollary of the above is
\begin{corollary}
Let $\mcal{C}$ be a $J$-holomorphic current of $I(\mcal{C}) =1$. Then for generic $J$, the current $\mcal{C}$ must satisfy
\begin{enumerate}
    \item It contains an unique connected embedded curve $C$ of multiplicity one that is not a trivial cylinder. The ends of $C$ approach Reeb orbits according to partition conditions. (See \cite[Section 3]{bn} for a discussion of partition conditions). We will review the relevant partition conditions in the Morse-Bott setting later).
    \item The other components of $\mcal{C}$ are trivial cylinders with multiplicities.
\end{enumerate}
\end{corollary}
\begin{convention}\label{nontrivial}
In this paper we describe a correspondence between ECH index 1 currents in the nondegenerate setting and ECH index 1 cascades in the Morse-Bott setting. We will only care about the nontrivial part of the ECH index 1 current, as the trivial cylinders correspond trivially in the non-degenerate and Morse-Bott situations. Hence when we say cascade, or a sequences of ECH index one curves/currents degenerating into a cascade, unless stated otherwise, we will always be considering what happens to the nontrivial part of the ECH index one current, and what cascade it corresponds to.
\end{convention}

\subsection{\texorpdfstring{$J_0$}{J0} index and finiteness}
We recall (without proof) the following proposition (see \cite{Hutchings2002},\cite{bn}):
\begin{proposition}
Let $\alpha,\beta$ be ECH generators. We choose a generic $J$, and let $\mcal{M}^{I=1}(\alpha,\beta)/\bb{R}$ denote the moduli space of ECH index $=1$ currents from $\alpha$ to $\beta$ modulo the action of $\bb{R}$. Then $\mcal{M}^{I=1}(\alpha,\beta)/\bb{R}$ is a finite collection of points.
\end{proposition}
We will mention two results that go into this proof, for we will need analogous constructions in the Morse-Bott context.
\begin{definition}
Let $\alpha =\{(\alpha_i,m_i)\},\beta=\{(\beta_i,n_i)\}$ be ECH generators, let $Z\in H_2(\alpha, \beta, Y)$ be a relative homology class. We define:
\begin{equation}
    J_0(\alpha,\beta,Z) =-c_\tau(Z)+Q_\tau(Z) + CZ^{J_0}(\alpha,\beta)
\end{equation}
where
\begin{equation}
    CZ^{J_0}(\alpha,\beta):=\sum_i\sum_{k=1}^{m_i-1}CZ(\alpha_i^k)-\sum_i\sum_{k=1}^{n_i-1}CZ(\beta_i^k)
\end{equation}
\end{definition}
We have the following proposition bounding the topological complexity of holomorphic curves counted by ECH index 1 conditions:
\begin{proposition}
Let $\mcal{C} \in \mcal{M}^{I=1}(\alpha,\beta)$, which decomposes as $\mcal{C} = C_0 \cup C$ where $C_0$ is a union of trivial cylinders, and $C$ is somewhere injective and nontrivial. Let $n_i^+$ denote the number of positive ends $C$ has at $\alpha_i$, plus 1 if $C_0$ includes cylinders of the form $\bb{R} \times \alpha_i$, define $n_j^-$ analogously for $\beta$ and negative ends of $C$ then
\begin{equation}
    -\chi(C)  +\sum_i(n_i^+-1)+\sum_j(n_j^--1) \leq J_0(C).
\end{equation}
\end{proposition}
Finally we state the version of Gromov compactness for currents. Let $\alpha, \beta$ be orbit sets, we define a broken holomorphic current from $\alpha,\beta$ to be a finite sequence of $J$ nontrivial holomorphic currents $(\mcal{C}^0,..,\mcal{C}^k)$ in $\bb{R} \times Y$ such that there exists orbit sets $\alpha =\gamma^0,\gamma^1,..,\gamma^{k+1}=\beta$ so that $\mcal{C}^i \in \mcal{M}(\gamma^i,\gamma^{i+1})$ (this notation means $\mcal{C}^i$ is a current from the orbit set $\gamma^i$ to $\gamma^{i+1}$). By nontrivial we mean a current is not entirely composed of unions of trivial cylinders. We say a sequence of holomorphic currents $\{\mcal{C}_{v\geq1} \} \in \mcal{M}(\alpha,\beta)$ converges to $(\mcal{C}^0,..,\mcal{C}^k)$ if for each $i=0,..,k$ there are representatives $\mcal{C}_\nu^i$ of $\mcal{C}_\nu \in \mcal{M}(\alpha,\beta)/\bb{R}$ such that the sequence $\{\mcal{C}_{v\geq1} \}$ converges as a current and as a point set on compact sets to $\mcal{C}^i$.
\begin{proposition}(\cite{bn}, \cite{Taubescompactness} Prop 3.3 )
Any sequence $\{\mcal{C}_v\} $ of holomorphic currents in $\mcal{M}(\alpha,\beta)/\bb{R}$ has a subsequence which converges to a broken holomorphic current $(\mcal{C}^0,..,\mcal{C}^k)$. Further if we denote $\{\mcal{C}_v\} $ the convergent subsequence,
we have the equality
\begin{equation}
    [\mcal{C}_v] = \sum_{i=0}^k[\mcal{C}^i] \in H_2(\alpha,\beta,Y)
\end{equation}
\end{proposition}

\section{Morse-Bott setup and SFT type compactness} \label{degenerations}
Let $(Y,\lambda)$ be a contact 3 manifold with Morse-Bott contact form $\lambda$. Throughout we assume the Morse-Bott orbits come in families of tori. 
\begin{convention}
Throughout this paper we fix action level $L>0$ and only consider ECH generators of action level up to $L$. This is implicit in all of our constructions and will not be mentioned further. We construct Morse-Bott ECH up to action level $L$, and the full ECH is recovered by taking $L \rightarrow \infty$.
\end{convention}

The following theorem, which is a special case of a more general result in  \cite{oh_wang_2018}, gives a characterization of the neighborhood of Morse-Bott Tori. Let $\lambda_0$ denote the standard contact form on $(z,x,y) \in S^1\times S^1 \times \mathbb{R}$ of the form
\[
\lambda_0= dz-ydx.
\]
\begin{proposition} \cite{oh_wang_2018} \label{prop_locform} 
Let $(Y,\lambda)$ be a contact 3 manifold with Morse-Bott contact form $\lambda$. We assume the Morse-Bott orbits come in families of tori $\mcal{T}_i$  with minimal period $T_i$. Then we can choose coordinates around each Morse-Bott torus so that a neighborhood of $\mcal{T}_i$ is described by $S^1\times S^1 \times (-\ep,\ep)$, and the contact form $\lambda$  in this coordinate system looks  like:
\[
\lambda = h(x,y,z) \lambda_0
\]
where $h(x,y,z)$ satisfies:
\[
h(x,0,z)=1, dh(x,0,z) =0
\]
Here we identify $z\in S^1 \sim \bb{R}/2\pi T_i \bb{Z}$
\end{proposition}
See \cite{Yaocas} Theorem Proposition 2.2 for a sketch of the proof.
By the Morse-Bott assumption there are only finitely many such tori up to fixed action $L$.
We assume we have chosen such neighborhoods around all Morse Bott Tori ${\mcal{T}_i}$. Next we shall perturb them to nondegenerate Reeb orbits by perturbing the contact form in a neighborhood of each torus as described below. This is the same perturbation as in \cite{Yaocas}.

Let $\delta>0$, let $f:x\in \mathbb{R}/\bb{Z} \rightarrow \mathbb{R}$ be a smooth Morse function with maximum at $x=1/2$ and minimum $x=0$. Let $g(y):\bb{R} \rightarrow \bb{R}$ be a bump function that is equal to $1$ on $[-\ep_{\mcal{T}_i},\ep_{\mcal{T}_i}]$ and zero outside $[-2\ep_{\mcal{T}_i},2\ep_{\mcal{T}_i}]$. Here $\ep_{\mcal{T}_i}$ is a small number chosen for each $\mcal{T}_i$ small enough so that the normal form in the above theorem applies to all Morse-Bott tori of action $<L$, and that all such chosen neighborhoods these Morse-Bott tori are disjoint. Then in neighborhood of the Morse-Bott tori $\mcal{T}_i$ we perturb the contact form as
\[
\lambda \longrightarrow \lambda_\delta:= e^{\delta gf}\lambda.
\]
We can describe the change in Reeb dynamics as follows:
\begin{proposition}
For fixed action level $L>0$ there exists $\delta>0$ small enough so that the Reeb dynamics of $\lambda_\delta$ can be described as follows. In the trivialization specified by Proposition \ref{locform}, each Morse-bott torus splits into two non-degenerate Reeb orbits corresponding to the two critical points of $f$. One of them is hyperbolic of index $0$, the other is elliptic with rotation angle $|\theta| <C\dt <<1$ and hence its Conley-Zehnder index is $\pm 1$. There are no additional Reeb orbits of action $<L$.
\end{proposition}
For proof see \cite{BourPhd}.

\begin{remark}
Later when we define various terms in the ECH index, they will depend on the choice of trivializations of the contact structure on the Reeb orbits. We will always choose the trivialization specified by Proposition \ref{prop_locform}. For convenience of notation we will call this trivialization $\tau$ and write for example $c_\tau$ or $Q_\tau$ for the definition of relative Chern class or intersection form with respect to this trivialization.

We also observe that after iterating the Reeb orbit in the Morse-Bott tori, their Robbin-Salamon index stays the same (\cite{Gutt2014}). So up to action $L$, in the nondegenerate picture, we will only see Reeb orbits of Conley-Zehnder index $-1,0,1$.
\end{remark}

\begin{definition}
We say a Morse Bott torus is positive if the elliptic Reeb orbit has Conley-Zehnder index 1 after perturbation; otherwise we say it is negative Morse Bott torus. This condition is intrinsic to the Morse-Bott torus itself, and is independent of trivializations or our choice of perturbations.
\end{definition}

We recall our goal is to define the ECH chain complex up to filtration $L$, and then take $L\rightarrow \infty$ to recover the entire ECH chain complex. Hence, let us consider for small $\delta>0$ the symplectization
\[
(M^4, \omega_\delta) := (\bb{R} \times Y^3,d(e^s\lambda_\delta))
\]
We equip $(M,\omega_\delta)$ with a $\lambda_\dt$ compatible almost complex structure $J_\delta$, and $(M,\omega):= (\bb{R} \times Y^3,d(e^s\lambda_\delta))$ with $\lambda$-compatible almost complex structure $J$. Both $J$ and $J_\delta$ should be chosen generically, with genericity condition specified in Definition \ref{def:transversality conditions} and Theorem \ref{generic path J}. In particular $J_\dt$ should be a small perturbation of $J$, i.e. the $C^\infty$ norm difference between $J_\dt$ and $J$ should be bounded above by $C\dt$.
For fixed $L$ and small enough and generic choice of $\delta$, the ECH of $(Y^3,\lambda_\delta)$ is defined for generators of action less than $L$ via counts of embedded J-holomorphic curves of ECH index 1. To motivate our construction, we next take $\delta \rightarrow 0$ to see what kinds of objects these $J$ holomorphic curves degenerate into. By a theorem of that first appeared in Bourgeois' thesis \cite{BourPhd} and also stated in \cite{SFT} (for a proof see the Appendix of \cite{Yaocas}), they degenerate into $J$-holomorphic cascades. (For a more careful definition of cascades see the appendix of \cite{Yaocas} that takes into account of stability of domain and marked points, but the definition here suffices for our purposes).

\begin{definition} [
\cite{BourPhd}, See also definition 2.7 in \cite{Yaocas}]
Let $\Sigma$ be a punctured (nodal) Riemann surface, potentially with multiple connected components.
A cascade of height 1, which we will denote by $\cas{u}$, in $(\bb{R}\times Y^3,d(e^s\lambda)$ consists of the following data :
\begin{itemize}
    \item A labeling of the connected components of $\Sigma ^*=\Sigma \setminus \{ \text{nodes} \}$ by integers in
    $\{1, . . . , l\}$, called sublevels, such that two components sharing a node have sublevels differing by at most 1. We denote by $\Sigma_i$ the union of connected components of sublevel $i$, which might itself be a nodal Riemann surface.
    \item $T_i \in [0,\infty)$ for $ i = 1, . . . , l - 1$.
    \item  $J$-holomorphic maps $u^i: (\Sigma_i, j) \rightarrow (\bb{R}\times Y^3, J)$ with $E(u_i) < \infty$ for $ i = 1, . . . , l$, such that:
    \begin{itemize}
        \item Each node shared by $\Sigma_i$ and $\Sigma_{i+1}$, is a negative puncture for $u^i$ and is a positive puncture for $u^{i+1}$. Suppose this negative puncture of $u^i$ is asymptotic to some Reeb orbit $\gamma_i \in \mcal{T}$, where $\mcal{T}$ is a Morse-Bott torus, and this positive puncture of $u^{i+1}$ is asymptotic to some Reeb orbit $\gamma_{i+1} \in \mcal{T}$, then we have that $\phi^{T_i}_f(\gamma_{i+1}) = \gamma_{i}$. Here $\phi^{T_i}_f$ is the upwards gradient flow of $f$ for time $T_i$ lifted to the Morse-Bott torus $\mcal{T}$. It is defined by solving the ODE
        \[
        \frac{d}{ds} \phi_f(s) = f'(\phi_f(s)).
        \]
        \item $u^i$ extends continuously across nodes within $\Sigma_i$.
        \item No level consists purely of trivial cylinders. However we will allow levels that consist of branched covers of trivial cylinders.
    \end{itemize}
\end{itemize}
\end{definition}
\begin{convention}
We fix our conventions as in \cite{Yaocas}.
\begin{itemize}
    \item We say the punctures of a $J$-holomorphic curve that approach Reeb orbits as $s\rightarrow \infty$ are positive punctures, and the punctures that approach Reeb orbits as $s\rightarrow -\infty$ are negative punctures. We will fix cylindrical neighborhoods around each puncture of our $J$-holomorphic curves, so we will use ``positive/negative ends'' and ``positive/negative punctures'' interchangeably.  By our conventions, we think of $u^1$ as being a level above $u^2$ and so on.
    \item We refer to the Morse-Bott tori $\mcal{T}_j$ that appear between adjacent levels of the cascade $\{u^i,u^{i+1}\}$ as above, where negative punctures of $u^i$ are asymptotic to Reeb orbits that agree with positive punctures from $u^{i+1}$ up to a gradient flow, \textit{intermediate cascade levels}.
    \item We say that the positive asymptotics of $\cas{u}$ are the Reeb orbits we reach by applying $\phi_f^\infty$ to the Reeb orbits hit by the positive punctures of $u^1$. Similarly, the negative asymptotics of $\cas{u}$ are the Reeb orbits we reach by applying $\phi_f^{-\infty}$ to the Reeb orbits hit by the negative punctures of $u^l$. They are always Reeb orbits that correspond to critical points of $f$. We note if a positive puncture (resp. negative puncture) of $u^1$ (resp. $u^l$) is asymptotic to a Reeb orbit corresponding to a critical point of $f$, then applying $\phi^{+\infty}_f$ (resp. $\phi_f^{-\infty}$) to this Reeb orbit does nothing.
\end{itemize}
\end{convention}

\begin{definition}[\cite{BourPhd}, Chapter 4, See also definition 2.9 in \cite{Yaocas} ] \label{def height k cascade}
A cascade of height $k$ consists of $k$ height 1 cascades, $\cas{u}_k =\{u^{1\text{\Lightning}},...,u^{k\text{\Lightning}}\}$ with matching asymptotics concatenated together. 

By matching asymptotics we mean the following. Consider adjacent height one cascades, $u^{i\text{\Lightning}}$ and $u^{i+1\text{\Lightning}}$. Suppose a positive end of the top level of $u^{i+1\text{\Lightning}}$ is asymptotic to the Reeb orbit $\gamma$ (not necessarily simply covered). Then if we apply the upwards gradient flow of $f$ for infinite time we arrive at a Reeb orbit reached by a negative end of the bottom level of $u^{i\text{\Lightning}}$. We allow the case where $\gamma$ is at a critical point of $f$, and the flow for infinite time is stationary at $\gamma$. We also allow the case where $\gamma$ is at the minimum of $f$, and the negative end of the bottom level of $u^{i\text{\Lightning}}$ is reached by following an entire (upwards) gradient trajectory connecting from the minimum of $f$ to its maximum. If all ends between adjacent height one cascades are matched up this way, then we say they have matching asymptotics.

We will use the notation $\cas{u}_k$ to denote a cascade of height $k$. We will mostly be concerned with cascades of height 1 in this article, so for those we will drop the subscript $k$ and write $\cas{u} = \{u^1,...,u^l\}$.
\end{definition}
\begin{remark}
As mentioned in \cite{Yaocas}, we can also think of a cascade of height $k$ as a cascade of height 1 where $k-1$ of the intermediate flow times are infinite.
\end{remark}

We now state a SFT style compactness theorem relating non-degenerate $J_\dt$ holomorphic curves to cascades. However, the precise statement is rather technical and requires us to take up Deligne-Mumford compactifications of the moduli space of Riemann surfaces. The full version is stated in \cite{SFT} (see also the Appendix of \cite{Yaocas}, where we also sketch a proof). For our purposes it will be sufficient to state the theorem informally as below.

\begin{theorem}
(See \cite{SFT})
Let $u_{\dt_n}$ be a sequence of $J_{\dt_n}$-holomorphic curves with uniform upper bound on genus and energy, then a subsequence of $u_{\dt_n}$ converges to a cascade of $J$- holomorphic curves (which can be apriori of arbitrary height).
\end{theorem}

Since ECH is really a theory of holomorphic currents, we find it also useful to define a \textit{cascade of holomorphic currents}, which is what we shall primarily work with.

\begin{definition}
A height 1 holomorphic cascade of currents $\cas{\mathbf{u}}= \{u^1,..,u^n\}$ consists of the following data:
\begin{itemize}
    \item Each $u^i$ consists of holomorphic currents of the form $(C^i_j,d_j^i)$. Each $C^i_j$ is a somewhere injective holomorphic curve with $E(C^i_j)<\infty$. The positive integer $d^i_j$ is then the multiplicity.
    \item Numbers $T_i \in [0,\infty), i=1,..,n-1$
    \item Let $\gamma_i$ be a simply covered Reeb orbit that is approached by the negative end of some component of $u^i$, say the components $C^i_{j_1},...,C^i_{j_k}$ (such curves have associated multiplicity $d^i_{j_1},...,d^i_{j_k}$). Each $C^i_{j_*}$ approaches $\gamma_i$ with a covering multiplicity $n_{j_*}$, which is how many times $\gamma_i$ is covered by $C^i_{j_*}$ as currents. Then the total multiplicity of $\gamma_i$ as covered by $u^i$ is given by $\sum_{*=1,..k} d^i_{j_*}n_{j_*}$. Then consider $\phi_{f}^{T_i}(\gamma_{i+1}):= \gamma_i$. Then $u^{i+1}$ is asymptotic to $\gamma^{i+1}$ in its positive end with total multiplicity $\sum_{*=1,..k} d^i_{j_*}n_{j_*}$ also.
    \item No level consists of purely of trivial cylinders (even if they have higher multiplicities).
\end{itemize}

\end{definition}
We define the positive asymptotics of $
\cas{\mathbf{u}}:=\{u^1,..,u^n\}$ as before, except we only care about Reeb orbits up to multiplicity. Then we can similarly define a cascade of currents of height $k$ by stacking together cascades of currents of height $1$.

We will refer to ordinary cascade a ``cascade of curves'' when we wish to distinguish it from a cascade of currents. Then given a cascade of curves, we can pass it to a cascade of currents by using the following procedure:
\begin{procedure}\label{curve_to_current}

\begin{itemize} 
    \item Replace every multiple covered non-trivial curve with a current of the form $(C,m)$ where $C$ is  a  somewhere  injective  curve,  and  we  translate  all $m$ copies  along $\bb{R}$ to  make  the  entire collection somewhere injective.
    \item If we see a multiply covered trivial cylinder we replace it with $(C,m)$  where $m$ is the multiplicity and $C$ is a trivial cylinder.
    \item If we see a nodal curve in one of the levels, we separate the node and apply the above process to each of the separated components of the nodal curve.
    \item We remove all levels that only have currents made out of trivial cylinders. Suppose $u^i$ is a level only consisting of trivial cylinders to be removed, and suppose the $s\rightarrow +\infty$ end is a intermediate cascade level with flow time $T_{i-1}$, and the $s\rightarrow -\infty$ end of $u^i$ has associated flow time $T_i$, after the removal of $u^i$ level, the newly adjacent levels $u^{i-1}$ and $u^{i+1}$ have flow time between them equal to $T_i+T_{i-1}$.
\end{itemize}
\end{procedure}
In passing from cascades of curves to currents we have lost some information, but we shall see currents are the natural settings to talk about ECH index.

We later wish to make sense of the Fredholm index of a cascade of currents. To this end we make the definition of \emph{reduced cascade of currents}.

\begin{definition}\label{def:reduced cascade}
Given a cascade of currents $\cas{\mathbf{u}}$, for components within it of the form $(C,m)$ where $m>1$ and $C$ is a nontrivial holomorphic curve, we then replace $(C,m)$ with just $(C,1)$. After we perform this operation we obtain another cascade of currents, which we label $\cas{\tilde{\mathbf{u}}}$, which we call the reduced cascade of currents.
\end{definition}

\section{Index calculations and transversality}
The heart of the calculation that underlies ECH is this: the ECH index bounds from the above the Fredholm index, and if there are curves of ECH index one with bad behaviour (singularities, multiply covers), this would imply the existence of somewhere injective curves of Fredholm index less than 1, which cannot happen for generic $J$. In this section we take up the issue of establishing Fredholm index for $J$ holomorphic cascades, and explain the transversality issue we encounter.

Given a reduced cascades of currents, $\cas{\mathbf{\tilde{u}}}= \{\tilde{u}^1,...,\tilde{u}^n\}$, we would like to assign to it a Fredholm index. Ideally this Fredholm index measures geometrically the dimension of the moduli space this particular cascade lives in. We note that by passing to the reduced cascade the multiplicities associated to ends of adjacent levels, $\tilde{u}^i$ and $\tilde{u}^{i+1}$ do not necessarily match up, but by imposing there is a single flow time parameter $T_i$ between adjacent levels still means we can think of $\cas{\mathbf{u}}$ as living in a fiber product with virtual dimension. 

To this end we first recall some conventions when it comes to $J$-holomorphic curves with ends on Morse-Bott critical submanifolds (in this case, tori). Consider $\tilde{u}^i$, for simplicity suppose its domain $\dot{\Sigma}_i$ is a punctured Riemann surface that is connected. Let $p_j^\pm$ label the positive/negative punctures, and the map $\tilde{u}^i$ is asymptotic to Reeb orbits (of some multiplicity) on Morse-Bott tori at each of its punctures. We wish to associate to $\tilde{u}^i$  a moduli space of curves that contain $\tilde{u}^i$ as an element and contains curves that are ``close'' to $\tilde{u}^i$. To this end we recall some conventions.

To each puncture $p_j^\pm$ of $\tilde{u}^i$, we can designated it as ``fixed'' or ``free'', and each choice of these designations leads to a different moduli space. The designation  ``free'' means we consider $J$-holomorphic maps from $\dot{\Sigma}_i$ so that $p_j^\pm$ can land on any Reeb orbit with the same multiplicity on the same Morse-Bott torus at the corresponding end of $\tilde{u}^i$. For a puncture to be considered ``fixed'', we consider moduli space of $J$-holomorphic curves from $\dot{\Sigma}_i$ so that $p_j^\pm$ lands on a fixed Reeb orbits on a Morse-Bott torus with fixed multiplicity (the same Reeb orbit as $\tilde{u}^i$). Given a designation of ``fixed'' or ``free'' on punctures of $\tilde{u}^i$, we can then consider the moduli space of $J$ holomorphic curves from $\dot{\Sigma}_i$ into $\bb{R}\times Y$ with the same asymptotic constraints as $\tilde{u}^i$ and living in the same relative homology class. We shall denote this moduli space as $\mcal{M}_{\mathbf{c}}(\tilde{u}^i)$, using $\mathbf{c}$ to denote our choice of fixed/free ends. This moduli space has virtual dimension given by:
\begin{equation}
    Ind(\tilde{u}^i) : = -\chi(\tilde{u}^i) +2c_1(\tilde{u}^i) + \sum _{p_j^+} \mu(\gamma^{q_{p_j^+}} )- \sum _{p_j^-} \mu(\gamma^{q_{p_j^-}}) + \frac{1}{2}\# \text{free ends} - \frac{1}{2}\# \text{fixed ends}
\end{equation}
where $\chi$ is the Euler characteristic, $c_1$ the relative first Chern class, $\mu(-)$ is the Robbin Salamon index for path of symplectic matrices with degeneracies defined in \cite{Gutt2014}. We use the symbol $\gamma$ to denote the Reeb orbit the end $p_j^\pm$ is asymptotic to, with multiplicity $q_{p_j^\pm}$.

Given a reduced cascade of currents, $\cas{\mathbf{\tilde{u}}}$, let $\alpha$ denote the designation of ``free''/``fixed'' ends of $\tilde{u}^1$ at the $s\rightarrow +\infty$ end, and let $\beta$ denote the ``fixed''/``free'' designation of $\tilde{u}^n$ at the $s\rightarrow -\infty$ end. Later we will see we can replace $\alpha$ and $\beta$ with Morse-Bott ECH generators. In order to define the Fredholm index we need to assign free/fixed ends to the rest of the ends.

\begin{convention}\label{convention:free/fixed}
If a non trivial curve $u^i$ has an end landing on a critical point of $f$, then we consider that end to be fixed. If a trivial cylinder has one end on critical point of $f$, the other end must also land on the same critical point.
We allow trivial cylinders with both ends free. If the trivial cylinder is at a critical point of $f$, we take the convention we can only designate one of its ends as fixed.
\end{convention}

\begin{definition}\label{index}
Let $\cas{\tilde{\mathbf{u}}}=\{u^1,..,u^{n-1}\}$ denote a reduced cascade of currents of height 1. Let $ind(u^i)$ denote the Fredholm index of each of $u^i$. Note this makes sense since we have assigned free/fixed ends to all ends of $u^i$ by our conventions above. 

Suppose there are $R_2,..., R_{n-1} \in \bb{Z}$ distinct Reeb orbits approached by free ends as $s\rightarrow-\infty$ at each intermediate cascade level. Let us denote $k_i$ and $k_i'$ the number of free ends in each intermediate cascade level. e.g. elements in $u^1$ has $k_2$ free ends as $s\rightarrow -\infty$, and $u^2$ has $k_2'$ free ends as $s\rightarrow +\infty$. Both counts of $k_i$ and $k_i'$, as well as $R_i$ ignores ``free'' ends of fixed trivial cylinders, as such ``free'' ends are artificial to our convention.  
Now we define the cascade dimension
\begin{align*}
   Ind(\cas{\tilde{\mathbf{u}}})
   :=& Ind(u^1)+..+Ind(u^{n-1}) \\
    &- [k_2'...+k_{n-1}']-[k_2+...+k_{n-1}]+ [R_ 2+..+R_{n-1}]+(n-2)- (n-1)-L
\end{align*}
where $L$ is the number of intermediate cascade levels without free ends plus the number of intermediate cascade levels whose flow time is zero. Again in the count of $L$ we ignore ``free'' ends coming from fixed trivial cylinders.
\end{definition}
Observe for (reduced) cascades of height $1$, we always have $k_i\geq R_i$ and $k_i'\geq R_i$.

We next explain how to define/compute the dimension of height $k$ cascades. 
Let $\cas{\tilde{\mathbf{u}}}=\{u^1,..,u^{n-1}\}$ denote a reduced cascade of currents of height $N$. We recall the difference between height one and height $N$ cascade is that between cascade levels $u^i$ and $u^{i+1}$ we allow flow times $T_i = \infty$. We assign the free/fixed ends to $u^i$ depending on whether they land on critical points of $f$ as before. We can split a height $N$ cascade into $N$ height 1 cascades by partitioning the levels where the flow times are infinite. In particular we write $\cas{\tilde{\mathbf{u}}} = \left\{ \cas{\tilde{\mathbf{v^1}}}, ...,\cas{\tilde{\mathbf{v^N}}} \right \}$. Then the index of $\cas{\tilde{\mathbf{u}}}$ is given by the sum of the indices of $\cas{\tilde{\mathbf{v^i}}}$. 

Here we come to the key transversality assumption of this paper. We first make sense of the notion of transversality.

\begin{definition}\label{def:good j}
Let $\lambda$ be a Morse-Bott contact form, whose Reeb orbits come in $S^1$ families.
We say a $\lambda$ compatible almost complex structure $J$ is \textbf{good} if all reduced cascades of height one are tranversely cut out, which is defined below.
\end{definition}
\begin{remark}
We note the transversality conditions needed to count cascades given below are quite natural. However, since cascades have many parts the notation is bit complicated.
\end{remark}
\begin{definition} \label{def:transversality conditions}

Let $\cas{\tilde{\mathbf{u}}}=\{u^1,..,u^{n-1}\}$ denote a reduced cascade of currents of height 1.

We say $\cas{\tilde{\mathbf{u}}}=\{u^1,..,u^{n-1}\}$ is \textbf{transversely cut out} if the conditions below are met. 
\begin{itemize}
    \item Each moduli space $\mcal{M}_{c}(u^i)$ is transversely cut out with dimension given by the Fredholm index formula. Here the subscript $c$ implicitly denotes the assignments of fixed and free ends we assigned to each end of $u^i$ according to Convention \ref{convention:free/fixed}. Note fixed trivial cylinders are assigned index zero.
\end{itemize}
Suppose there are $R_2,..., R_{n-1} \in \bb{Z}$ distinct Reeb orbits reached by free ends at each intermediate cascade level. We label them by $\gamma(i,j)$ where $j=1,...,R_i$, and $i$ indexes which level we are referring to. For each $\gamma(i,j)$, we choose a negative puncture of $u^{i-1}$ that is asymptotic to $\gamma(i,j)$. We call this puncture $p^-(i-1,j)$. The other negative ends of $u^{i-1}$ that are asymptotic to $\gamma(i,j)$ are labelled $p^-(i-1,j,c, l)$, where $l=1,2.., n(\gamma(i,j),-)$. Next consider $\phi^{-T_{i-1}}(\gamma(i,j)))$. They are approached by positive punctures of $u^i$. For each $\phi^{-T_{i-1}}(\gamma(i,j)))$, we pick out a special free puncture $p^+(i,j)$. The remaining free positive ends of $u^i$ that are asymptotic to $\phi^{-T_{i-1}}(\gamma(i,j)))$ are labelled $p^+(i,j,c,l)$ for $l=1,...,n(\gamma(i,j),+)$. 

We next consider the evaluation maps. Given the collection of flow times $T_1,...,T_{n-1}$. Let $\mathfrak{I} \subset \{1,..,n-1\}$ denote the subset for which $T_i>0$, we consider the evaluation map
\begin{align}
        EV^-: \mcal{M}(u^1) \times  \mcal{M}(u^2)\times  ...\times \mcal{M}(u^{n-2}) \rightarrow (S^1)^{R_2} \times (S^1)^{R_3} \times ... \times (S^1)^{R_{n-1}}
    \end{align}
 given by
    \begin{align}
        (u'^{1},...,u'^{n-2}) \rightarrow (ev_1^-(u'^1), ev_2^-(u'^2),...,ev_{n-2}^-(u'^{n-2}))
    \end{align}

Here the evaluation is at the $p^-(i-1,j)$ puncture of $u^{i-1}$. We also consider the 
map
    \begin{align}
        EV^+: \mcal{M}(u^2) \times \mcal{M}(u^3) ...\times \mcal{M}(u^{n-1})  \rightarrow (S^1)^{R_2} \times (S^1)^{R_3} \times ... \times (S^1)^{R_{n-1}} 
    \end{align}
    given by:
    \begin{align}
        (u'^{2},...,u'^{n-1}) \rightarrow (ev_2^+(u'^2),...,ev_{n-1}^+(u'^{n-1}))
    \end{align}
    where the evaluation is at $p^+(i,j)$ of $u^i$. 
    We consider the flow map
    \[
    \Phi_f: (S^1)^{R_2} \times \bb{R}^* \times..\times (S^1)^{R_{n-1}} \times \bb{R}^* \rightarrow (S^1)^{R_2} \times (S^1)^{R_3} \times ... \times (S^1)^{R_{n-1}}.
    \]
    The notation $\bb{R}^*$ means the following: if $i \in \mathfrak{I}$ then we include a factor of $\bb{R}$ in the above product, otherwise we omit the factor. For $x_i\in S^1$ (i.e. a copy of $S^1$ among the product $(S^1)^{R_i}$), if $i \in \mathfrak{I}$ then the image of $x_i$ under $\Phi_f$ is given by $\phi_f^{T_i'}(x_i)$. If the index $i$ is not in $\mathfrak{I}$, then the image under $\Phi_f$ is $x_i$.
    We use the notation $\Phi_f\circ EV^+$ to denote the composition of the two maps, with domain  $\mcal{M}(u^2) \times \bb{R}^* \times \mcal{M}(u^2) ...\times \mcal{M}(u^{n-1})\times \bb{R}^*$ and image $(S^1)^{R_2} \times (S^1)^{R_3} \times ... \times (S^1)^{R_{n-1}}$.
    
    Let $\mcal{K}_-$ denote the subset of $\mcal{M}(u^1) \times  \mcal{M}(u^2)\times  ...\times \mcal{M}(u^{n-2})$ so that the ends $p^-(i,j)$ and $p^-(i,j,c,l)$ approach the same Reeb orbit. Let $\mcal{K}_+$ denote the subset of $\mcal{M}(u^2) \times \mcal{M}(u^3) ...\times \mcal{M}(u^{n-1})$ where $p^+(i,j)$ and $p^+(i,j,c,l)$ are asymptotic to the same Reeb orbit. Then
    \begin{itemize}
        \item Near $\cas{\mathbf{\tilde{u}}}$, both $\mcal{K}_\pm$ are transversly cut out submanifolds.
    \end{itemize}
    Then we can restrict $EV^\pm$ to $\mcal{K}_\pm$, in particular the map $\Phi_f \circ EV^+$ admits a natural restriction to $\mcal{K}_-\times \bb{R}^{|\mathfrak{I}|}$, our final condition is:
    \begin{itemize}
        \item $\Phi_f \circ EV^+$ and $EV^-$, when restricted to $\mcal{K}_+\times (\bb{R})^{|\mathfrak{I}|}$ and $\mcal{K}_-$ respectively are transverse at $\cas{\tilde{\mathbf{u}}}=\{u^1,..,u^{n-1}\}$
    \end{itemize}
\end{definition}
\begin{assumption}\label{assumption}
We assume we can choose $J$ to be good so that all reduced cascades of current we encounter satisfy the transversality condition above.
\end{assumption}
In particular, this implies all reduced cascades of currents live in a moduli space whose dimension is given by the index formula, and if such index is less than zero, then such cascades cannot exist.

We note that in general the transversality assumption is not automatic. In a reduced cascades of currents, all our curves are somewhere injective, but this is not enough. The issue lies in the fact that the fiber product that defines cascade can fail to have enough transversality. This is because all different levels of the cascade have the same $J$, and this $J$ cannot be perturbed independently in each level. When the cascade is complicated enough, the same curve can appear multiple times in different levels, and this causes difficulty with the evaluation map. 
Consequently when there is not enough transversality for the naive definition of the universal moduli space of reduced cascades to be a Banach manifold, one usually needs some additional arguments.

However in simple enough cases we can still achieve the above transversality condition. This is the content of Theorem \ref{thm:list of transversality conditions}, which is proved in the Appendix.

\section{ECH Index of Cascades}\label{Section:ECH index}
In this section we develop the analogue of ECH index one condition for cascades of currents. We shall see this will impose severe limits on currents that can appear in a cascade, provided transversality can be achieved.

To start the definition, we first consider one-level cascades, i.e. holomorphic curves from Morse-Bott tori to Morse-Bott tori. We want to define an index $I$ so that for somewhere injective curves:
\[
I(C) \geq dim\mcal{M}(C)+2\delta(C)
\]
where $\mcal{M}(C)$ denotes the moduli space of holomorphic curves $C$ belongs in. Note the definition of $dim\mcal{M}$ is ambiguous, because we need to specify which ends are ``fixed" and which are ``free". Our definition of $I$ will depend on the type of end conditions imposed on our curve. The key to our construction will be the relative adjunction formula.

\subsection{Relative adjunction formula in the Morse-Bott setting}
Here we clarify what we mean by the intersection form $Q$. We first provide a provisional definition that is very much similar to regular ECH, then we show this definition descends to a more natural definition adapted to the Morse-Bott setting.

Let $\alpha,\beta$ be orbit sets. Observe here this means that they pick out discrete Reeb orbits (potentially with multiplicity) among the $S^1$ family of Reeb orbits. Then we can define the relative intersection formula as: 
\begin{definition}\label{intersection_form_prelim_definition}
We fix trivializations of Morse-Bott tori as we have specified, and denote it by $\tau$. Given $\alpha,\beta$ orbit sets, given $Z,Z'\in H_2(\alpha,\beta,Y)$ we choose $\tau$ representatives $S$ $S'$ as before, then $Q_\tau(Z,Z')$ is defined as before as the algebraic count of intersections between $S$ and $S'$.
\end{definition}
Because $\tau$ here provides a global trivialization of all Reeb orbits in a given Morse-Bott torus, the intersection $Q$ doesn't depend on \textit{which} specific Reeb orbit $\alpha$ or $\beta$ picks out in a given Morse-Bott torus. We state the phenomenon in terms of a proposition:
\begin{proposition}
Given orbit sets $\alpha,\beta$ and relative homology classes $Z,Z'\in H_2(\alpha,\beta)$. For definiteness let $\gamma$ be a Reeb orbit in the $s\rightarrow +\infty$ end of $\alpha$, let $\gamma'$ be any translation of $\gamma$ in its Morse-Bott torus, then using $\gamma'$ to replace $\gamma$ defines another orbit set $\alpha'$. There exists corresponding relative homology classes $\hat{Z},\hat{Z}' \in H_2(\alpha',\beta,Y)$ obtained by attaching a cylinder that connects between $\gamma$ to $\gamma'$ to ends of $S$ and $S'$ that are asymptotic to $\gamma$, then
\[
Q_\tau(Z,Z')=Q_\tau(\hat{Z},\hat{Z}')
\]
\end{proposition}
\begin{proof}
Choose $\tau$ representatives for $Z,Z'$ which we write as $S$, $S'$, then attach a cylinder connecting between $\gamma$ to $\gamma'$ to $S$ and $S'$. In our trivialization the resulting surfaces are still $\tau$ representatives, and this process does not introduce additional intersections.
\end{proof}

The above proposition suggests $Q_\tau$ in the Morse-Bott case descends to a intersection number whose input is not $H_2(\alpha,\beta,Y)$ but a more general relative homology group adapted to the Morse-Bott setting.

\begin{definition}
We define the relative homology classes $\mathcal{H}_2(\alpha,\beta,Y)$. Here $\alpha,\beta$ are collections of Morse-Bott tori, and multiplicities. For instance we can write $\alpha :=\{(\mcal{T}_i,m_i)| m_i \in \bb{Z}_{\geq 0}\}$ where $\mcal{T}_i$ are Morse-Bott tori, and $m_i$ are multiplicities. A element $Z \in \mathcal{H}_2(\alpha,\beta,Y)$ is a 2-chain in $Y$ so that 
\[
\partial Z = \alpha -\beta.
\]
The above equality means the boundary (which includes orientation) of $Z$ consists of Reeb orbits on Morse-Bott tori $\{\mcal{T}_i\}$, and each $\mathcal{T}_i \in \alpha$ has a total of $m_i$ Reeb orbits (counted with multiplicity) to which the ends of $Z$ are asymptotic. Likewise for $\beta$.
We define a equivalence relation on $\mathcal{H}_2(\alpha,\beta,Y)$, which we write as $Z\sim Z'$ as follows: $Z$ and $Z'$ are equivalent if there is a 3-chain $W$ whose boundary takes the following form:
\[
\partial W = Z-Z' + \{I\times S^1\}
\]
where the collection $\{I\times S^1\}$ consists of 2 chains on Morse-Bott tori that appear in either $\alpha$ or $\beta$. We think of these 2-chains as an Reeb orbit (which we think of $S^1$) times an interval, $I$.
\end{definition}

The idea is we consider 2-chains but allow their ends to slide along the Reeb orbits in the Morse-Bott family. The next proposition proves the relative intersection $Q$ remains well defined.

\begin{proposition}
$Q_\tau$ as defined above descends into a intersection form:
\[
Q_\tau:\mathcal{H}_2(\alpha,\beta,Y) \times \mathcal{H}_2(\alpha,\beta,Y)\rightarrow \bb{Z}.
\]
\end{proposition}
\begin{proof}
For clarity we use $\hat{Q}_\tau$ to denote the intersection form defined in Definition \ref{intersection_form_prelim_definition}. Suppose $Z,Z' \in \mathcal{H}_2(\alpha,\beta,Y)$, and suppose $Z'' \sim Z$. We  pick a distinguished Reeb orbit $\gamma_i$ for each Morse-Bott torus that appears in $\alpha,\beta$, and chosen so that $\gamma_i$ does not appear as a Reeb orbit in $Z,Z'$ and $Z''$. We connect Reeb orbits in $Z$, $Z'$ and $Z''$ to $\{\gamma_i\}$ counted with multiplicities using cyliners along each Morse-Bott tori to obtain $\hat{Z},\hat{Z}',\hat{Z}''$. We then define 
\[
Q_\tau(Z,Z') : = \hat{Q}_\tau(\hat{Z},\hat{Z''}).
\]
Observe in the above $\hat{Q}_\tau$ is an intersection form defined on $H_2(\alpha',\beta',Y)$ where $\alpha'$ and $\beta'$ are collections of Reeb orbits of the form $\{(\gamma_i,n_i)\}$. It suffices to prove $Q_\tau(Z'',Z') = Q_\tau(Z,Z'')$. To do this note the fact $Z \sim Z''$ in $\mathcal{H}_2(\alpha,\beta,Y)$ extends to an equivalence of $\hat{Z} \sim \hat{Z''}$ in $H_2(\alpha',\beta',Y)$, hence $\hat{Q}_\tau(\hat{Z}'',\hat{Z}') = \hat{Q}_\tau(\hat{Z},\hat{Z}')$, and hence the proof.
\end{proof}

We observe using the above reasoning the relative Chern class also descends to $\mathcal{H}_2(\alpha,\beta,Y)$. We state this in the form of a definition:
\begin{definition}
Given $Z\in \mcal{H}_2(\alpha,\beta,Y)$, we define the relative Chern class $c_\tau(Z)$ the same way as before: choose a representative $S$ of $Z$ that is embedded near the boundary. Let $\iota:Z\rightarrow Y$ be the inclusion, then consider the pullback of the contact structure $\iota^*\xi$ to $Z$, pick a section $\psi$ of $\iota^*\xi$ that does not rotate with respect to $\tau$ near the end points and has transverse zeroes, then $c_\tau(Z)$ is the signed count of zeroes of $\psi$.
\end{definition}

Finally we define writhe the same way as before:
\begin{definition}
Let $C$ be a somewhere injective curve that is not a trivial cylinder. We assume at $s\rightarrow +\infty$ (resp. $-\infty$) it is asymptotic to orbit set $\alpha$ (resp. $\beta$). The trivialization specified in Theorem \ref{locform} gives an identification a neighborhood of each Reeb $\gamma \in \alpha,\beta$ with $S^1\times \bb{R}^2$, then using this we can define writhe of $C$ as we had before in section \ref{ECH review}.
\end{definition}
\begin{remark}
The definition of writhe depends crucially on the fact $C$ is a holomorphic curve, and does not admit constructions as before where we can slide the Reeb orbits of $\alpha,\beta$ around and obtains a surface with same relative intersection number/Chern class.
\end{remark}
Hence we are ready to state the relative adjunction formula.
\begin{theorem}
If $C$ is a simple $J$-holomorphic curve, then 
\[
c_\tau(C) = \chi(C)+Q_\tau(C) +w_\tau(C)-2\delta(C)
\]
with the definition of relative chern class, relative intersection number, and writhe given above.
\end{theorem}
\begin{proof}
This is a purely topological formula. The same proof as in \cite{Hutchings2002} follows through.
\end{proof}
Hence we would like to define a version of ECH index by applying the relative adjunction formula to the Fredholm index formula of holomorphic curves as in \cite{Hutchings2002}. Recall then the proof of index inequality boils down to bounding the writhe of the $J$ holomorphic curve in terms of various algebraic expressions involving the Conley Zehnder indices that the curve is asymptotic to. We turn to this writhe bound in the next subsection.

\subsection{Writhe Bound}
We recall the Fredholm index of a somewhere injective curve $u$ depends on which end is free and which end is fixed. Hence we anticipate that the ECH index we assign to a holomorphic curve $u$ will depend on which end is fixed and which end is free. The writhe inequality we prove shall take into account of the assignment of free and fixed ends. We note that this assignment of an index to a curve that depends on which end is free/fixed is somewhat artificial, but it will be less artificial once we use this index to define the ECH index of an entire cascade.

First we fix some conventions on Conley Zehnder indices. For a given Morse-Bott Torus $\mcal{T}$ assume the $J$ holomorphic curve has $N$ ends that are positively (resp. negatively) asymptotic to Reeb orbits on this torus. They are asymptotic to the individual Reeb orbits labelled $R_1, .., R_n$. Writhe bound is a local computation so we only consider a particular Reeb orbit, called $R_1$. Assume $k$ ends of $C$ are asymptotic to $R_1$. They have multiplicity $q_1,..,q_k$.
We adopt the following convention on Conley Zehnder indices. 
\begin{convention}

Recall for positive Morse-Bott torus $\mu =1/2$. We declare $\mu_+ =1$, $\mu_-=0$. 
For negative Morse-Bott torus we declare $\mu_+=0$, $\mu_- = -1$. 

This has the following significance: for a curve with free end as $s\rightarrow +\infty$ landing in a Morse-Bott torus (regardless of whether it is positive or negative torus), the Conley Zehnder index term in the Fredholm index formula associated to this end is $\mu_+$ (the specific value depends on the positive/negative Morse-Bott torus as above), and the Conley Zehnder index term assigned to fixed end is $\mu_-$. Conversely, at the $s\rightarrow -\infty$ end we assign $\mu_-$ to free ends and $\mu_+$ to fixed ends.
\end{convention}

Using the above conventions given a somewhere injective holomorhic curve $u$, we assign its total Conley-Zehnder index denoted by $CZ^{Ind}(u)$ according to the convention above. The goal of the writhe inequality is to come up with another Conley-Zehnder index term $CZ^{ECH}(u)$ so that the total writhe of $u$ is bounded above by
\begin{equation}
    wr_\tau(u) \leq CZ^{ECH}(u) - CZ^{Ind}(u)
\end{equation}

By way of convention we will use $CZ^*(R_1,\pm \infty)$ where $*=Ind,ECH$ to denote the Conley Zehnder index we should assign to the free/fixed ends approaching $R$ as $s\rightarrow \pm \infty$

\subsubsection{Positive Morse-Bott tori}
\begin{theorem}\label{winding pos}
In the case of positive Morse-Bott torus, $s\rightarrow -\infty$, if $\xi_i$ is an end of $u$ with covering multiplicity $q_i$ and $u$ is not the trivial cylinder, we have the following inequality
\begin{equation*}
    \eta (\xi_i) \ge 1 \quad \text{(single end winding number)}.
\end{equation*}
For single end writhe, we have:
\begin{equation*}
    w(\xi_i) \ge \eta(\xi_i) (q_i-1). \quad 
\end{equation*}
Note this holds true for trivial cylinders (as long as it's somewhere injective).

Let $\xi_1$ and $\xi_2$ be two braids that correspond to two distinct ends of $u$ that approach the same Reeb orbit, with multiplicities $q_i$ and winding numbers $\eta_i$, then:
\begin{equation*}
    l(\xi_1,\xi_2)\ge min (q_1\eta_2,q_2\eta_1)
\end{equation*}
Note this holds if one of the ends $\xi_i$ came from a trivial cylinder.

And finally to calculate the writhe of all ends approach the same Reeb orbit, $w(\xi)$, let $\xi$ denote the total braid and $\xi_i$ the various components coming from incoming ends of $u$ (this holds for both $s=\pm \infty$):
\begin{equation*}
    w(\xi) = \sum_i w(\xi_i) + \sum_{i \neq j } l(\xi_i,\xi_j)
\end{equation*}

In the case of $s\rightarrow +\infty$, using the exactly the same notation, we have the following inequalities:
\begin{equation*}
    \eta (\xi_i) \le 0
\end{equation*}
\begin{equation*}
    w(\xi_i) \le \eta(\xi_i) (q_i-1) \,\,\,\, \text{for single end writhe}
\end{equation*}
\begin{equation*}
    l(\xi_1,\xi_2)\le max (q_1\eta_2,q_2\eta_1)
\end{equation*}
\end{theorem}
\begin{proof}
(Sketch) The proof constitutes an amalgamation of existing results in the literature. The key result is an description of asymptotics of ends of holomoprhic curves on Morse-Bott torus \cite{siefring}. Namely, near the $s\rightarrow +\infty$ end of $u$,  the $s$ constant slice of $\{s\} \times Y$ of $u$ can be described as follows. We can choose a neighborhood of trivial cylinder $\bb{R}\times \gamma$  as $\bb{R}\times S^1\times \bb{R}^2$ where $s$ is  the symplectization direction, $t$ is the variable along the Reeb orbit and $\{0\}\times \bb{R}^2$ is the contact structure along the Reeb orbit, then we can write an end $\xi_i$ of $u$ as
\begin{equation}\label{locform}
    u(s,t) = (qs, qt, \sum_{i=1}^ne^{\lambda_i s}e_i(t))
\end{equation}
where $\lambda_i$ and $e_i$ are respectively the (negative) eigenvalues and corresponding eigenfunctions of the operator $A(t): L^2(S^1,\bb{R}^2) \rightarrow L^2(S^1,\bb{R}^2)$ coming from the linearization of the Cauchy Riemann operator, which can be written as
\[
A(t) = -J\p_t -S
\]
With this normal form,
the winding number bound comes from combining the results in \cite{Gutt2014} about the meaning of Robbin-Salamon index and results in \cite{Hofer2} relating Conley-Zehnder indices to crossing of eigenvalues. The relations on writhe and linking number come from direct modifications from the proofs in \cite{Hutchings2002}, once we realize that locally the braids can be described by Equation \ref{locform}.
\end{proof}

Next we move to use these relations to prove writhe bound. As in the case of ECH, equality of the writhe bound implies certain partition conditions, which we will carefully state.

\begin{proposition}[link, $-\infty$, positive Morse Bott torus]
Consider the $J$ holomorphic curve $u$ with negative ends on a Reeb orbit $\gamma$. We have $k_{free}$ free ends of multiplicity $q_i^{free}$, and $k_{fixed}$ fixed ends with multiplicity $q_i^{fixed}$ and of total multiplicity $N_{fixed}$. The writhe bound reads
\[
w(\xi) \geq  -\sum_{i=1}^{k_{free}+k_{fixed}} \eta_i +\sum_{i,j}^{k_{free}+k_{fixed}}  min(\eta_i q_j, \eta_j q_i) \geq (N_{free} -1 + N_{fixed}) -(k_{fixed})
\]
with equality holding implying there can be only free/fixed ends at this Reeb orbit. If there are only fixed ends the partition conditions is $(n)$, and if there are only free ends the partition condition is $(n)$ or $(1,n-1)$.

\end{proposition}
\begin{proof}
We have the respective bounds 
\[
-k_{free} + \sum_{i}^{k_{free}}  min(\eta_i q_j, \eta_j q_i)\geq N_{free}-1 
\]
and 
\[
-k_{fix}+\sum_{i,j}^{k_{fix}}  min(\eta_i q_j, \eta_j q_i)\geq N_{fix}-k_{fixed}
\]
and cross terms will imply strict inequality, hence only free or fixed term appears. In the case of only fixed points, we see the only way equality can hold is with partition condition $(n)$. Similar considerations produces the partition conditions for free ends.
\end{proof}

\begin{proposition} [link, $\infty$, positive Morse Bott Torus]
In the $s\rightarrow +\infty$ end, consider the $J$ holomorphic curve $u$ with ends on a Reeb orbit $\gamma$. We have $k_{free}$ free ends of multiplicity $q_i^{free}$, and $k_{fixed}$ fixed ends $q_i^{fixed}$ of total multiplicity $N_{fixed}$:
\[
w(\xi) \leq -\sum_{i=1}^{k_{free}+ k_{fix}} \eta_i + \sum_{i,j}^{k_{free}+ k_{fix}} max(q_j \eta_i, q_i,\eta_j) \le N_{free} -(k_{free}).
\]
The partition condition implies $(1,...,1)$ on the free ends.
\end{proposition}
\begin{proof}
We see that $lhs \leq 0$, and $RHS= 0$ iff the free end satisfies partition conditions $(1,...1)$; there are no requirements on fixed ends.
\end{proof}

\subsubsection{Negative Morse-Bott tori}
In this subsection we take up the analogous writhe bounds for negative Morse-Bott tori.

\begin{theorem}\label{winding neg}
In the case of negative Morse Bott torus, $s\rightarrow -\infty$, we have the following inequalities:

If $\xi_i$ is an end of $u$ and $u$ is not the trivial cylinder, we have the following inequality:
\begin{equation*}
    \eta (\xi_i) \ge 0
\end{equation*}

For writhe of a single end, with covering multiplicity $q_i$, we have:
\begin{equation*}
    w(\xi_i) \ge \eta(\xi_i) (q_i-1) 
\end{equation*}
Note this holds for the case of a trivial cylinder.

Let $\xi_1$ and $\xi_2$ be two braids that correspond to two distinct ends of $u$ that approach the same Reeb orbit, with multiplicities $q_i$ and winding numbers $\eta_i$, then:
\begin{equation*}
    l(\xi_1,\xi_2)\ge min (q_1\eta_2,q_2\eta_1)
\end{equation*}
Note this holds if one of the ends $\xi_i$ came from a trivial cylinder.

And finally to calculate the writhe of all ends approach the same Reeb orbit, $w(\xi)$, let $\xi$ denote the total braid, and $\xi_i$ the various components coming from incoming ends of $u$ (this holds for both $s=\pm \infty$):
\begin{equation*}
    w(\xi) = w(\xi_i) + \sum_{i \neq j } l(\xi_i,\xi_j)
\end{equation*}

In the case of $s\rightarrow +\infty$, we have the following inequalities
\begin{equation*}
    \eta (\xi_i) \le -1
\end{equation*}
\begin{equation*}
    w(\xi_i) \le \eta(\xi_i) (q_i-1) \,\,\,\, \text{for single end writhe}
\end{equation*}
\begin{equation*}
   l(\xi_1,\xi_2)\le max (q_1\eta_2,q_2\eta_1)
\end{equation*}
\end{theorem}
\begin{proof}
The exact same proof for the positive Morse-Bott torus except we use Robbin-Salamon index $\mu=-1/2$.
\end{proof}

\begin{proposition}[link, $-\infty$,negative Morse Bott torus]
Let $u$ have ends asymptotic to $\gamma$ on a negative Morse-Bott torus as $s\rightarrow -\infty$, suppose there are $k_{free}$ free ends of multiplicity $q_i^{free}$, of total multiplicity $N_{free}$; suppose there are $k_{fix}$ fixed ends each of multiplicity $q_{fix}$, of total multiplicity $N_{fix}$.
Then we have the writhe bound:
\[
w(\xi) \geq -\sum_{i}^{k_{fix}+k_{free}} \eta_i +\sum _{i,j}^ {k_{fix}+k_{free}} min(\eta_i q_j, \eta_j q_i) \geq -N_{free}-(-k_{free})
\]
with equality enforcing partition condition $(1,..,1)$ on free ends and no partition condition on fixed ends.
\end{proposition}
\begin{proof}
$\eta \geq0$ so $lhs\geq 0$, $rhs =k_{free} - N_{free}$ so inequality holds, and equality if free ends has partition conditions $(1,..,1)$, no restrictions on fixed ends.
\end{proof}

\begin{proposition}[link, $+\infty$,negative Morse-Bott torus]
Let $u$ have ends asymptotic to $\gamma$ on a negative Morse-Bott torus as $s\rightarrow +\infty$, suppose there are $k_{free}$ free ends of multiplicity $q_i^{free}$, of total multiplicity $N_{free}$; and suppose there are $k_{fix}$ fixed ends each of multiplicity $q_{fix}$, of total multiplicity $N_{fix}$.
\[
w(\xi) \leq -\sum_{i}^{k_{fix}+k_{free}} \eta_i +\sum _{i,j}^ {k_{fix}+k_{free}} max(\eta_i q_j, \eta_j q_i)  \leq -N_{fix}-N_{free}+1 +k_{fix}
\]
with equality enforcing only free or fixed ends. In the case of only fixed ends the partition condition is $(n)$, and in the case of only free ends the partition condition is either $(n)$ or $(n-1,1)$.
\end{proposition}
\begin{proof}
We can split the sum into:
\[
-\sum_i^{k_{free}} \eta_i + \sum _{i,j}^ {k_{free}} min(\eta_i q_j, \eta_j q_i)\leq 1-N_{free}
\]
and 
\[
-\sum_i^{k_{fixed}} \eta_i + \sum _{i,j}^ {k_{fixed}} min(\eta_i q_j, \eta_j q_i)\leq k_{fix}-N_{fix}.
\]
Each of the above inequalities hold individually, and when there are both free and fixed ends, there are cross terms that make the inequality strict. As before, we can deduce the partition conditions directly from imposing the equality condition.
\end{proof}

\subsection{Morse-Bott tori as ECH generators} \label{MBT as ECH}
Recall that for ECH of nondegenerate contact forms, the generators of the chain complex are orbit sets satisfying the condition that if an orbit is hyperbolic then it can only have multiplicity $1$. There are analogues of this in Morse Bott tori. In Morse-Bott ECH, we think of the generators of the chain complex as collections of Morse-Bott tori with additional data, written schematically as:
\[
\alpha = \{ (\mcal{T}_j,\pm, m_j)\}
\]
and the differential as counting ECH index one height one $J$ holomorphic cascades connecting between chain complex generators as above (which we will also call orbit sets). In the above definition $m_j$ is the total multiplicity, which we think of as total multiplicity of Reeb orbits on $\mcal{T}_j$ hit by the $J$ holomorphic curves that have ends on this Morse-Bott torus on the 
top (resp. bottom) level of a (height 1) cascade. $\pm$ is additional information, which specifies how many ends of the $J$ holomorphic curve landing on $\mcal{T}_j$ are free/fixed. We see that this also depends on whether $\alpha$ appears as the top or bottom level of a $J$ holomorphic cascade, and in context of our correspondence theorem free/fixed ends correspond to elliptic/hyperbolic orbits in the non-degenerate case. We state this explicitly in the next definition in which we also describe the expected correspondence between Morse-Bott ECH generators and nondegenerate ECH generators after the perturbation.

\begin{definition} \label{mbgenerator}
We consider the case of positive Morse Bott tori. In the nondegenerate case we let $\gamma_-$ denote the hyperbolic Reeb orbit that arises from perturbation with Conley Zehnder index 0, and $\gamma_+$ the elliptic orbit that arose out of the perturbation with Conley Zehnder index 1. Then the description of our Morse-Bott generator, say $(\mcal{T},\pm,m)$ (this is just one Morse-Bott torus, in general $\alpha$ will consist of a collection of such tori, we focus on an example for the sake of brevity) and its correspondence with ECH generators in the perturbed non-degenerate case is given by:

\begin{enumerate}
            \item positive side $s\rightarrow \infty$,
            \begin{enumerate}
                \item The Morse-Bott generator $(\mcal{T},+,m)$ is defined to require all ends on $\mcal{T}$ are free, with total multiplicity on the torus being $m$. In the perturbed nondegenerate case, this corresponds to ECH orbit set  $(\gamma_+, m)$. We observe the nondeg partition ($\theta$ positive) condition is $(1,..,1)$, and the Morse-Bott partition condition from the writhe bound is $(1,..1)$. 
                
                By the Conley-Zehnder index convention the ECH conley Zehnder index assigned to $(\mcal{T},+,m)$ is given by: $CZ^{ECH}((\mcal{T},+\infty,+,m)) =m$
                \item The Morse-Bott generator $(\mcal{T},-,m)$ there is one end on $\mcal{T}$ that is fixed with multiplicity 1, on the critical point of $f$ that corresponds to the hyperbolic orbit. The rest of the ends are free, and the total multiplicity of orbits on $\mcal{T}$ is $m$. This corresponds to the orbit set $\{(\gamma_-,1),(\gamma_+,m-1)\}$ in the nondegenerate case. Note the partition conditions between nondegenerate case and Morse-Bott case agree.
                
                We also have $CZ^{ECH}((\mcal{T},+\infty,-,m)) =m-1$.
            \end{enumerate}
        
            \item In the case of negative ends,  $s\rightarrow -\infty$,
            \begin{enumerate}
                \item The Morse-Bott generator $(\mcal{T},+,m)$ is defined to require all ends are fixed and asymptotic to the critical point of $f$ corresponding to the elliptic orbits, and the total multiplicity is $m$. In the nondegenerate case this correspond to the orbit set $(\gamma_+,m)$. We observe Morse-Bott and nondegenerate partition conditions agree, both being $(m)$.
                By our conventions, $CZ^{ECH}(\mcal{T},+,m)=m$
                \item The Morse-Bott generator $(\mcal{T},-,m)$ requires there is a multiplicity 1 free end landing on $\mcal{T}$, the remaining ends are fixed and are also required to land on the critical point corresponding to elliptic Reeb orbit. This corresponds to the orbit set $\{(\gamma_+,m-1),(\gamma_-,1)\}$ in the nondegenerate case, and we have analogous partition conditions for both Morse-Bott and nondegenerate case. $CZ^{ECH}(\mcal{T},-,m) = m-1$
            \end{enumerate}
            
        \end{enumerate}
        We observe $(\mcal{T},\pm,m)$ imposes different free/fixed end conditions, depending whether it appears as $s\rightarrow \pm \infty$ ends, however we should think of it as being the same generator in the chain complex, as is evidenced by the fact that it is identified to the same nondegenerate orbit set regardless of whether it appears at $+\infty$ or $-\infty$ end.
\end{definition}
We also briefly summarize the analogous result for negative Morse-Bott torus.
\begin{definition}\label{def:mb_generator_neg}
    In the case of negative Morse Bott tori, we use $\gamma_- $ to denote the  elliptic Reeb orbit after perturbation of Conley Zehnder index -1, and let $\gamma_+$ denote the  hyperbolic orbit after perturbation of Conley Zehnder index 0. Let $(\mcal{T},\pm,m)$ denote a Morse-Bott generator.
        \begin{enumerate}
            \item  At the positive end as $s\rightarrow \infty$,
                \begin{enumerate}
                    \item $(\mcal{T},-,m)$ requires all ends fixed at the critical point of $f$ corresponding to $\gamma_-$, corresponds to $(\gamma_-,m)$ in nondegenerate case, both degenerate and nondegenerate case has partition conditions $(m)$. $CZ^{ECH}((\mcal{T},-,m))=-m$
                    \item $(\mcal{T},+,m)$ requires one end free with multiplicity 1, the rest have multiplicity $m-1$ fixed at the critical point of $f$ corresponding to $\gamma_-$. The generator corresponds to $ \{(\gamma_+,1),(\gamma_-,m-1)\}$. $CZ^{ECH}((\mcal{T},+,m))=-m+1$. Partition conditions match.
                \end{enumerate} 
            \item Negative end, as $s\rightarrow -\infty$,
                \begin{enumerate}
                    \item $(\mcal{T},-,m)$ has all ends free, of total multiplicity $m$. This corresponds to $(\gamma_-,m)$ in the nondegenerate case. Partition conditions match. $CZ^{ECH}((\mcal{T},-,m))=-m$.
                    \item  $(\mcal{T},+,m)$ has one fixed end corresponding to the critical point of $f$ at $\gamma_+$ of multiplicity one; the rest are free and of multiplicity $m-1$. This corresponds to the orbit set $\{(\gamma_+,1),(\gamma_-,m-1)$. The partition conditions correspond, and $CZ^{ECH}((\mcal{T},+,m))=-m+1$.
                \end{enumerate}
        \end{enumerate}

\end{definition}

We would also like a more general notion of ECH Conley Zehnder index for when there are more free/fixed ends than allowed by ECH generator conditions are above. To keep track of the more refined intersection theory information, we need to make our definition depend slightly on the behaviour of the $J$-holomorphic curve as its ends approach Reeb orbits on Morse-Bott tori. We consider a nontrivial somewhere injective holomorphic curve $u:\Sigma \rightarrow \bb{R} \times Y^3$. We isolate this into the following definition.

\begin{definition}
Let $u:\Sigma \rightarrow \bb{R} \times Y^3$ be a nontrivial somewhere injective holomorphic curve.
Let $\gamma$ be a simple Reeb orbit on a positive Morse-Bott torus. 
\begin{enumerate}
    \item At the $s\rightarrow \infty$ end, suppose $k_{free}$ ends approach $\gamma$ with total multiplicity $N_{free}$, and $k_{fixed}$ ends approach $\gamma$ with total multiplcity $N_{fixed}$, then $CZ^{ECH}(\gamma) := N_{free}$.
    \item At the $s\rightarrow - \infty$ end, suppose $k_{free}$ ends approach $\gamma$ with total multiplicity $N_{free}$, and $k_{fixed}$ ends approach $\gamma$ with total multiplcity $N_{fixed}$, then $CZ^{ECH}(\gamma) := N_{free}+N_{fixed} -1$.
\end{enumerate}
Similarly if $\gamma$ is a simply covered Reeb orbit on a negative Morse-Bott torus.
\begin{enumerate}
    \item At the $s\rightarrow \infty$ end, suppose $k_{free}$ ends approach $\gamma$ with total multiplicity $N_{free}$, and $k_{fixed}$ ends approach $\gamma$ with total multiplcity $N_{fixed}$, then $CZ^{ECH}(\gamma) := -N_{fix}-N_{free}+1$.
    \item At the $s\rightarrow - \infty$ end, suppose $k_{free}$ ends approach $\gamma$ with total multiplicity $N_{free}$, and $k_{fixed}$ ends approach $\gamma$ with total multiplcity $N_{fixed}$, then $CZ^{ECH}(\gamma) := -N_{free}$.
\end{enumerate}
\end{definition}
Note the above definition agrees with that of the ECH Morse-Bott generator.
Then let $u$ be a somewhere injective $J$ holomorphic curve with no trivial cylinder components, and we have chosen which ends of $u$ are fixed/free. Then we define its ECH index using the above notion of ECH Conley-Zehnder index:
\begin{definition}
We define the ECH index of $u$ as:
\begin{equation}
    I(u):= c_\tau(u) + Q_\tau(u) + CZ^{ECH}(u)
\end{equation}
\end{definition}
Note the above definition not only depends on the relative homology class of $u$, it also depends on how the ends of $u$ are distributed among the Reeb orbits (for information of free/fixed beyond that of the Morse-Bott ECH generators)- in particular we have to keep the information of not only how many free/fix ends land on a Morse-Bott torus, we also need to retain the information which ends are asymptotic to which Reeb orbit. 

By using the writhe bound we recover directly
\begin{proposition}
Let $u$ be a $J$-holomorphic map satisfying the conditions above,
\begin{equation}
Ind(u) \leq I(u) -2\dt(u).
\end{equation}
with equality enforcing partition conditions described in the writhe bound section.
\end{proposition}

We next include the case of trivial cylinders in our definition of ECH Conley-Zehnder index.

\begin{definition}
Let $\gamma$ be a simply covered Reeb orbit on a positive Morse-Bott torus. Let $u:\Sigma \rightarrow \bb{R} \times Y$ be a $J$-holomorphic curve with potentially disconnected domain. When we say trivial cylinders below, we allow trivial cylinders with higher multiplicities.
\begin{enumerate}
    \item At the $s\rightarrow \infty$ end, suppose $k_{free}$ ends approach $\gamma$ with total multiplicity $N_{free}$, and $k_{fixed}$ ends approach $\gamma$ with total multiplcity $N_{fixed}$, then $CZ^{ECH}(\gamma) := N_{free}$. Here we allow holomorphic curves to be trivial cylinders.
    \item At the $s\rightarrow - \infty$ end, suppose $k_{free}$ ends approach $\gamma$ with total multiplicity $N_{free}$, and $k_{fixed}$ ends approach $\gamma$ with total multiplcity $N_{fixed}$. If at least one of the approaching ends is not that of a trivial cylinder, then $CZ^{ECH}(\gamma) := N_{free}+N_{fixed} -1$. If all the approaching ends are trivial cylinders, then $CZ^{ECH} : = N_{fixed}$.
\end{enumerate}

Next let $\gamma$ be a simply covered Reeb orbit on a negative Morse-Bott torus.

\begin{enumerate}
    \item At the $s\rightarrow \infty$ end, suppose $k_{free}$ ends approach $\gamma$ with total multiplicity $N_{free}$, and $k_{fixed}$ ends approach $\gamma$ with total multiplcity $N_{fixed}$, If at least one of the approaching ends is not that of a trivial cylinder, then $CZ^{ECH}(\gamma) := 1-N_{free}-N_{fix}$. If there are only trivial cylinders, then $CZ^{ECH}= -N_{fixed}$.
    \item At the $s\rightarrow - \infty$ end, suppose $k_{free}$ ends approach $\gamma$ with total multiplicity $N_{free}$, and $k_{fixed}$ ends approach $\gamma$ with total multiplcity $N_{fixed}$. Then we set $CZ^{ECH}(\gamma) := -N_{free}$. This includes the case of trivial cylinders.
\end{enumerate}
\end{definition}

\begin{proposition}
Let $C$ be a $J$ holomorphic current which can contain trivial cylinders. Each end in $C$ is implicitly assigned ``free'' or ``fixed'', and recall the convention that we can at most designate one end of a trivial cylinder as fixed. With $CZ^{ECH}$ as defined above, we have the inequality:
\[
Ind(C) \leq I(C) -2\dt (C)
\]

\end{proposition}

\begin{proof}
Let $C$ be a $J$-holomorphic current of the form $\{(C_i,m_i\}$ where $C_i$ are pairwise distinct. If $C_i$ is nontrivial, and $m_i >1$, then as in \cite{Hutchings2002}, we can consider $m_i$ copies of $C_i$ translated by $m_i$ distinct factors in the symplectization direction. Then we can represent $(C_i,m_i)$ as $m_i$ distinct somewhere injective $J$-holomorphic curves. We do this for all nontrivial components of $C$. Each resulting end of $C_i$ receives an assignment of ``free/fixed'', hence both sides of the inequality above are defined. (One can make all the copies of $C_i$ coming from $(C_i,m_i)$ have the same free/fixed assignments at their corresponding ends, but this won't be necessary.)

As before this boils down to writhe bounds at $s=+\infty$ and $s=-\infty$.
We first consider $\gamma$ a Reeb orbit on a positive Morse-Bott torus.
We first consider the $s=+\infty$ case. Here for trivial cylinders $q_i=1$ and the linking number is zero, so the same proof as before produces the writhe bound.

In the case $s\rightarrow -\infty$, let $N_{trivial}$ denote the multiplicity of trivial ends. Let $N_{trivial}$ denote the total multiplicity of trivial ends, fixed or free. First assume there is at least one nontrivial end. The apriori bound on writhe is:
\[
w(\xi) \geq - \# \textup{nontrivial ends} + \sum_{i,j \textup{nontrivial ends}} min(q_i,q_j)  + N_{trivial} \cdot (\# \textup{nontrivial ends} ).
\]

With our new definition of $CZ^{ECH}$, we need to establish the writhe bound that
\[
- \# \textup{nontrivial ends} + \sum_{i,j \textup{nontrivial ends}} min(q_i,q_j)  + N_{trivial} \cdot (\# \textup{nontrivial ends} ) \geq N_{free} + N_{fixed} -1 -(k_{fixed})
\]
We use the superscript $^T$ and $^{NT}$ to distinguish whether the multiplicity is coming from trivial ends or nontrivial ends.
But the writhe bound already established implies 
\[
- \# \textup{nontrivial ends} + \sum_{i,j \textup{nontrivial ends}} min(q_i,q_j) \geq  N^{NT}_{free} + N^{NT}_{fixed} -1 - k_{fixed}^{NT}
\]
Then it suffices to establish that
\[
N_{trivial} \cdot (\# \textup{nontrivial ends} ) \geq N_{free}^T +N_{fixed}^T-k_{fixed}^T
\]
which always holds, hence the writhe bound continues to hold. 

When there are only trivial cylinders, the writhe is automatically zero, likewise the writhe bound is trivially satisfied.

We next consider the case $\gamma$ a Reeb orbit on a negative Morse-Bott torus. We first consider the $s\rightarrow -\infty$ case. Since the winding number $\eta$ in this case is bounded below by zero, the writhe bound continues to hold even in the presence of trivial cylinders.

In the case of $s\rightarrow +\infty$, the computation is very much similar to the $-\infty$ end of a positive Morse-Bott torus. Assuming there is at least one nontrivial end
\[
w(\xi) \leq +\#\textup{nontrivial ends} +\sum_{i,j \textup{nontrivial ends}} max(\eta_iq_j,\eta_jq_i) -N_{trivial}\cdot \# \textup{nontrivial ends} \leq -N_{fix}-N_{free} +1+k_{fix}.
\]
With the previous writhe bound we have already proven
\[
\#\textup{nontrivial ends} +\sum_{i,j \textup{nontrivial ends}} max(\eta_iq_j,\eta_jq_i) \leq -N_{fix}^{NT} -N_{free}^{NT}+1+k_{fix}^{NT}
\]
hence suffices to prove 
\[
-N_{trivial}\cdot \# \textup{nontrivial ends} \leq -N_{fix}^{T} -N_{free}^{T}+k_{fix}^{T}
\]
but this follows directly from our assumptions.

In the case there are only trivial ends the total writhe is zero, and the writhe bound is achieved.
\end{proof}

We next establish the subadditivity property of the ECH index.
\begin{proposition}
Let $\mcal{C}_1 = \{(C_a,m_a)\}$ and $\mcal{C}_2 = \{ (C_b,m_b)\}$ denote two $J$-holomorphic currents, and $C_a$ is never the same as $C_b$ unless they are both trivial cylinders (they can be $\bb{R}$ translates of each other). Then their ECH indices satisfy
\begin{equation}
    I(\mcal{C}_1 \cup \mcal{C}_2) \geq I(\mcal{C}_1) + I(\mcal{C}_2) + 2\mcal{C}_1 \cap \mcal{C}_2. 
\end{equation}
In the above $\mcal{C}_1 \cap \mcal{C}_2$ counts the intersection with multiplicity of $C_a$ with $C_b$. Note by intersection positivity each multiplicity is positive. Further by construction the intersection between trivial cylinders is zero.
\end{proposition}
\begin{proof}
We again apply the translation in the symplectization trick to represent nontrival currents $(C_a,m_a)$ (resp. $(C_b,m_b)$) by $m_a$ (rep. $m_b$) distinct somewhere injective curves. After relabelling we can also denote them by $C_a$ (resp. $C_b$). 
We apply the adjunction inequality as in \cite{Hutchings2002,hutching_revisited} to obtain
\begin{equation}
    I(\mcal{C}_1\cup \mcal{C}_2) -I(\mcal{C}_1) -I(\mcal{C}_2) -2 \# \mcal{C}_1\cdot \mcal{C}_2 = CZ^{ECH}(\mcal{C}_1 \cup \mcal{C}_2) - CZ^{ECH}(\mcal{C}_1)-CZ^{ECH}(\mcal{C}_2) - 2\sum_{a, b} l_\tau(C_a,C_b)
\end{equation}
Then this reduces to a local computation relating linking number and our choice of Conley-Zehnder indices. We take this up case by case. First consider $\gamma$ a Reeb orbit on a positive Morse-Bott torus, consider the $s\rightarrow \infty$ end. In this case we have $CZ^{ECH}(\mcal{C}_1 \cup \mcal{C}_2) - CZ^{ECH}(\mcal{C}_1)-CZ^{ECH}(\mcal{C}_2) =0$ and $l_\tau(C_a,C_b) \leq 0$. Hence all the contributions from this end is $\geq 0$.

We next consider $\gamma$ on a positive Morse-Bott torus at $s\rightarrow -\infty$ ends. Because how we assigned Conley-Zehnder indices depends on whether all the ends are trivial, we split into cases. In the case where all ends of $\mcal{C}_1$ and $\mcal{C}_2$ asymptotic to $\gamma$ as $s\rightarrow -\infty$ are trivial, we have again $CZ^{ECH}(\mcal{C}_1 \cup \mcal{C}_2) - CZ^{ECH}(\mcal{C}_1)-CZ^{ECH}(\mcal{C}_2)=0$ and the linking number vanishes. If one of them has non-trivial ends approaching $\gamma$ (WLOG take this to be $\mcal{C}_1$ and take $\mcal{C}_2$ consists purely of trivial ends), then we have the Conley Zehnder contribution being
\[
N_{free}^1+N_{fixed}^1-1 +N_{fixed}^2 - (N_{free}^1+N_{free}^2+N_{fixed}^1+N_{fixed}^2-1) = -N_{free}^2
\]
where we write $N_{free}^1$ to denote the free ends coming from $\mcal{C}_1$ etc. The linking number contribution is bounded below by $2(N_{fixed}^2 + N_{free}^2)$, hence the overall contribution is non-negative.
The case where both $\mcal{C}_1$ and $\mcal{C}_2$ contains nontrivial ends at $\gamma$ as $s \rightarrow -\infty$, then the Conley-Zehnder difference term is just $-1$, and the linking number term $2l_\tau (C_a,C_b)\geq 2$, hence once again the overall contribution is non-negative.

We next consider the case $\gamma$ is a Reeb orbit in a negative Morse-Bott torus. This will be largely analogous to the positive Morse-Bott torus case. For $s\rightarrow -\infty$, we have the Coneley-Zehnder indices contribute zero, and $l_\tau(C_a,C_b) \geq 0$ as $s \rightarrow \infty$, hence the overall contribution is non-negative. We next consider $\gamma$ as $s\rightarrow +\infty$. Again we break into cases because of trivial cylinders. In the case where all ends approaching $\gamma$ from $\mcal{C}_1$ and $\mcal{C}_2$ are trivial cylinders, the Conley-Zehnder index contribution as well as the linking number is zero. Then in the case $\mcal{C}_1$ has nontrivial ends but $\mcal{C}_2$ has all ends trivial, then the Conley-Zehnder index contribution is given by $-N_{free}^2$, and the linking number $\sum2l_\tau(C_a,C_b) \leq  -2(N_{free}^2+N_{fixed}^2)$, hence the overall contribution is nonnegative. Similarly in the case where both $\mcal{C}_1$ and $\mcal{C}_2$ have nontrivial ends, the difference of Conley-Zehnder index contribution is $-1$, whereas the linking number $2l_\tau(C_a,C_b) \leq -2$, hence the overall contribution is positive.
Hence combining all of the above local inequalities we obtain the overall ECH index inequality.
\end{proof}

\subsection{Multiple level cascades and ECH index}
In this subsection we describe ECH index one cascades. We recall ECH index one cascades should come from degenerations of ECH index one curves, and in particular should respect partition conditions on the end points. In particular we should always keep in mind that ECH index one cascades should flow from a generator Morse-Bott ECH $\alpha_1$ to another $\alpha_n$, which includes the information of multiplicities of free/fixed ends that land on Morse-Bott tori.

Given any cascade $\cas{u}$ as given in our previous definition, we first turn it into a ``cascade of currents": $\cas{\mathbf{u}}=\{u^1,..,u^{n-1}\}$.
Then we can proceed to define the ECH index of $\cas{\mathbf{u}}$.
The following is half definition half theorem, as in if this cascade is transverse and rigid and we glued it into a $J$ holomorphic curve the ECH index of its homology class is given by the following calculation. Conversely, if $u$ came from a cascade of curves that came from a degeneration of $I=1$ holomorphic curve in the $\lambda_\dt$ setting, then our definition of $I$ for the cascade of current will also be one.

\begin{definition}
Let $\cas{\mathbf{u}} = \{u^1,...,u^{n-1}\}$ be a height 1 cascade of currents. Let its positive asymptotics be denoted by $\alpha_1$ and negative asymptotics be denoted by $\alpha_n$, both Morse-Bott ECH generators.
We can then define the ECH index for the cascade of currents as:
\begin{equation}
    I(\cas{\mathbf{u}}) = c_1(\cas{\mathbf{u}}) + Q_\tau( \cas{\mathbf{u}}) + CZ^{ECH}(\cas{\mathbf{u}}).
\end{equation}
The $CZ^{ECH}$ index term for cascade is just the ECH index terms of $\alpha_1$ and $\alpha_n$, which corresponds to the nondegenerate ECH Conley Zehnder index once we have identified free/fixed ends with elliptic/hyperbolic orbits. 
The cascade Chern class and relative intersection terms are just the sum of the Chern class of each of the levels, i.e.
\[
c_1(\cas{\mathbf{u}}):= c_1(u^1) + ...+c_1(u^{n-1})
\]
and
\[
Q_\tau(\cas{\mathbf{u}}): = Q_\tau (u^1) + ...+ Q_\tau(u^{n-1})
\]
\end{definition}
We would like to compare the ECH index of cascade to the Fredholm index of the reduced version, because then with enough transversality we would be able to rule out certain configurations of cascade of ECH index one by index reasons. To this end, we decompose the ECH index of a cascade into ECH index of its constituents, as follows:

\begin{proposition}
We assume all ends of $u^2,..,u^{n-2}$ are free, and all ends of $u^1$ and $u^{n-1}$ are considered free except those mandated by $\alpha_1$ and $\alpha_n$, and we recall our conventions on trivial cylinders with only one fixed end. Then let $R_{pos,i+1}'$ denote the number of distinct Reeb orbits on positive Morse-Bott tori approached by nontrivial ends of $u^i$ as $s\rightarrow -\infty$, and let $V_{pos,i+1}'$ denote the total multiplicity of Reeb orbits on positive Morse-Bott tori approached by $u^{i}$ at the $s\rightarrow -\infty$, so that at these Reeb orbits there are only trivial ends as $s\rightarrow -\infty$. Similarly we let $R_{neg,i}'$ denote the number of distinct Reeb orbits on negative Morse-Bott tori approached by nontrivial ends of $u^i$ as $s\rightarrow +\infty$, and let $V_{neg,i}'$ denote the total multiplicity of Reeb orbits on negative Morse-Bott tori approached by $u^{i}$ at the $s\rightarrow +\infty$, so that at these Reeb orbits there are only trivial ends as $s\rightarrow +\infty$.
Then we have 
\begin{align*}
    I(\cas{\mathbf{u}})
    =& I(u^1) ...+ I(u^{n-1})\\
    &-R_{pos,2}'-...-R_{pos,n-1}'- V_{pos,2}'-...-V_{pos,n-1}'-R_{neg,2}' -..-R_{neg,n-1}' -V_{neg,2}'-..-V_{neg,n-1}'
\end{align*}

\end{proposition}
\begin{proof}
Follows directly from definition of ECH Conley Zehnder index.
\end{proof}

\begin{remark}
Note the assignment of free/fixed end points for calculation of ECH index purposes is different from when we defined free/fixed punctures in the calculation of the Fredholm index. 
\end{remark}

\begin{remark}
We remark the above formula makes sense in the case our cascade consists purely of a chain of cylinder at a critical point. If it started at the minimum of $f$, the trick is to notice by our convention all trivial cylinders below it are considered free.
\end{remark}

In order to compare $I(\cas{\mathbf{u}})$ and $Ind(\cas{\tilde{\mathbf{u}}})$, we first define 
\begin{align*}
     I(\cas{\tilde{\mathbf{u}}}) &:= I(\tilde{u^1}) ...+ I(\tilde{u^{n-1}})\\ &-R_{pos,2}'-...-R_{pos,n-1}'- V_{pos,2}'-...-V_{pos,n-1}'-R_{neg,2}' -..-R_{neg,n-1}' -V_{neg,2}'-..-V_{neg,n-1}'
\end{align*}
by removing all multiple covers of nontrivial curves. Note we have
\begin{equation}
    I(\cas{\tilde{\mathbf{u}}}) \leq I(\cas{\mathbf{u}})
\end{equation}
with equality holding only if $\cas{\mathbf{u}}$ is already reduced.
Next we compare $Ind(\cas{\tilde{\mathbf{u}}})$ and $ I(\cas{\tilde{\mathbf{u}}})$. 
\begin{proposition}\label{prop:indexinequality}
$Ind(\cas{\tilde{\mathbf{u}}}) \leq I(\cas{\tilde{\mathbf{u}}})-2\dt(\cas{\tilde{\mathbf{u}}})-1$
\end{proposition}
\begin{proof}
We make a term-wise comparison, e.g. we compare
\begin{equation}
    Ind(\tilde{u}^i)-k_{i+1}' -k_{i+1} +R_{i+1}
\end{equation}
and 
\begin{equation}
    I(\tilde{u}^i)-2\dt(\tilde{u}^i) - R_{pos,i+1}'-V_{pos,i+1}' - R_{neg,i+1}' - V_{neg,i+1}'.
\end{equation}
Note there are two different conventions by which we assigned ``free'' and ``fixed'' ends to ends of curves appearing in the cascade, we will refer to them respectively as the Fredholm convention and the ECH convention. 

We further refine our notation to $k_{pos,i+1}, k_{neg,i+1}, k_{pos,i+1}',k_{neg,i+1}'$ to denote the number of ends among the $k_{i}$ and $k_{i+1}$ ends that land on positive/negative Morse-Bott tori, i.e. we have $k_{i} = k_{pos,i} + k_{neg,i}$.

We first restrict to $1<i<n-1$
To compare these two terms, we first decompose $\tilde{u^i} = C_i\cup T_{free,i} \cup T_{fixed,i}$, where $C_i$ is a collection of nontrivial somewhere injective curves, $T_{free,i}$ is a collection of free trivial cylinders according to Fredholm convention, and $T_{fixed,i}$ is a collection of fixed cylinder according to the Fredholm index convention. 
Assume $C_i$ has $l_{free,i}$ free ends, and $l_{fixed,i}$ ends according to Fredholm convention, then we have:
\[
Ind(C_i\cup T_{free,i} \cup T_{fixed,i}) +l_{fixed,i}\leq I (C_i\cup T_{free,i}\cup T_{fixed,i})-2\dt (C_i\cup T_{free,i}\cup T_{fixed,i}) -|T_{fixed,i}|
\]
We may at later points further refine the notation to $l_{fixed,pos/neg,\pm, i}$ to indicate fixed ends at positive/negative Morse-Bott tori, at positive/negative ends.
Note $T_{fixed,i}$ is regarded as free cylinders when we measure its ECH index. $|T_{fixed,i}|$ denotes the total number of fixed trivial cylinders that appear in this level. 

We will also later refine our notation to distinguish $T_{fixed/free,pos/neg,i}$ for trivial cylinders on positive/negative Morse-Bott tori.

We next consider the case for $i=1$. We can decompose as before $\tilde{u}^1=C_1\cup T_{free,1} \cup T_{fixed,1} \cup T_{fixed,1}'$. We explain the notation. $C_1$ is a collection of nontrivial somewhere injective holomorphic curves. The information of Morse-Bott generator $\alpha_1$ tells us which of $C_1$ should already be considered as fixed as $s\rightarrow \infty$. There are additionally $l_{fixed}$ ends of $C$ that we count as fixed when we compute its Fredholm index because they land on critical points of $f$. $T_{free,1}$ is a collection of free cylinders. $T_{fix,1}$ is a collection of fixed trivial cylinders that come from requirements of $\alpha_1$.
Each positive Morse Bott torus can only have one of these, and they must all be multiplicity 1. $T_{fixed,1}'$ is a collection of trivial cylinders that don't come from requirements of $\alpha_1$ but also happen to land on a critical point of $f$.
The index inequality we have gives:
\[
Ind(C_1\cup T_{fixed,1} \cup T_{free,1}\cup T_{fixed,1}') +l_{fixed,1} \leq I(C_1 \cup T_{fixed,1} \cup T_{free,1}\cup T_{fixed,1}') -2\dt(\tilde{u}^1)- |T_{fixed,1}'|
\]
where for the purpose of computing ECH index we have counted elements of $T_{fixed,1}'$ as free cylinders.

Similarly for the $i=n-1$ level. As before we can decompose $\tilde{u}^{n-1}=C_{n-1}\cup T_{free,n-1} \cup T_{fixed,n-1} \cup T_{fixed,n-1}'$ with the same convention as before. Here we only need to prove:
\[
Ind(C_{n-1}\cup T_{free,n-1} \cup T_{fixed,n-1}) +l_{fixed,n-1}\leq I(C_{n-1}\cup T_{free,n-1} \cup T_{fixed,n-1}) - 2\dt(\tilde{u}^{n-1})-|T'_{fixed,n-1}|
\]
which holds by the one-level ECH index inequality.
When we take the difference between $I(\tilde{\cas{u}})$ and $Ind(\tilde{\cas{u}})$, we can break down their difference into the following form:
\begin{align*}
I(\cas{\tilde{\mathbf{u}}}) =& I(\tilde{u^1}) ...+ I(\tilde{u}^{n-1})\\ &-R_{pos,2}'-...-R_{pos,n-1}'- V_{pos,2}'-...-V_{pos,n-1}'-R_{neg,2}' -..-R_{neg,n-1}' -V_{neg,2}'-..-V_{neg,n-1}'
\end{align*}
and the index term can be re written as
\[
Ind= \sum_i ind(\tilde{u}^i) - \sum_{i=2,...,n-1} (k_{pos,i} + k_{pos,i}'-R_{pos,i}) - \sum_{i=2,...,n-1} (k_{neg,i} + k_{neg,i}'-R_{neg,i})-1-L
\]
If we take their difference, and take advantage of the inequalities we proved in the previous paragraphs, we get:
\begin{align*}
I -Ind  =& \sum I(\tilde{u}^i) - ind(\tilde{u}^i) + \sum_{i=2,...,n-1}((k_{pos,i} + k_{pos,i}'-R_{pos,i} - R_{pos,i}' -V_{pos,i}') \\
&+ \sum_{i=2,...,n-1}(k_{neg,i} + k_{neg,i}'-R_{neg,i} - R_{neg,i}' -V_{neg,i}') +L+1\\
\geq& 2\delta (\cas{\tilde{\mathbf{u}}}) +  \sum_{i=2,..,n-2}( l_{fixed,i} + |T_{fixed,i}| )+  l_{fixed,1} + l_{fixed,n-1} + |T'_{fixed,1}| + |T'_{fixed,n-1}| \\
&+\sum_{i=2,...,n-1}((k_{pos,i} + k_{pos,i}'-R_{pos,i} - R_{pos,i}' -V_{pos,i}') \\
&+ \sum_{i=2,...,n-1}(k_{neg,i} + k_{neg,i}'-R_{neg,i} - R_{neg,i}' -V_{neg,i}') +L+1
\end{align*}
It suffices to prove the above expression is bounded below by one. It suffices to prove
\begin{align*}
    &\sum_{i=2,...,n-1}R_{pos,i} + R_{pos,i}' +V_{pos,i}' + \sum_{i=2,...,n-1}R_{neg,i} + R_{neg,i}' +V_{pos,i}'\\
    &\leq \sum_{i=2,..,n-2}( l_{fixed,i} + |T_{fixed,i}| )+  l_{fixed,1} + l_{fixed,n-1} + |T'_{fixed,1}| + |T'_{fixed,n-1}|  \\
    &+ \sum_{i=2,...,n-1}(k_{pos,i} + k_{pos,i}') + \sum_{i=2,...,n-1}(k_{neg,i} + k_{neg,i}') +L
\end{align*}
We break down the above inequality into several components. We first observe for $i=2,..,n-2$ we have
\[
R_{pos,i+1}' + V_{pos, i+1}' \leq l_{fixed,pos,-\infty,i}+l_{fixed,pos,+\infty,i}+|T_{fixed,i}|+k_{pos, i+1} + k'_{pos,i+1} - R_{pos,i+1}
\]

We first observe the multiplicities counted by $R_{pos,i+1}'$ and $V_{pos,i+1}'$ are disjoint - if a Reeb orbit appear in considerations of $R_{pos, i+1}'$ then it is not considered for $V_{pos,i+1}'$ and vice versa. Multiplicities counted by $V_{pos,i+1}'$ are contained in $k_{pos,-\infty,i+1}$ and $|T_{fixed,i+1}|$, and the Reeb orbits counted by $R_{pos,i+1}'$ are contained in the ends counted by $l_{fixed,pos,-\infty,i+1}$ and $k_{pos,-\infty,i+1}$. We observe for this range of $i$, we only needed to use the fixed ends of $C_{i}$ in $l_{fixed,i}$ as $s\rightarrow -\infty$ to achieve this inequality, and the prescence of $l_{fixed,i,pos,+\infty}$ will make this inequality strict by that factor. Finally we observe $k'_{pos,i+1} - R_{pos,i+1}\geq 0$. This concludes this inequality.

We next consider the case for $i=1$ for positive Morse-Bott tori, i.e. we consider the inequality
\[
R'_{pos,2} + V'_{pos,2} + R_{pos,2} \leq l_{fixed,pos,-\infty 1} + l_{fixed,pos,+\infty 1} + |T'_{fixed,pos,1}| + k_{pos,2} + k_{pos,2}'
\]
This inequality does not hold in general.
We first observe $k_{pos,2}' -R_{2,pos} \geq 0$, and the Reeb orbits counted by $R_{pos,2}'$ are included in $k_{pos,2}$ and $l_{fixed,pos,-\infty,1}$. The issue for $V_{pos,2}'$ is slightly more subtle, because each positive Morse-Bott torus can contain one fixed trivial cylinder that is not included in $|T_{fixed,pos,1}'|$, hence a Reeb orbit counted by $V_{pos,2}'$ that does not necessarily appear on the right hand side. If we follow this trivial cylinder downwards, if we encounter an end of a non-trivial $J$-holomorphic curve that approaches this Reeb orbit at $s\rightarrow \infty$, then it will contribute to $l_{fixed,pos,+\infty,i}$ terms in one of the lower levels. And this $l_{fixed,pos,+\infty,i}$ term was not used in our previous computations, so after we add up all the terms in the inequality, the overall inequality will still hold.

If we go downwards and do not see a nontrivial end, then there must be a trivial cylinder at the bottom level of the cascade making a contribution to $T_{fixed,,pos,n-1}'$ located at this specific Reeb orbit on this positive Morse-Bott torus. This cylinder counted by $T_{fixed,pos,n-1}'$ is not used anywhere else in any of our other inequalities, so makes up for the deficit coming from the $i=1$ inequality.

Finally we consider the terms on the last level concerning the positive Morse-Bott tori contributing to our inequality. This is just 
\[
|T_{fixed,pos,n-1}'| \geq 0
\]
which holds trivially. $|T_{fixed,pos,n-1}'|$ being nonzero does not necessarily mean our inequality is strict, as some of these may be borrowed to make the inequality hold on the $i=1$ level as per above.

We now repeat the analogous series of inequalities concerning negative Morse-Bott tori. We first prove the inequalities
\[
R_{neg,i}+R'_{neg,i}+V'_{neg,i} \leq k_{neg,i}+k'_{neg,i}+l_{fixed,neg,+\infty,i} + |T_{fixed,neg,i}|
\]
for $i$ in range $2,...,n-2$. We have as before that $R_{neg,i} \leq k_{neg,i}$. Similarly the count of orbits in $R_{neg,i}'$ is included $k'_{neg,i}$ and $l_{fixed,neg,+\infty,i}$, and the count of $V_{neg,i}'$ is included among $T_{fixed,neg,i}$ and $k'_{neg,i}$. This concludes the proof of this inequality.

Next we focus on the $i=n-1$ case. We consider the inequality
\[
R'_{neg,n-1} + V'_{neg,n-1} + R_{neg,n-1} \leq l_{fixed,neg,+\infty,n-1} + |T'_{fixed,neg,n-1}| + k_{neg,n-1} + k_{neg,n-1}'
\]
This does not always hold, as before we first observe $k_{pos,n-1} - R_{neg,n-1} \geq 0$, and $R'_{neg,n-1}$ is included in $l_{fixed,neg,+\infty, n-1}$ and $k_{neg,n-1}$. However each negative Morse-Bott torus can contain one fixed trivial cylinder not included in $T'_{fixed,neg,n-1}$. If we follow this trivial cylinder upwards, if we encounter an end of a non-trivial $J$-holomorphic curve that approaches this Reeb orbit at $s\rightarrow -\infty$, then it will contribute to $l_{fixed,neg,-\infty,i}$ terms in one of the upper levels. And this $l_{fixed,,neg,-\infty,i}$ term was not used in our previous computations, so after we add up the terms in the inequality, the overall inequality will still hold.

If we go upwards and do not see a nontrivial end, then there must be a trivial cylinder contributing to $T'_{fixed,neg,1}$ appearing at the very same Reeb oribt. This cylinder's contribution is not used up by any of our previous inequalities, so makes up for the deficit in the above inequality.

The $i=1$ level terms for negative Morse-Bott tori is simply 
$|T_{fixed,neg,1}'| \geq 0$ which holds trivially. This inequality being strict does not necessarily imply the overall inequality is strict, by the mechanism discussed above.

Adding up the above inequalities we get the inequality in the proposition.
\end{proof}
We now state some consequences of the ECH index one condition, assuming transversality can be satisfied.
\begin{corollary} \label{conditions on currents}
Assuming $J$ can be chosen to be good, and we have a height one cascade $\cas{u}$. Then we pass to cascade of currents $\cas{\mathbf{u}}$, the ECH index being one imposes the following conditions:
\begin{enumerate}
    \item $\cas{\mathbf{u}}$ is reduced.
    \item All flow times are strictly positive.
    \item All curves are embedded. Curves on the same level are disjoint.
    \item Each level only has one nontrivial curve, the rest are trivial cylinders.
    \item With the above choice of fixed/free ends, all curves obey partition conditions of free ends for ends that do not land on critical points. They obey the partition conditions for fixed ends for those that land on critical points of $f$.
    \item For any nontrivial curve $C$ appearing in the cascade of currents $\mathbf{\cas{u}}$:
    \begin{itemize}
        \item If $C$ appears in either $u^1$ or $u^{n-1}$, then its ends can appear on critical points of $f$ only as mandated by $\alpha_1$ or $\alpha_2$. All other ends must avoid critical points of $f$.
        \item If $C$ appears in a level between $u^1$ and $u^{n-1}$, its ends can only end on a critical point of $f$ if this end is then connected by a fixed chain of trivial cylinders to fixed points mandated by $\alpha_1$ or $\alpha_n$. All other ends avoid critical points of $f$, and hence are free.
        \item Further, if we see a chain of fixed trivial cylinders connecting a positive or negative end of $C$ to a critical point of $f$, suppose the fixed Reeb orbit is called $\gamma$. Then no nontrivial end may land on $\gamma$ on any of the levels of the components of the chain of trivial cylinders in either $s\rightarrow +\infty$ or $s\rightarrow -\infty$. On the level where $C$ is asymptotic to $\gamma$ as $s\rightarrow \infty$ or $s\rightarrow -\infty$, the end of $C$ is the only end that is asymptotic to $\gamma$ as $s\rightarrow +\infty$ and $s\rightarrow -\infty$ respectively.
    \end{itemize}

    \item In particular, if $C$ is a nontrivial curve in the cascade, and an end of $C$ is asymptotic to $\gamma$, a Reeb orbit in the $s\rightarrow +\infty$ (resp. $-\infty$) end, then no other curve (or other ends of $C$) in the same level may be asymptotic to $\gamma$ as $s\rightarrow +\infty$ (resp. $-\infty)$. 
    \item If an end of a nontrivial curve $C$ is asymptotic to $\gamma$ with multiplicity $>1$, as $s\rightarrow \infty$, and if we follow $\gamma $ upwards, e.g. we consider $C'$ in the level above which is asymptotic to $\gamma$ as $s\rightarrow -\infty$. If all curves above $C$ that are asymptotic to $\gamma$ are trivial cylinders, then we cannot draw any conclusions aside from partition conditions of $C$. However, if after some chain of gradient flow lines a nontrivial curve  $C''$ above $C$ is asymptotic to $\phi_T^f(\gamma)$ as $s\rightarrow -\infty$ and is connected to the positive end of $C$ at $\gamma$ via a gradient flow, then by partition conditions both $C$ and $C''$ can only be asymptotic to $\gamma$ with multiplicity 1.
\end{enumerate}
\end{corollary}

\begin{proof}
All statements in the above proposition comes from taking all the inequalities in the previous proposition to be equalities. $(a)$ comes $I(\mathbf{\cas{u}}) = I(\mathbf{\cas{\tilde{u}}})$. $(b)$ comes from $L=0$. $(c)$ comes from $\dt(\mathbf{\cas{u}})=0$.  $(d)$ comes from $Ind=0$, otherwise the cascade lives in a moduli space of dimension greater than zero.  $(e)$ comes from the fact that violations of partition conditions for nontrivial curves would make the inequalities comparing Fredholm index to ECH index strict. 

Next consider $(f)$, for the nontrivial curves appearing in $u^1$ or $u^{n-1}$. We first consider the case of $u^1$. We observe all contributions to $l_{fixed,+\infty,1}$ from the $s\rightarrow +\infty$ must be zero for equality in \ref{prop:indexinequality} to hold. Similarly we observe that for $u^{n-1}$ all contributions to $l_{fixed,-\infty,n-1}$ from the $s\rightarrow -\infty$ must be zero for equality to hold. 

If $C$ is a nontrivial curve between $u^1$ and $u^{n-1}$, we have to separate this into cases. We first assume it has a negative end landing on a critical point of $f$ on a positive Morse-Bott torus. Then this end makes a contribution to $l_{fixed,pos,-\infty}$, and was used in our computation of inequality. Call this Reeb orbit $\gamma$, and consider levels below $C$ that have nontrivial ends asymptotic to $\gamma$ as $s\rightarrow +\infty$. Say this occurs on level $i$. If there are such curves, and if $\gamma$ does not appear as a fixed end assigned by $\alpha_1$ and connected to a trivial cylinder in $u^1$, then it is a appearance of $l_{fixed,pos,+\infty,i}$ that was not used in our proof of inequality in \ref{prop:indexinequality}, hence the inequality is strict. 

The case where $\gamma$ appears in $\alpha_1$ as a fixed end of a trivial cylinder is handled as follows. In the case there is a contribution to $T_{fixed,pos,n-1}'$ on the $u^{n-1}$ level from a trivial cylinder at $\gamma$, then we can use the additional $l_{fixed,pos,+\infty,i}$ at $\gamma$ to make the inequality strict. In the case $T_{fixed,n-1}'$ does not have a trivial cylinder at $\gamma$, then for multiplicity reasons the total multiplicity of nontrivial ends asymptotic to $\gamma$ as $s\rightarrow +\infty$ in the entire cascade must be greater than equal to two. If they come from two different ends (potentially at different levels), then their contribution to $l_{fixed,pos,+\infty,*}$ (of various levels) is at least two, which makes the inequality in proposition \ref{prop:indexinequality} strict. If we only see a single nontrivial end approach $\gamma$ as $s\rightarrow +\infty$ below $u^1$ level, then this end must have multiplicity $\geq 2$, and this violation of writhe inequality also ensures the index inequality is strict.

If no nontrivial curves below $C$ that are positively asymptotic to $\gamma$ exist, then with the negative puncture of $C$ landing at $\gamma$, the negative puncture is connected to the last level $u^{n-1}$ at $\gamma$ via a chain of fixed trivial cylinders. If $\gamma$ is a minimum of $f$, then this is a contribution to $|T_{fixed,pos,n-1}'|$ that was not considered in the proof of inequality. This will make the overall inequality strict if $\gamma$ did not appear as a fixed end connected to a trivial cylinder in $u^1$. If $\gamma$ did appear (as a fixed end mandated by $\alpha_1$), then again for multiplicity reasons there is either an additional $l_{fixed,pos,+\infty i}$ contribution from $s\rightarrow +\infty$ ending on $\gamma$ on one of the middle levels, or $|T_{fixed,n-1}'|$ at $\gamma$ has multiplicity greater than or equal to two. Either case makes the index inequality strict.

However if $\gamma$ is at a maximum of $f$, the inequality is not violated if this is a chain of trivial cylinders connecting to a fixed end mandated by $\alpha_{n}$. If $\alpha_{n}$ assigns free ends to this chain of cylinders, then we have extra contributions to $T_{fixed,pos,n-1}'$ which make the index inequality strict (in this case $\alpha_1$ cannot assign $\gamma$ as a fixed end). Finally if this is indeed a chain of fixed trivial cylinders connecting to a fixed orbit mandated by $\alpha_{n}$, then on the level where $C$ appears no other nontrivial end may be asymptotic to $\gamma$ as $s\rightarrow -\infty$, this is because if this is true, then we consider the inequality for $C$'s level
\[
R_{pos,i+1}' + V_{pos, i+1}' \leq l_{fixed,pos,-\infty,i}+|T_{fixed,pos,i}|+k_{pos, i+1} + k'_{pos,i+1} - R_{pos,i+1}
\]
Both nontrivial ends at $\gamma$ are counted once by $R_{pos,i+1}'$, but twice by $l_{fixed,pos,-\infty, i}$, which makes this inequality strict. This automatically imposes the partition condition $(n)$ on this particular negative end of $C$. 
Further, down this chain of fixed trivial cylinders, all the way to $\alpha_{n}$, no further lower levels may have non-trivial curves whose ends are asymptotic to $\gamma$ as $\rightarrow -\infty$. This is clear for the lowest level $u^{n-1}$. We already argued $l_{fixed,pos,-\infty,n-1}=0$, then all fixed ends landing on $\gamma$ must be fixed ends assigned by $\alpha_{n}$, then the partition conditions imposed by ECH index implies we cannot have both trivial and nontrivial ends at $\gamma$. On levels above the lowest level and below the level of $C$, this follows from the inequality
\[
R_{pos,i+1}' + V_{pos, i+1}' \leq l_{fixed,pos,-\infty,i}+|T_{fixed,pos,i}|+k_{pos, i+1} + k'_{pos,i+1} - R_{pos,i+1}.
\]
If we have both a trivial cylinder and an nontrivial end asymptotic to $\gamma$ in the negative end, they make an overall contribution of $1$ to the left hand side, but make a overall contribution of 2 to the right hand side by increasing $l_{fixed,pos,-\infty,i}$ and $|T_{fixed,pos,i}|$, hence making this inequality strict.

We next consider $C$ has a positive end ending on a critical point of $f$. Call this Reeb orbit of $\gamma$. If $\gamma$ is not a fixed Reeb orbit mandated by $\alpha_{1}$, then this already makes a contribution to $l_{fixed,pos,+\infty,i}$ we did not use in the index inequality, which makes the overall inequality strict. If $\gamma$ indeed appears in $\alpha_1$ and is in fact connected to a trivial cylinder, then either this end of $C$ is connected upwards to $\gamma$ via a sequence of trivial cylinders, or there are more nontrivial ends above $C$ that ends on $\gamma$ as $s\rightarrow +\infty$, but this makes the index inequality strict due to multiplicity reasons ($\alpha_1$ can only require a fixed end of multiplicity 1 at $\gamma$). Hence it must be the case $C$ is connected to $\gamma$ on the top level via sequence of fixed trivial cylinders, and no level above $C$ have nontrivial ends approaching $\gamma$ as $s\rightarrow +\infty$. If a curve above $C$ has a negative end approaching $\gamma$, we are back to the previous case and this also makes the index inequality strict.

The case of negative Morse-Bott tori is similar to positive Morse-Bott tori but with the signs reversed, so we will not repeat it. We remark the proof of Negative Morse-Bott tori is independent of the proof of positive Morse-Bott tori because when we compute $|T_{fixed,i}'|$ the trivial cylinders at negative and positive Morse-Bott tori are independent of each other.

To prove $(g)$ and $(h)$. We already took care of the case a non-trivial curve that is asymptotic to a Reeb orbit corresponding to a critical point of $f$. We next consider the case of free ends. Let our curve be $C$ in some level of the cascade and consider its $+\infty$ free ends asymptotic to positive Morse-Bott tori. We have $k_{pos,i+1}' = R_{pos,i+1}$, this implies each free Reeb orbit as $s\rightarrow +\infty$ is approached by a unique positive end of $C$. The ECH index also imposes partition conditions of $(1,..,1)$, hence this end is simply covered. Recalling $\cas{\mathbf{u}}$ is reduced, any $s\rightarrow -\infty$ free end of curves above $C$ arrived at by following the gradient flow is also simply covered. This proves $(g)$ and $(h)$ for positive Morse-Bott tori. The result for negative Morse-Bott tori holds by considering the negative free ends of $C$.
\end{proof}

We would also like a way to prove that provided our transversality conditions hold (i.e. $J$ is good), $J_\dt$-holomorphic curves of ECH index one degenerate into cascades of height one, as opposed to cascades of greater height. To do this we need a slight strengthening of the above index inequality where we allow fixed trivial cylinders with higher multiplicities.

\begin{proposition}
Let $\alpha_1$ and $\alpha_n$ be ECH Morse-Bott generators, except we relax the condition on multiplicities of fixed/free ends - they are allowed to be arbitrary. Let $\cas{u}$ be a cascade of height one connecting from $\alpha_1$ to $\alpha_n$. Then we have the inequality
\[
Ind(\cas{\mathbf{\tilde{u}}}) \leq I(\cas{\mathbf{u}}) - 2\dt (\cas{\mathbf{u}})-1
\]
\end{proposition}
\begin{proof}
We repeat the proof of index inequality in Proposition \ref{prop:indexinequality} and observe the inequalities concerning the intermediate level curves continue to hold. The issue is in allowing fixed trivial cylinders of high multiplicities allowed by $\alpha_1$ and $\alpha_n$ at the top and bottom levels. We first focus on what happens near positive Morse-Bott tori. 
For simplicity we fix $\gamma$ a Reeb orbit corresponding to the hyperbolic orbit in a positive Morse-Bott torus and consider what happens to ends of holomorphic curves with fixed ends at $\gamma$. As we have seen above the problematic term comes from the inequality
\[
R_{pos,2}' + V_{pos,2}' \leq k_{pos,2} +k_{pos,2}' -R_{pos,2} + l_{fixed,pos,-\infty,1} + l_{fixed,pos,+\infty,1} +|T_{fixed,1}'|,
\]
where $V_{pos,2}'$ can contain fixed trivial cylinders mandated by $\alpha_1$ that appear in $V_{2,pos}'$ but does not appear in $|T_{fixed,1}'|$. For simplicity we consider $T_{\gamma,fixed}$ appearing at $\gamma$ of multiplicity $N$. In order for this to make a contribution to $V_{2,pos}'$ instead of $R_{2,pos}'$, we assume that $u^1$ has no nontrivial end that are asymptotic to $\gamma$ as $s\rightarrow -\infty$. We recall we would like to prove an inequality of the form
\[
I(\cas{u})-1 \geq Ind(\cas{\mathbf{\tilde{u}}}) +2\dt(\cas{u})
\]
Consider for $i=2,...,n-1$, the nontrivial currents $(C_{i,j},m_{i,j}) \subset u^i$, where we think of $m_{i,j}$ as the multiplicity of $C_{i,j}$ (since we are working in the nonreduced case). We assume each $C_{i,j}$ has $l_{i,j}$ ends asymptotic to $\gamma$ as $s\rightarrow \infty$, and suppose $C_{i,j}$ has total multiplicity $n_{i,j}$ asymptotic to $\gamma$ as $s\rightarrow \infty$. Finally let $T_{fixed,n-1,\gamma}$ denote the number of trivial cylinders at the last level $u^{n-1}$ at $\gamma$. We have the inequality
\[
N-\sum_{i,j} m_{i,j} n_{i,j} \leq |T'_{fixed,n-1,\gamma}|.
\]
Let's consider $I(C_{i,j})$, by virtue of it being nontrivial and the writhe inequality, $\sum_j I(C_{i,j}) \geq \sum_{j}(n_{i,j}+1) $. This is coming from the fact in order for the $C_{i,j}$ to exist its Fredholm index must be greater or equal to one, and at the ends of $\gamma$ the ECH index is treated as free ends whereas the Fredholm index is treated as fixed ends. So in passing from $u^i$ to $\tilde{u}^i$ we decreased the ECH index by at least $\sum_{i,j}(m_{i,j}-1)(n_{i,j} +1) $.

We next compare the ECH index of reduced cascade with its Fredholm index, in particular we consider the inequalities 
\[
I(\tilde{u^i}) - Ind (\tilde{u^i})+ R_{pos,i+1}' + V_{pos, i+1}' -[ l_{fixed,i}+|T_{fixed,i}|+k_{pos, i+1} + k'_{pos,i+1} - R_{pos,i+1}] \geq 0
\]
for $i =2,..,n-2$.
We have that by virtue of the writhe inequality occurring at $\gamma$ across these levels, the $\gamma$ orbit's contribution is that the left hand side is at least $\sum_{j} n_{i,j}-l_{i,j}$ bigger than the right hand side.

Finally, on the $u^{n-1}$ level, we originally had the inequality
\[
|T'_{fixed,n-1}| \geq 0
\]
In the above inequality we have included the $|T'_{fixed,n-1,\gamma}|$ term coming from the last level in our cascade contributed by $\gamma$, and the writhe bound for this level also implies this there is also an excess of the index inequality of size $\sum_{j} n_{n-1,j}-l_{n-1,j}$. 

Hence we can think of proving the index inequality as follows: there is a deficit of $N$ at the top level contributed purely by $\gamma$, and by making the inequalities of the lower levels strict, we can make up for it. In passing from nonreduced to reduced curve, the ``excess'' of ECH index is bounded below by $\sum_{i,j}(m_{i,j}-1)(n_{i,j} +1) $. The excess of comparing ECH index of reduced curves $C_{i,j}$ to their Fredholm index coming from writhe inequality is given by $\sum_{i,j} n_{i,j}-l_{i,j}$, and the excess in the index inequality of various levels due to contributions to $l_{fixed,pos,+\infty,i}$ coming from $\gamma$ is precisely $\sum_{i,j} l_{i,j}$.
And on the last level the excess is given simply by
$|T'_{fixed,n-1,\gamma}|$
Hence the excess due to $\gamma$ is bounded below by
\[
\sum_{i,j}(m_{i,j}-1)(n_{i,j} +1) + \sum_{i,j} n_{i,j}-l_{i,j} + \sum_{i,j} l_{i,j} + |T'_{fixed,n-1,\gamma}|
\]
Using the fact $N-\sum_{i,j} m_{i,j} n_{i,j} \leq |T'_{fixed,n-1,\gamma}|$, we see the excess outweighs the deficit at the top level, so fixed trivial cylinders at $\gamma$ will keep the overall index inequality intact. We can apply the same reasoning for every $\gamma$ at positive Morse-Bott tori.

We next consider negative Morse-Bott tori. We assume $\gamma$ is Reeb orbit on a negative Morse-Bott torus, and $\alpha_{n-1}$ assigns a fixed end of multiplicity $N$ to $\gamma$. We consider the overall inequality and show it still holds after we factor in the contributions from other terms. Let $|T_{fixed,1,\gamma}'|$ denote the number of free trivial cylinders located at $\gamma$ at the $u^1$ level. For $i =1,..,n-2$ we consider $(C_{i,j},m_{i,j})\subset u^i$ nontrivial curves that asymptote to $\gamma$ as $s\rightarrow -\infty$. We let $l_{i,j}$ denote the number of such ends at each level and $n_{i,j}$ denote the multiplicity. Then the same proof as before will show the inequality continues to hold.
\end{proof}

In fact we have equality of ECH index to Fredholm index also enforces that the cascade is simple. 

We now take care of the case of height $k$ cascades. 
\begin{proposition}\label{prop:height1}
Consider a sequence of $J_{\dt_n}$-holomorphic ECH index one curves $u_n$ of bounded energy from $\alpha_1$ to $\alpha_{n}$ (as nondegenerate ECH generators) converging to a cascade $\cas{u}$ from $\alpha_1$ and $\alpha_n$ viewed as Morse-Bott ECH generators, then $\cas{u}$ has height one.
\end{proposition}
\begin{proof}
Suppose $\cas{u}$ is a height $k$ cascade, then it can be written as $k$ height $1$ cascades, which we write as $\cas{v_1},...,\cas{v_k}$. We recall that between cascades $\cas{v_i}$ and $\cas{v_{i+1}}$ their end asymptotics are connected by either infinite or semi-infinte gradient flows. We pass each to a cascade of currents, and to each cascade $\cas{\mathbf{v}_i}$ we assign to it two generalized ECH generators at its topmost and bottom-most level, which we write as $\alpha_i$ and $\alpha_{i+1}'$. For $\alpha_i$ we assign all the ends approaching the minimum of $f$ as fixed, and all others are free. For $\alpha_{i+1}'$ we consider all ends approaching the maximum of $f$ are fixed, and the rest are free. The exception to this rule is $\alpha_1$ and $\alpha_{k+1}'$ which we assign Morse-Bott ECH generators corresponding to the degenerating $J_\dt$-holomorphic curve. With this we can assign an ECH index to each cascade $I(\cas{v}_i)$. We can also assign a relative ECH index between the general ECH generators $\alpha_{i}$ and $\alpha_{i}'$, which we write as $I(\alpha_i',\alpha_i)$. This number is always $\geq 0$, and we illustrate it as follows. Let $\mcal{T}$ be a Morse-Bott torus, and suppose coming from $\alpha_i$ there is multiplicity $n_1$ at the minimum of $f$ and $n_2$ away from minimum of $f$. From $\alpha_{i}'$ there is $n_1'$ multiplicity at the maximum of $f$, and $n_2'$ away from the maximum of $f$. Then we have the inequalities
\[
n_1' \geq n_2
\]
and
\[
n_2'\leq n_1.
\]
Then we say contribution to $I(\alpha_i',\alpha_i)$ from this Morse-Bott torus is $ (n_1-n_2')=n_1'-n_2 \geq 0$. Then we add up this term for each Morse-Bott torus that appears in $\alpha_i$. Geometrically this is the total mulitplicity of complete gradient trajectories flowing between $\cas{v_i}$ and $\cas{v_{i-1}}$ and has potentially nonzero contributions to the ECH index. Then the fact that the cascade came from a ECH index one curve implies
\[
I(\cas{v}_1) + I(\alpha_{2}',\alpha_{2}) +...+ I(\cas{v}_k) =1
\]
And by previous proposition each $I(v_i)\geq 0$,with equality only if it consisted entirely of fixed trivial cylinders. Hence there is a unique $\cas{v_i}$ with ECH index 1, the rest have ECH index zero, and all $I(\alpha_i',\alpha_i)=0$. This means there can only be fixed trivial cylinders above and below $\cas{v_i}$ and cannot be infinite gradient flows. This is equivalent to saying the cascade of currents is height one.
\end{proof}

The above gives a description of what ECH index one cascades look like from the perspective of currents, we now reverse the process, and use the above to understand all cascades of curves of ECH index one. We need to add back in the information that was lost from passing from curves to currents. We only care about the cascades of curves that resulted from degeneration of a nondegenerate connected ECH index one curve. Call this curve $C_\dt$. We observe the Fredholm index of $C_\dt$, which we denote by $\text{Fred Ind}(C_\dt)$, is equal to one. We  assume as $\dt \rightarrow 0$, $C_\dt$ degenerates into a 
cascade of curves $\cas{u}$, and denote $\cas{\mathbf{u}}$ the resulting cascade of holomorphic currents. From the above we know $\cas{\mathbf{u}}$ is a cascade of currents of height one, however $\cas{u}$ could apriori be of arbitrary height, and the levels that are removed from $\cas{u}$to form $\cas{\mathbf{u}}$ must all be branched covers of trivial cylinders occurring at critical points of $f$.

The first case we need to consider is if $\cas{\mathbf{u}}$ is empty, then this implies that $\cas{u}$ consists purely of branched covers of trivial cylinders. To be precise $\cas{u}$ may contain many levels that consists of branched covers of trivial cylinders, and levels that begin and end on critical point of $f$, however it may also contain levels where the trivial cylinders (branched covered or not) are away from critical points of $f$. Here we allow levels where there is only a single unbranched cylinder away from critical points of $f$. We assume $C_\dt$ is connected. If at level $i$ a trivial cylinder is at the critical point of $f$ corresponding to elliptic Reeb orbit (hyperbolic for negative Morse-Bott torus) then all levels above $i$ the trivial cylinders that connected to the original cylinder will be at the same Reeb orbit. Similarly if at level $i$ a trivial cylinder is at the hyperbolic orbit (resp elliptic orbit for negative Morse-Bott torus) then all the trivial cylinders below this level connecting to this original (potentially branched cover of) cylinder will also be at the same Reeb orbits.

If all the levels of $\cas{u}$ are at the same Reeb orbit which is also a critical point, then $u$ came from a branched cover of trivial cylinder in the nondegenerate case. If this is not the case, then remove the top most and bottom most levels until none of the trivial cylinders in $\cas{u}$ begin/end on critical point of $f$.
Then as currents we don't care where the branched points are, so we can think of $u'$ as a cascade of currents with only 1 level. Then the ECH index of $\cas{\mathbf{u}}$ is equal to one, which implies $\cas{\mathbf{u}}$ consists of a free trivial cylinder with multiplicity one. Hence the same must be true of $\cas{u}$ and there are no top/bottom branch covers.

We now turn our attention to the case where $\cas{\mathbf{u}}$ is nonempty. 
We shall use the fact the Fredholm index of $C_\dt$ is one to rule out configurations of height $> 1$. We observe the trivial cylinders on levels above/below $\cas{\mathbf{u}}$ admit the following description:
\begin{proposition}
\begin{enumerate}
    \item Let $\mcal{T}$ denote a positive Morse Bott torus contained in the top level of $\cas{\mathbf{u}}$. For curves on the top level of $\cas{\mathbf{u}}$, as $s\rightarrow +\infty$ all free ends have multiplicity one, and avoid critical point of $f$. The fixed end can only have multiplicity one. Hence all branched covers of trivial cylinders above this level can only happen at the critical point of $f$ corresponding to the elliptic orbit. Moreover, because $C_\dt$ obeys partition conditions, the top most level in $\cas{u}$ of the stack of branched trivial cylinders has partition conditions $(1,..,1)$.
    \item Let $\mcal{T}$ denote a negative Morse Bott torus contained in the top level of $\cas{\mathbf{u}}$, as $s\rightarrow +\infty$. The positive free end of the top level of $\cas{\mathbf{u}}$ has multiplicity 1, so there cannot be branched cover of trivial cylinder at the critical point of $f$ corresponding to the hyperbolic orbit. The fixed end at the critical point of $f$ corresponding to the elliptic orbit can have a stack of branched cover of trivial cylinders on top of it on height levels above $u'$, and again by partition conditions on $C_\dt$ the top most level is hit by partition condition $(n)$.
    \item Let $\mcal{T}$ denote a positive Morse Bott torus contained in the bottom level of $\cas{\mathbf{u}}$. The free end has multiplicity one, so there cannot be branched covers of trivial cylinders at the critical point of $f$ corresponding to the hyperbolic orbit. The fixed end at critical point of $f$ corresponding to elliptic end can have a stack of branched cover of trivial cylinders below it on height levels below $u'$, and again by partition conditions on $C_\dt$ the top most level is hit by partition condition $(n)$.
    \item Let $\mcal{T}$ denote a negative Morse Bott torus contained in the bottom level of $\cas{\mathbf{u}}$. As $s\rightarrow -\infty$ all free ends have multiplicity one, and avoid the critical points of $f$. The fixed end can only have multiplicity one. Hence all branched covers of trivial cylinders above this level can only happen at the critical point of $f$ corresponding to the elliptic orbit. Moreover, because $C_\dt$ obeys partition conditions, the bottom most level (in terms of height) of the stack of branched trivial cylinders has partition conditions $(1,..,1)$.
\end{enumerate}
\end{proposition}
In light of the above, we can compute the topological Fredholm index of $C_\dt$ via the following procedure:

First consider the height level corresponding to $\cas{\mathbf{u}}$, we know all trivial cylinders connecting between nontrivial curves are simply covered, so all the possible branched covers that appear on this height level are chains of trivial branched covers of cylinders that connect to the top and bottom levels of $\cas{\mathbf{u}}$. We then create two additional height levels, one above $\cas{\mathbf{u}}$, denoted by $\overline{\cas{\mathbf{u}}}$ and one below $\cas{\mathbf{u}}$, denoted by $\underline{\cas{\mathbf{u}}}$, and push all branch points of trivial cylinders that appear in $\cas{\mathbf{u}}$ onto these 2 levels $\overline{\cas{\mathbf{u}}}$,$\underline{\cas{\mathbf{u}}}$, so that all trivial cylinders that appear in $\cas{\mathbf{u}}$ have no branch point (though they may be multiply covered), and hence are transversely cut out. We recall we assign $Ind(\cas{\mathbf{u}})$ as the dimension of moduli space of $\cas{\mathbf{u}}$ lives in, viewed as a cascade of currents

Then the Fredholm index of $C_\dt$ is computed as:
\begin{align*}
    &Ind(C_\dt) =\\
    & Ind(\cas{\mathbf{u}})+1 -\chi (\overline{\cas{\mathbf{u}}}) -\chi (\underline{\cas{\mathbf{u}}})
\end{align*}
Note by the ECH index assumption $Ind(\cas{\mathbf{u}})=0$, so it will enforce no branched cover of trivial cylinders appear. Hence we have the proved the following proposition:
\begin{proposition} \label{nice cascades}
Suppose $J$ is chosen to be good, if $C_\dt$ is a sequence of connected nontrivial ECH index one curves of bounded energy that converges to a cascade of curves, $\cas{u}$, then either
\begin{itemize}
    \item $\cas{u}$ is a free cylinder of multiplicity one
    \item $\cas{u}$ is the same as a height one cascade of currents of ECH index one, described above, and all trivial cylinders that appear in levels of $\cas{u}$ either unbranched chains of fixed trivial cylinders, or trivial cylinders over a Reeb orbit of multiplicity one.

\end{itemize}
In the latter case, $\cas{u}$ does not contain a sequence of fixed trivial cylinders that do not connect to any nontrivial $J$ holomorphic curve. See Convention \ref{nontrivial}.
\end{proposition}
We call cascades of curves of ECH index one of the form stated in the above theorem \emph{good cascades of ECH index 1}.

Then this is more or less a complete characterization of ECH index one cascades we should count in the Morse-Bott case provided we can achieve enough transversality.
Assuming transversality conditions, we quote a theorem from \cite{Yaocas} to show ECH index one cascades can be glued uniquely (up to translation) to ECH index one curves.
\begin{theorem}[Theorem 3.5 in \cite{Yaocas}] \label{gluing theorem}Assuming transversality conditions \ref{assumption}, any given ECH index one cascades can be glued uniquely to ECH index one $J_\dt$-holomorphic curves for sufficiently small values of $\dt >0$ up to translation in the symplectization direction. 
\end{theorem}
The key is to note ECH index one and transversality implies all of the cascades above are transverse and rigid, as in Definition 3.4 of \cite{Yaocas} and hence can be glued.
The final ingredient we need is to show that assuming $J$ is good, the set of good ECH index one cascades is finite. To do this we need the notion of $J_0$ index for cascades.

\section{Finiteness} \label{Finite}
In order to prove the differential in Morse-Bott ECH is well defined we need to prove the for given generators $\alpha,\beta$ the set of good ECH index one cascades from $\alpha$ to $\beta$ is finite. For $J$-chosen to be good, we already know this set is a zero dimensional space, hence it suffices to prove that it is compact. To this end we develop the analogue of $J_0$ index in the Morse-Bott world. We start with 1-level cascades then build upwards to $n$ level cascades. In this section we assume $J$ is good throughout.

\subsection{Level 1 cascades}
Consider an level 1 cascade of ECH index 1 from generator $\alpha$ to $\beta$. In anticipation of multiple level ECH index 1 cascades, here we relax some (but not all) of the conditions on $\alpha,\beta$ to remove conditions that require certain free/fixed ends (depending on whether we are on a positive/negative Morse-Bott torus) to only have multiplicity 1. This corresponds to relaxing the condition in the nondegenerate case to only allow hyperbolic orbits of multiplicity one (see Theorem \ref{mbgenerator}). We recall the consequences of generic $J$:
\begin{enumerate}
    \item For positive Morse-Bott tori, as $s\rightarrow \infty$, all free ends are disjoint and are asymptotic to Reeb orbits in the torus with multiplicity 1. Let $n^{pos,free}_+$ denote the number of such orbits.
    \item For positive Morse-Bott tori, the fixed ends at $s\rightarrow \infty$ are disjoint from the free ends. They are hit with partition condition $(1)$. Suppose there are $N^{pos,fix}_+$ such ends.
    \item  For positive Morse-Bott tori, as $s\rightarrow -\infty$ all free ends are disjoint and cover the Reeb orbits in the torus with multiplicity 1. Let $n^{pos,free}_-$ denote the number of such orbits.
    \item For positive Morse-Bott tori, as $s\rightarrow -\infty$, all fixed ends have partition conditions $(n)$. Suppose there are $N^{pos,fix}_-$ such ends, each with multiplicity $n^{pos,fix}_{-,j}$
    \item For negative Morse-Bott tori, as $s\rightarrow \infty$, all free ends are disjoint and cover the Reeb orbits in the torus with multiplicity 1. Let $n^{neg,free}_+$ denote the number of such orbits.
    \item For negative Morse-Bott tori, the fixed ends at $s\rightarrow \infty$ are disjoint from the free ends. They are hit with partition conditions $(n)$. Suppose there are $N^{neg,fix}_+$ such ends with multiplicity $n^{neg,fix}_{+,j}$
    \item  For negative Morse-Bott tori, as $s\rightarrow -\infty$ all free ends are disjoint and cover the Reeb orbits in the torus with multiplicity 1. Let $N^{neg,free}_-$ denote the number of such orbits.
    \item  For negative Morse-Bott tori, as $s\rightarrow -\infty$ there is only 1 fixed end for each Morse-Bott tori, and has partition conditions $(1)$. Let there be $N^{neg,fix}_-$ such ends total
\end{enumerate}
\begin{definition}
For a level 1 good ECH index 1 cascade $C$ connecting generator $\alpha$ to $\beta$, we define:
\begin{equation}
    J_0(C,\alpha,\beta) := -c_\tau(C) + Q_\tau(C,C) - [\sum_j (n^{pos,fix}_{-,j} -1)]- [\sum_j (n_{+,j}^{neg,fix}-1)]
\end{equation}
\end{definition}
We observe that $J_0(C,\alpha,\beta)$ can be computed from the knowledge of $\alpha,\beta$ and the relative homology class of $C$ alone. We also remark that the $J_0$ index can be similarly be defined for nontrivial curves of higher ECH index, as long as they satisfy the long list of partition conditions we listed above, and the same genus bounds below holds. We shall have need for this fact for the proof of finiteness below.

Then we have the following genus bound:
\begin{proposition}
Let $g$ denote the genus of a holomorphic curve $C$. Then we have the upper bound
\begin{equation}
    -\chi(C)  \leq J_0(C,\alpha,\beta).
\end{equation}
\end{proposition}
\begin{proof}
We recall the adjunction formula in our case says
\[
c_\tau(C)= \chi(C) +Q_\tau(C) +w_\tau(C)-2\dt(C)
\]
plugging this into $J_0$ yields
\[
J_0(C,\alpha,\beta) = -\chi(C) -w_\tau(C) -[\sum (n^{pos,fix}_- )-1]- [\sum n_+^{neg,fix}-1] +2\dt(C)
\]
hence it suffices to prove \[
-w_\tau -[\sum (n^{pos,fix}_- )-1]- [\sum n_+^{neg,fix}-1] \geq 0.
\]
We break this into cases. If $C$ is a trivial cylinder, then this is trivial. If $C$ has a nontrivial component along with fixed trivial cylinders, we only consider the nontrivial component, also denoted by $C$. All of the computations below follow from the computations of the writhe bound:
\begin{itemize}
    \item At a positive Morse-Bott torus
    \begin{itemize} 
        \item $s\rightarrow \infty$, free end. $-w_\tau \geq 0$ because the multiplicity is one.
        \item $s\rightarrow \infty$, fixed end $-w_\tau \geq 0$ because multiplicity is one.
        \item $s\rightarrow -\infty$, free end. $w_\tau \geq 0 $ by multiplicity.
        \item $s\rightarrow -\infty$, for given fixed end $j$, the writhe at this end satisfies $w_\tau \geq n_-^{pos,fix}-1$.
    \end{itemize}
    \item At a negative Morse-Bott torus 
    \begin{itemize}
         \item $s\rightarrow \infty$, free end. $-w_\tau \geq 0 $ due to multiplicity constraints.
        \item $s\rightarrow \infty$, for a single fixed end $j$, the writhe satisfies $-w_\tau \geq n_+^{neg,fix}-1$.
        \item $s\rightarrow -\infty$, free end. $w_\tau \geq 0 $ due to multiplicity constraints.
        \item $s\rightarrow -\infty$, fixed end. $w_\tau \geq 0$ by multiplicity.
    \end{itemize}
\end{itemize}
combining all of the above we conclude our inequality.

\end{proof}
\subsection{Multiple level cascades}
We now explain how to generalize the definition of $J_0(C,\alpha,\beta)$ to good ECH index one cascades of arbitrary number of levels. Consider a $n$ level cascade $\cas{u}= \{u^1,..,u^n\}$ of ECH index one with input $\alpha$ and output $\beta$. Recall we have so called fixed chains of trivial cylinders, i.e. chain of trivial cylinders that all begin/end on a fixed end orbit of either $\alpha$ or $\beta$ until this chain of trivial cylinders meet an nontrivial holomorphic curve in one of the intermediate levels (which has an fixed end at said Reeb orbit). We remove all of these kinds of trivial cylinders, then the number $J_0$ is defined for each of the intermediate cascade levels, which we denote by $J_0(u^i)$, then we define the $J_0$ of the entire cascade as
\begin{definition}
\begin{equation}
J_0(\cas{u}) := \sum J_0(u^i)
\end{equation}
\end{definition}
We observe this definition also only dependents on the relative homology class and $\alpha,\beta$.
Recall the Euler characterisitc of the cascade $\chi (\cas{u})$ is the Euler characterstic of the surface obtained if we glued a cylinder between each matching end of $u^i$ and $u^{i+1}$, clearly then the Euler characteristic of the cascade  is the sum of the Euler characteristic of each of its components. Applying the proposition for level one cascades we get
\begin{proposition}
\[
-\chi (\cas{u})\leq J_0(\cas{u}).
\]
\end{proposition}

\subsection{Finiteness}
We finally prove
\begin{theorem}\label{thm:finiteness}
Given generators $\alpha, \beta$, the moduli space of good ECH index 1 cascades from $\alpha$ to $\beta$ is compact.
\end{theorem}
\begin{proof}
Let $\{\cas{u}_m\}$ be a sequence of good ECH index one cascades from $\alpha$ to $\beta$. Each $\cas{u}_m$ is a cascade of the form $\{u_{m}^{n}\}_n$. We show $\{\cas{u}_m\}$ has a convergent subsequence. From the Morse-Bott assumption there is an upper bound to how many cascade levels there are, so we pass to a subsequence where they all have $N$ levels. For each $n=1,..,N$, we apply the compactness for holomorphic current from \cite{bn} to each of $u^{m}_{n}$. To see this, note for fixed $n$, the energy constraint of $\{\cas{u}_m\}$ and Morse-Bott condition implies there are only finitely many possible choices for the positive and negative asymptotics of $u^{m}_n$, so we pick a subsequence (also denoted by $u^{m}_n$) where the positive and negative asymptotics of $u^{m}_n$ is independent of $m$. Here, by positive and negative asymptotics of $u^{m}_n$ we simply mean the Morse-Bott tori $\mcal{T}$ that $u^{m}_n$ are asymptotic to at its positive/negative ends, and the total multiplicity of Reeb orbits at each such Morse-Bott tori.

Then using the Gromov compactness for currents (see \cite{bn}) applied to $\{u^{m}_n\}$ we conclude we can refine a further subsequence of $\{u^{m}_n\}$ (for all $n=1,..,N$) with the same relative homology class (our notion of relative homology class here is in $\mathcal{H}_2 (-,-,Y)$)). Now for each $u^{m}_n$ simply the knowledge of its asymptotics (which we can read off directly: by virtue of being part of ECH index one cascade all the ends that avoid the critical points of $f$ are free, and those at critical points of $f$ is fixed) and its relative homology class provides an upper bound on its $J_0$ index. This upper bound on $J_0$ then provides a bound the genus of each $u_n^m, n=1,...,N$. 

With the genus bound we can apply SFT compactness: for fixed $n$, we observe $u^{m}_{n}$ cannot break into a building, for that would yield (if we view $\cas{u}_m$ as cascade of currents) an ECH index 1 cascade of currents with $T_i=0$, which does not exist by genericity conditions. Similarly ruled out by genericity conditions are overlapping free ends and free ends migrating to fixed ends. The $u^n_m$ also cannot converge to a multiple cover of nontrivial curve, for that would yield an ECH cascade of current of index 1 with multiple covers of nontrivial curve, which is ruled out by genericity. Hence we conclude that $\{\cas{u}_m\}$ has a subsequence that converges to a ECH index 1 cascade, and hence we have compactness.
\end{proof}

\section{Computing ECH in the Morse-Bott setting using cascades}\label{section:computing using cascades}

We now define the Morse-Bott ECH chain complex (over $\bb{Z}_2$). We write the chain complex as \[
C_*^{MB}(\lambda,J) := \bigoplus_{\Theta_i} \bb{Z}_2\la \Theta_i \ra.
\]
Here $\Theta_i = \{(\mcal{T}_j,\pm,m_j)\}$ denotes a collections of Morse-Bott ECH generators. Suppose we can choose our $J$ to be good, the differential, which we write as $\partial_{MB}$ is defined as
\begin{equation}
\la\p_{MB} \Theta_1, \Theta_2 \ra
:=   \left\lbrace  
  \begin{tabular}{@{}l@{}}
 $\bb{Z}_2$\, \textup{count of J-holomorphic cascades}\, $\mcal{C}$\, \textup{of ECH index} \,$I=1$,\\
\textup{so that as}  $s\rightarrow +\infty, \, \mcal{C}$ \,\text{approaches}\, $\Theta_1$  \textup{and as} $s\rightarrow -\infty$, \\
 $\mcal{C}$\, \text{approaches} $\Theta_2$.
   \end{tabular}
  \right\rbrace
 \end{equation}
We clarify that in the above definition the cascade $\mcal{C}$ must be decomposable into $\mcal{C}_0 \sqcup \mcal{C}_1$, where $\mcal{C}_0$ is a (potentially empty) collection of fixed trivial cylinders with multiplicity, and $\mcal{C}_1$ is a good ECH index one cascade. We note if $(T,n)$ is an element of $\mcal{C}_0$, if it is positively asymptotic to Morse-Bott ECH generator $(\mcal{T},n,\pm)$, it is also negatively asymptotic to the Morse-Bott ECH generator $(\mcal{T},n,\pm)$ (thus far we only considered nontrivial cascades when we talked about their asymptotics).

We note by Theorem \ref{thm:finiteness} the operator $\p_{MB}$ is well defined.
\begin{theorem}\label{theorem:cobordism in general}
Assuming $J$ is good, the chain complex  $(C_*^{MB},\partial_{MB})$ computes $ECH(Y,\xi)$.
\end{theorem}
Before we prove this theorem we choose a generic family of almost complex structures $J_\dt$.

Recall that the traditional definition of ECH requires choosing a generic $J$ from a residual subset of almost complex structures. For fixed $\dt>0$, we say $J_\dt$ is ECH adapted if it is an almost complex structure with which the ECH chain complex is well defined.

\begin{definition}
Consider $\dt \in (0,\dt_0]$, we say a path of almost complex structures $J_\dt$, each compatible with $\lambda_\dt$ for any $\dt \in (0,\dt_0]$, is generic if for any collection of Reeb orbits $\alpha,\beta$, the moduli space 
\begin{equation}
    \mcal{M}(\alpha,\beta,\dt):=\{ (u,\dt) | \db_{J_\dt} u =0, u \, \textup{somewhere injective}, \lim_{s\rightarrow +\infty} u \, \textup{converges to} \,\,\alpha, \lim_{s\rightarrow -\infty} u\, \textup{converges to} \,\, \beta\}
\end{equation}
is cut out transversely.
\end{definition}

\begin{theorem}\label{generic path J}
There is a small enough $\dt_0>0$ so that there is a generic path of almost complex structures $J_\dt$, $\dt \in (0,\dt_0]$ so that:
\begin{itemize}
    \item $J_{\dt_0}$ is ECH adapted.
    \item $\lim_{\dt \rightarrow 0} J_\dt = J$, where $J$ is a generic almost complex structure we have chosen above to count ECH index one cascades.
    \item $|J-J_\dt| \leq C\dt$ in $C^k$ norm, $k>100$, and $J_\dt$ take the prescribed form near small fixed neighborhood of Morse-Bott torus described in Section \ref{degenerations}.
    \item For a residual subset $S \subset (0,\dt_0]$, for all $\dt \in S$, $J_\dt$ is ECH adapted.
\end{itemize}
\end{theorem}
\begin{proof}
This is standard application of Sard-Smale theorem.
\end{proof}

\begin{proof}[Proof of theorem \ref{theorem:cobordism in general}] We observe for fixed $L>0$, there are only finitely many ECH index 1 cascades of energy $<L$. We fix $\dt_0$ small enough so that for all $\dt \in (0,\dt_0]$ the cascades can be glued (uniquely in our sense specified) to ECH index 1 curves.

We assume $\dt_0$ is such that $J_{\dt_0}$ is ECH adapted. We recall we have chosen a generic family $J_\dt, \dt\in[0,\dt_0]$ so that the space: 
\[
\{ (u,J_{\dt}) \,| \, \dt \in (0,\dt_0]\, \, u \, J_\dt \, \text{holomorphic, somewhere injective ECH index 1}\}
\]
modulo translation is a 1-manifold (not necessarily compact). A SFT compactness theorem (\cite{Yaocas,BourPhd,SFT}) tells us the $\dt =0$ ends of this manifold are precisely the good ECH index one cascades. 

We recall there is a residual set $A \subset (0,\dt_0]$ so that for all $\dt \in A$, $J_\dt$ is ECH adapted and the ECH homology can be computed by counting ECH index one $J_\dt$ holomorphic curves for $\dt \in A$.

We make the following observation:
if $u_\dt$ and $v_\dt$ are $J_\dt$-holomorphic curves of ECH index one that converge to the same cascade as $\dt \rightarrow 0$, by the gluing theorem, for small enough $\dt$ $u_\dt$ and $v_\dt$ are in fact the same curve up to $\bb{R}$ translation.

Then we claim we can find small enough $\dt' \in A$ so that the corbordism from $\dt =0 $ to $\dt'$ built by $\{ (u,J_{\dt}) \,| \, \dt \in (0,\dt']\, \, u \, J_\dt \, \text{holomoprhic, somewhere injective ECH index 1}\}$ is the trivial cobordism. Suppose not, then for arbitrarily small $\dt$ we can find $u_\dt$ a ECH index one somewhere injective curve that does not come from gluing, take $\dt \rightarrow 0$ and after taking a subsequence, $u_\dt$ degenerates into a good ECH index one cascade, but by our observation must have come from a curve obtained by gluing together an ECH index one cascade, contradiction.
\end{proof}

\section{ECH index one curves of genus zero} \label{section:ECH index one curves of genus zero}
We showed in the previous section that when there is enough transversality for cascades, the cascades of ECH index one take a particularly nice form. However this is not always achievable, except in special circumstances. In this section and the next we outline some special circumstances in which transversality can always be achieved. Here we consider the case where all ECH index one curves in the perturbed picture must have genus zero. This is the case for $T^3$ and some toric domains.

We shall use a slightly different description of cascades that do not allow for the presence of trivial cylinders. We will call this description ``tree-like'' cascades and will be described below. The reason we can use this description is that if the curve has genus zero, we can do the gluing without requiring that between each adjacent cascade levels there is a \emph{single} flow time parameter; instead we can assign a different flow time between each pair of adjacent nontrivial curves.

We use the following convention to represent our holomorphic curves. We use a vertex to represent a $J$ holomorphic curve of genus zero, and use directed edges to denote the positive and negative punctures of the curve. Edges directed away from the $J$-holomorphic curve correspond to positive punctures, and edges directed towards the vertex correspond to negative punctures.
The figure below illustrates how we go from $J$-holomorphic curve to vertex with directed edges.
\begin{figure}[!ht]
\centering
\includegraphics[width=.4\linewidth]{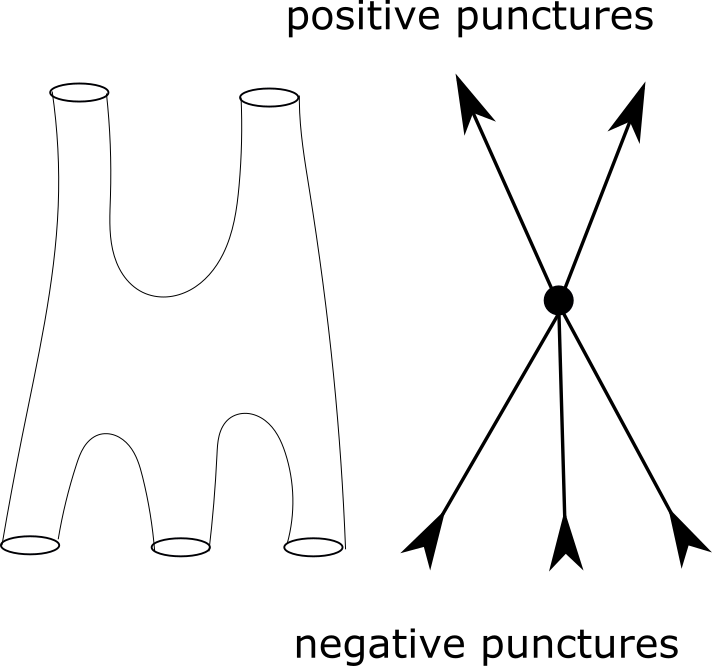}
\caption{Passing from genus zero curve to vertex with edges}

\end{figure}

Then a \emph{height one cascade with tree-like compactifications} from Morse-Bott ECH generator consists of the following data:
\begin{enumerate}
    \item A collection of vertices $\{v_1,..,v_n\}$ each equipped with the data of inward and outward pointing edges. Each vertex has at least one outgoing edge. Each edge is also equipped with the information of which Reeb orbit it lands on.
    \item Given two vertices $v_i$ and $v_j$, if we can find a Morse-Bott torus $\mcal{T}$ so that a positive puncture of $v_i$ lands on $\gamma$, and if we follow the gradient flow for time $T_{i,j}\in [0,\infty)$ along $\gamma$ we arrive at a negative puncture of $v_j$ landing on the corresponding orbit, then we say it is possible to connect $v_i$ and $v_j$ via the given pair of edges. The data of a height one cascade in this compactification consists of choices of connections between the vertices of $\{v_1,..,v_n\}$, so that after we connect the edges, we obtain a connected tree. See figure below for an example. We call these connections internal connections.
    \item The positive punctures of $\{v_1,..,v_n\}$ that are not assigned internal connections are assigned free/fixed as per required by ECH generator $\alpha_1$, and likewise for negative punctures and $\alpha_n$.
    \begin{figure}[!ht]
    \centering
    \includegraphics[width=.4\linewidth]{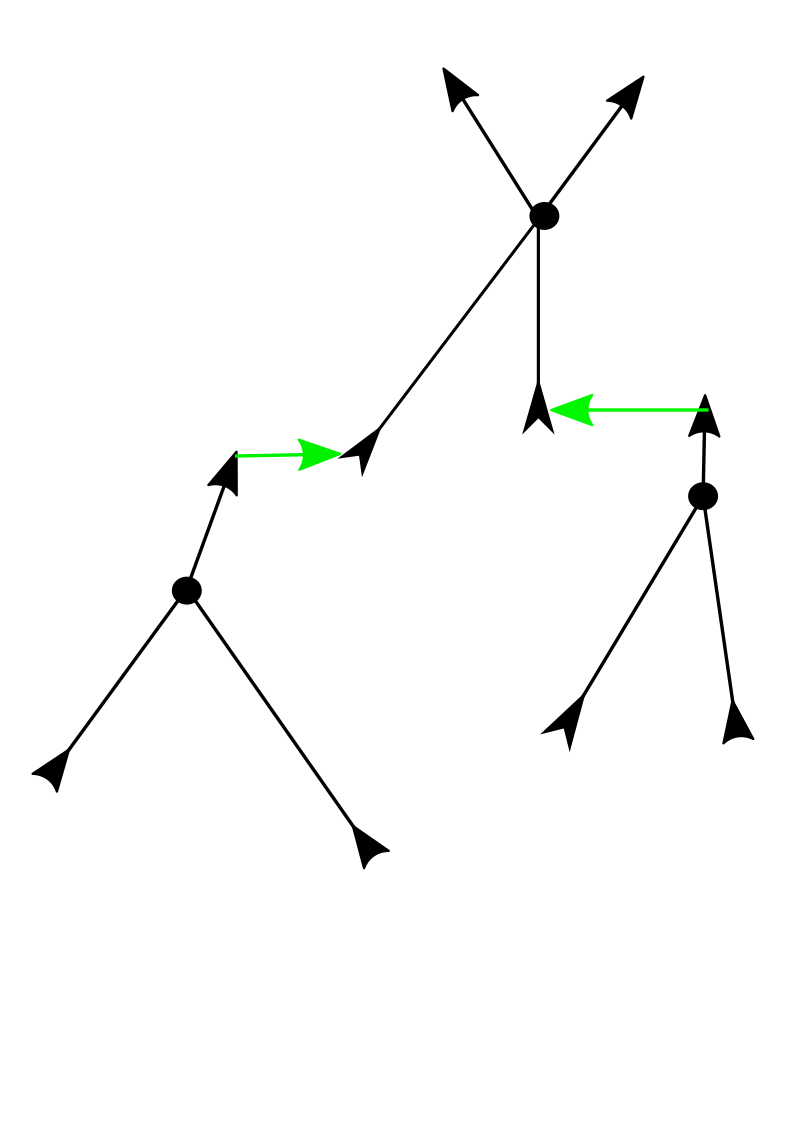}
    \caption{Cascade with tree like compactification. The green arrow denote finite gradient flow lines.}

    \end{figure}
\end{enumerate}

For genus zero $J_\dt$-holomorphic curves degenerating into a cascade with our previous compactification, we can easily pass to a tree like compactification by removing all the trivial cylinders.

Given a cascade of height one with tree like compactification, which we write as $\cas{u}=\{v_1,..,v_n\}$. We can compute its ECH index as follows: we treat all edges participating in internal connections as free, then the ECH index is simply given by
\[
I(\cas{u}) = I(v_1)+....+I(v_n)-n+1.
\]
In order to talk about Fredholm index we also need to pass to the reduced cascade $\cas{\tilde{u}}$ consisting of curves $\{\tilde{v_1},..,\tilde{v_n}\}$. If in our tree like compactification all free ends assigned by $\alpha_1$ and $\alpha_n$ as well as all internal connections avoided critical points of $f$, then the reduced cascade lives in a transversely cut out moduli space of dimension
\[
\sum_i Ind(\tilde{v_i})-1
\]
since being tree like removes the condition of needing to have the same flow time between adjacent cascade levels.

Hence to achieve the necessary transversality conditions to count ECH index one cascades, we choose a generic $J$ so that
\begin{enumerate}
    \item For any punctured sphere that is the domain of a $J$-holomorphic curve, we endow it with an assignment of incoming and outgoing punctures, and for each end we assign a free/fixed end; and if an end is assigned fixed it must land on a Reeb orbit corresponding to a critical point of $f$ under the $J$-holomorphic map; and if an end is free it must avoid critical points of $f$. Then all moduli spaces of somewhere injective $J$ holomorphic curves with the above information are transversely cut out with dimension given by the index formula.
    \item For any two curves $v_1$ and $v_2$ satisfying the above condition and both rigid, if their free ends land on the same Morse-Bott torus from opposite sides (one as a positive puncture the other as a negative puncture), then they do not land on the same Reeb orbit in the Morse-Bott family (we only care about where they land on the Morse-Bott torus and ignore information of multiplicity, i.e. even if they cover the same Reeb orbit of different multiplicity on their free ends, this is prohibited).
\end{enumerate}

The above conditions are easily achieved by choosing a generic $J$ by classical transversality methods.
We next consider cascades of height one. We observe we have the inequality (if we treat all internal connections as free for both ECH index and Fredholm index)
\[
I(\cas{u}) -n \geq \sum Ind(\tilde{v_i}) -1 \geq 0
\]
since each $\tilde{v}_i$, by virtue of it existing and transversality conditions, must have Fredholm index $\geq 0$. ECH index one implies $Ind(\tilde{v_i}) =1$, hence all these curves are rigid, and embedded. By the above genericity of $J$ all flow times are nonzero, and the cascade itself is already reduced. All free ends and ends coming from internal connections avoid critical points of $f$. Also observe that by partition conditions derived previous sections that between internal connections, the participating edges can only over Reeb orbits with multiplicity one.

Then suppose a sequence of genus zero ECH index one $J_\dt$ holomorphic curves from $\alpha_1$ to $\alpha_n$ degenerates into a cascade with tree like compactification for arbitrary height. This just means we allow internal connections adjoint to each other with semi-infinite or infinite gradient trajectories. Then for each internal connection whose flow time is infinite, we separate them into two different cascades. Then we get a collection of height one cascades each of which is tree like. We write them as $\cas{u}_1,...,\cas{u}_k$. Then we can assign generalized ECH generators to ends of $\cas{u}_i$ as before, and the ECH index one condition imposes
\[
I(\cas{u}_1) + I(\cas{u_2}) +\dots +I(\cas{u_k}) + \textup{relative difference between ECH generators} =1
\]
By relative difference between ECH generators we mean the same construction as proposition \ref{prop:height1}.
We have for all height one cascades that
\[
I(\cas{u_i})-1 \geq Ind(\cas{\tilde{u_i}}) \geq 0
\]
Hence there is either a unique cascade $\cas{u_i}$ of index zero, or the entire cascade is just one gradient flow line. By considerations of topological Fredholm index we also rule out additional branched cover of trivial cylinders at the top/bottom level of the cascade with tree- like compactifications. Hence using the above description we have the following proposition.
\begin{proposition}
In the nondegenerate case, ECH index one curves of genus zero degenerate into ECH index one tree like cascades that are reduced and transversely cut out.
\end{proposition}

We call the type of cascades of the above proposition ``good ECH index one tree like cascades'', because we eliminated branched covers of trivial cylinders via topological Fredholm index.

As in the previous section we choose $J_\dt$ to be a generic family of almost complex structures satisfying the same conditions as Theorem \ref{generic path J}.

We then quote a gluing theorem from \cite{Yaocas}.
\begin{theorem}
Let $\cas{u}$ be a good ECH index one cascade of genus zero as per above, then for small enough $\dt>0$ there exists a unique (up to translation) $J_\dt$-holomorphic curve in an $\epsilon$ neighborhood of this cascade.
\end{theorem}
\begin{proof}
The main difference is that because the whole curve is genus zero, we no longer need to make sure the pregluing is well defined by restricting our choice of asymptotic vectors to $\hat{\Delta}$, as in proposition 8.28 in \cite{Yaocas}.
\end{proof}

We define a chain complex as before. We 
We write the chain complex as \[
C_*^{MB,tree}(\lambda,J) := \bigoplus_{\Theta_i} \bb{Z}_2\la \Theta_i \ra.
\]
We use the superscript ``tree'' to denote the fact we are counting tree like cascades. As before $\Theta_i = \{(\mcal{T}_j,\pm,m_j)\}$ denotes a collections of Morse-Bott ECH generators.
After we choose a generic $J$, all good tree like cascades are transversely cut out. Then we define the differential $\p_{MB}^{Tree}$ to be
\begin{equation}
\la\p^{tree}_{MB} \Theta_1, \Theta_2 \ra
:=   \left\lbrace  
  \begin{tabular}{@{}l@{}}
 $\bb{Z}_2$\, \textup{count of tree like J-holomorphic cascades}\, $\mcal{C}$\, \textup{of ECH index} \,$I=1$,\\
\textup{so that as}  $s\rightarrow +\infty, \, \mcal{C}$ \,\text{approaches}\, $\Theta_1$  \textup{and as} $s\rightarrow -\infty$, \\
 $\mcal{C}$\, \text{approaches} $\Theta_2$.
   \end{tabular}
  \right\rbrace
 \end{equation}
As before, we clarify in the cascade $\mcal{C}$ must be decomposable into $\mcal{C}_0 \sqcup \mcal{C}_1$, where $\mcal{C}_0$ is a (potentially empty) collection of fixed trivial cylinders with multiplicity, and $\mcal{C}_1$ is a good ECH index one tree like cascade.
\begin{theorem}\label{theorem:cobordism for genus 0}
Suppose $J$ is chosen to be generic so that all ECH index one good tree like cascades are transversely cut out, and we can choose a generic family of perturbations to $J$, which we write as $J_\dt$ that meets the conditions of Theorem \ref{generic path J}. We further for small enough $\dt>0$, all $J_\dt$-holomorphic curves of ECH index one are genus zero. Then the chain complex $(C_*^{MB,tree}, \p_{MB}^{Tree})$ computes $ECH(Y,\xi)$.
\end{theorem}
\begin{proof}[Proof of Theorem \ref{theorem:cobordism for genus 0}] The same proof as in Theorem \ref{theorem:cobordism in general} works.
\end{proof}

\section{Applications to concave toric domains} \label{section concave}

As an application of our methods we show that for concave toric domains, ECH can be computed via enumeration of ECH index one cascades. By what we proved above, it suffices to show all ECH index one holomorphic curves after the Morse-Bott perturbation have genus zero.

We recall the definition of a concave toric domain. Consider $\mathbb{C}^2$ equipped with the standard symplectic product symplectic form. Consider the diagonal $S^1$ action on $\bb{C}^2$, and the associated moment map $\mu: \bb{C}^2 \rightarrow \bb{R}^2$ given by
\[
\mu(z_1,z_2) =(\pi|z_1|^2, \pi |z_2|^2).
\]

Let $\Omega \subset \mathbb{R}^2$ be a domain in the first quadrant of $\bb{R}^2$, we define the domain $X_\Omega$ to be
\[
X_\Omega: = \{(z_1,z_2) | \mu (z_1,z_2) \in \Omega\}.
\]

Suppose $\Omega$ is a domain bounded by the horizontal segment from $(0,0)$ to $(a,0)$, the vertical segment from $(0,0)$ to $(0,b)$ and the graph of a convex function $f:[0,a] \rightarrow [0,b]$ so that $f(0)=b$ and $f(a)=0$. We further assume $f$ is smooth, $f'(0)$ and $f'(a)$ are irrational, $f'(x)$ is constant near $0$ and $a$, and $f''(x)>0$ whenever $f'(x)$ is rational, then we say $X_\Omega$ is a \textbf{concave toric domain}. Note our definition is slightly more restrictive than that of \cite{intoconcave}, because we are not interested in capacities; we need the boundary of $X_\Omega$ to be well behaved enough to define ECH.

For a concave toric domain $X_\Omega$, its boundary $\partial X_\Omega$ is a contact 3-manifold diffeomorphic to $S^3$.
We now describe the Reeb orbits that appear in $\partial X_\Omega$. We also note their Conley Zehnder indices, having chosen the same trivializations as in \cite{intoconcave}.
\begin{enumerate}
    \item $\gamma_1 = \{ (z_1,0) \in \partial X_\Omega \}$. The orbit $\gamma_1$ is elliptic with rotation angle $-1/f'(a)$, hence $CZ(\gamma_1^k) = 2\floor{-k/f'(a)}+1$
    \item $\gamma_2 = \{ (0,z_2) \in \partial X_\Omega \}$. The orbit $\gamma_2$ has rotation angle $-f'(0)$, hence $CZ(\gamma_2^k)=2\floor{-kf'(0)}+1$.
    \item Let $x\in (0,a)$ be such that $f'(x)$ is rational. Then the torus described by $\{(z_1,z_2) | \mu(z_1,z_2) = (x,f(x))\}$ is a (negative) Morse-Bott torus. Each Reeb orbit has Robbin-Salamon index $-1/2$.
\end{enumerate}

We say a bit more about the Reeb dynamics for the third case. Consider the point $(x,f(x))$ so that $f'(x)$ is rational. We set $f'(x) = tan(\phi), \phi \in (-\pi/2,0)$. Then the Reeb vector field is given by (see \cite{mihai})
\[
R = \frac{2\pi}{-x \sin(\phi) + f(x) \cos(\phi)}(-\sin \phi \partial_{\theta_1} + \cos(\phi) \partial_{\theta_2}).
\]

For large action $L>0$, we perturb each Morse-Bott torus to a pair of orbits, one elliptic, the other hyperbolic. Then an ECH generator $\alpha = \{\alpha_i,m_i\}$ is a collection of nondegenerate Reeb orbits with multiplicities. We associate to each ECH generator a \textbf{combinatorial generator}.

\begin{definition}(see \cite{intoconcave})
A combinatorial generator is a quadruple $\tilde{\Lambda} = (\Lambda,\rho,m,n)$ where
\begin{enumerate}
    \item $\Lambda$ is a concave integral path from $(0,B)$ to $(A,0)$ such that the slope of each edge is in the interval $[f'(0),f'(a)]$.
    \item $\rho$ is a labeling of each edge of $\Lambda$ by $e$ or $h$.
    \item $m$ and $n$ are nonnegative integers.
\end{enumerate}
\end{definition}

Let $\Lambda_{m,n}$ denote the concatenation of the following sequence of paths:
\begin{enumerate}
    \item The highest polygonal path with vertices at lattice points from $(0,B+n+\floor{-mf'(0)})$ to $(m,B+n)$ which is below the line through $(m,B+n)$ with slope $f'(0)$.
    \item The image of $\Lambda$ under the translation $(x,y)\mapsto (x+m,y+n)$.
    \item The highest polygonal path with vertices at lattice points from $(A+m,n)$ to $(A+m+\floor{-n/f'(a)},0)$ which is below the line through $(A+m,n)$ with slope $f'(a)$.
\end{enumerate}
Let $\mcal{L}(\Lambda_{m,n})$ denote the number of lattice points bounded by the axes and $\Lambda_{m,n}$, not including the lattice points on the image of $\Lambda$ under the translation $(x,y) \mapsto (x+m,y+n)$. We then define 
\[
I^{comb}(\Lambda_{m,n}) =2 \mcal{L}(\Lambda_{m,n}) + h(\Lambda)
\]
where $h(\Lambda)$ is the number of edges in $\Lambda$ labelled by $h$.
To each ECH generator $\alpha = \{(\alpha_i,m_i)\}$ we associate a combinatorial ECH generator $(\Lambda,m,n)$ as follows. The number $m$ is the multiplicity of $\gamma_2$ as it appears in $\alpha$, and the integer $n$ is the multiplicity of $\gamma_1$ as it appears in $\alpha$. For other (nondegenerate) Reeb orbits of $\alpha$, they all come from small perturbations of Morse-Bott tori. If $\gamma \in \alpha$ is a Reeb orbit that comes from breaking the degeneracy of a Morse-Bott torus at $(x,f(x))$, then let $v_1$ be the smallest positive integer so that $v_2=f'(x)v_1\in \bb{Z}$. Let $v$ denote the vector $v=(v_1,v_2)$. The path is obtained by taking each Reeb orbit $\gamma$ in $\alpha$ that come from Morse-Bott tori, associating to it the vector that is $v$ multiplied by the multiplicity of $\gamma$ as it appears in $\alpha$, and concatenating these vectors in order of increasing slope. The labelling $\rho$ is obtained by labelling the vector associated to $\gamma$ the letter $h$ if $\gamma$ is hyperbolic, and $e$ if $\gamma$ is elliptic.

\begin{proposition}(\cite{intoconcave})
If $C$ is a current from $\alpha$ to $\beta$, its ECH index is given by $I^{comb}(\alpha) - I^{comb}(\beta)$.
\end{proposition}

For future usage, we also record how the Chern class is computed (see \cite{intoconcave}).
Let $\alpha$ denote a ECH generator, we associate to it the combinatorial generator $(\Lambda,\rho,m,n)$, then we take 
\[
c_\tau(\alpha) = A+B+m+n. 
\]
Then if we have a $J$-holomorphic curves from ECH generator $\alpha$ to $\beta$, then its relative first Chern class is calculated by $c_\tau(\alpha)-c_\tau(\beta)$.

We need a version of the local energy inequality, which we take up presently. Versions of this inequality have appeared in \cite{pfhdehn,ziwenyao,simplicityconjecture,choi2016combinatorial}. Consider the boundary of $\Omega$ with its intersections with the two coordinate axes removed, then 
its preimage under the moment map is an interval times a two torus. We write the two torus as $(x_1,x_2) \in S^1_1 \times S^1_2$, where the first $S^1_1$ is the $S^1$ coming from rotation in the first complex plane $\bb{C}$, and the second $S^1$ comes from the second copy of $\bb{C}$. We use $\bb{Z}\oplus \bb{Z}$ to denote the lattice of first homology with $\bb{Z}$ coefficients. Consider a Morse-Bott torus at $(x,f(x))$ with $f'(x) = v_2/v_1$ as before, then the homology class of the Reeb orbit is given by the pair $(-v_2,v_1)\in \bb{Z}^2$ (this is true before or after the Morse-Bott perturbation).

Consider $F_{[x_0,x_1]}$, by which we denote the preimage of the graph $\{(x,f(x)) | x\in [x_0,x_1]\}$ under the moment map. We similarly consider $F_x$, which is the preimage of $(x,f(x))$ under the moment map. Let $C$ be a somewhere injective $J$ holomorphic curve, we consider $C\cap F_{x_0}$ (we choose $x_0$ generically so this intersection is transverse). We orient this intersection using the boundary orientation of $C\cap F_{[x_0-\epsilon,x_0]}$. Its homology class in $\bb{Z}^2$ we write as $[F_x]$.

\begin{proposition}
Let $(p,q)\in \bb{Z}^2$ denote the homology of $C\cap F_{x_0}$, then we have the inequality
\[
p +f'(x) q \geq 0.
\]
We further observe equality holds only if $C$ is a trivial cylinder.
\end{proposition}

\begin{proof}
We consider $C\cap F_{[x_1,x_2]}$, and observe with our conventions $\partial (C\cap F_{[x_0,x_1]}) = C\cap F_{x_1} - C\cap F_{x_0}$.
We next consider 
\begin{align*}
    \int_{C\cap F_{[x_1,x_2]}} d\lambda &= \int_{C\cap F_{x_1}} \lambda - \int_{C\cap F_{x_0}} \lambda\\
    &= \int_{C\cap F_{x_1}} r_1d\theta_1 + r_2 d\theta_2 - \int_{C\cap F_{x_0}} r_1 d\theta_1 +r_2 d\theta_2\\
    &= (x_1-x_0) p +(f(x_1)-f(x_0))q \geq 0.
\end{align*}
By taking the limit $x_0\rightarrow x_1$, we conclude the proof.
\end{proof}

Suppose the $J$-holomorphic $C$ current connects from $\alpha_+$ to $\alpha_-$ and has ECH index one. Suppose $C$ does not contain trivial cylinder components, hence it is embedded. Let $\alpha_+$ contain $\gamma_1$ with multiplicity $n_+$, the orbit $\gamma_2$ with multiplicity $m_+$, and contains $e_+$ distinct elliptic orbits and $h_+$ hyperbolic orbits. Suppose further $C$ has $k_m^+$ ends at $\gamma_2$, with multiplicities $m_+^i$, and $C$ has $k_n^+$ ends at $\gamma_1$ with multiplicities $n_+^i$ Likewise we use $m_-,n_-,e_-,h_-$ and $k_m^-,m_-^i,k_n^-, n_-^i$ to denote the respective quantities in $\alpha_-$, except here $e_-$ denotes the number of elliptic Reeb orbits counted with multiplicity. Then the key is the following proposition (similar proofs have appeared in \cite{simplicityconjecture,pfhdehn,choi2016combinatorial})

\begin{proposition}\label{thm:concave genus zero}
For the case of concave toric domains, after a small perturbation away from the Morse-Bott degeneracies, all ECH index one curves have genus zero.
\end{proposition}

\begin{proof}
\textbf{Step 1}
We know that the integers $m_\pm^i$ and $n_\pm^i$ satisfy partition conditions because $C$ has ECH index one. Recall that for an elliptic Reeb orbit of rotational angle $\theta$, suppose $C$ is asymptotic to this Reeb orbit at its positive ends with multiplicity $m$. Consider the line $y=\theta x$ on the $x-y$ plane, then draw the maximal concave polygonal path connecting lattice points beneath $y=\theta x$. This polygonal path $\mcal{P}$ starts at the origin and connects to $(m,\floor{m\theta})$. The horizontal displacements of the edges in this path we will write as $(m_i)$ and take the convention that if $i<j$, then $m_i$ is the segment before $m_j$ if we count starting from the origin. This gives an integer partition of $m$, which is the partition conditions for positive ends of $C$ that are asymptotic to this Reeb orbit. 

We observe that $\sum_i \floor{m_i \theta} = \floor{\theta m}$. To see this, first it follows from the properties of the floor function that 
\[
\sum_i \floor{m_i\theta} \leq \floor{m\theta}.
\]
For the converse inequality, consider the polygonal path $\mcal{P}$ with vertices at $(\sum_{i}^km_i,\floor{\sum_{i}^km_i\theta })$. It suffices to show 
\[
\floor{m_k\theta} \geq \floor{\sum_i^k m_i \theta} - \floor{\sum_i^{k-1} m_i \theta}.
\]
This follows from the fact that 
\[
\theta \geq \frac{\floor{\sum_i^k m_i \theta} - \floor{\sum_i^{k-1} m_i \theta}}{m_k}
\]
which is a consequence of the fact that $\mcal{P}$ is maximally concave.

We next recall the partition conditions for negative ends of $C$ asymptotic to the Reeb orbit with rotation angle $\theta$. Consider the line $y=\theta x$, and the minimal convex path above $y=\theta x$ that connects between $(0,0)$ and $(m,\ceil{m\theta}$ through lattice points. The horizontal displacements of the edges of of this path are labelled (in order) $m_i$, and form the partition conditions for ends of $C$. Using a very similar proof as before, we can show
\[
\sum \ceil{m_i \theta} = \ceil {m\theta}.
\]
 
Then we can compute the Fredholm index of $C$ as 
\begin{align*}
    Ind(C) =& 2g-2 +(e_+ + h_+ + k_m^++k_n^+)+(e_- + h_- + k_m^-+k_n^-)\\
    &+2(A_++B_+ +m_+ +n_+ -A_--B_--m_--n_-)\\
    & -e_+  +e_- \\
    & + (k_n^+ + k_m^+ +k_m^- + k_n^-) \\
    & + \sum_{i=1}^{k_n^+} 2\floor{- n_+^i /f'(a)} + \sum_{i=1}^{k_m^+} 2\floor{-m_+^i f'(0)} - \sum_{i=1}^{k_n^-} 2\ceil{- n_-^i /f'(a)} - \sum_{i=1}^{k_m^-} 2\ceil{-m_-^i f'(0)}.
\end{align*}

\textbf{Step 2}
To analyze the above equation further, we first note that 
\begin{equation} \label{Eq:yaxiscomparison}
A_+ + n_+ +\sum_{i=1}^{k_m^+}\floor{-m_+^i f'(0)} -A_--n_- -  \sum_{i=1}^{k_m^-}\ceil{-m_-^i f'(0)} \geq 0
\end{equation}

This is accomplished by considering the interior intersections of $C$ with $\gamma_2 \times \bb{R}$. All such intersection points are positive, by positivity of intersections. The count of interior intersections is given by (see \cite{mean_action_calabi})
\[
l_+(C,\gamma_2) - l_-(C,\gamma_2) 
\]
where $l_+$ denotes the linking number of positive ends of $C$ with $\gamma_2$, and $l_-$ is the linking of negative ends of $C$ with $\gamma_2$. We note the linking numbers in a concave toric domain are calculated as follows (\cite{intoconcave}):
\[
lk(\gamma_1,\gamma_2) =1,\quad lk (\gamma_1, o_v) = -v_2, \quad lk (\gamma_2,o_v) =v_1, \quad lk(o_v,o_w) = \min \{-v_1w_2,-v_2w_1\}.
\]

Here we use $o_v$ to denote nondegenerate orbits that come from perturbing a Morse-Bott torus at $(x,f(x))$, with $f'(x) =v_2/v_1$.

From this we see that $lk_+ =A_+ + n_+ +\sum_{i=1}^{k_m^+}\floor{-m_+^i f'(0)} $, and $lk_-= A_-+n_- + \sum_{i=1}^{k_m^-}\ceil{-m_-^i f'(0)}$. The $A_\pm$ terms come from ends of $C$ asymptotic to $o_v$, the $n_\pm$ term comes from ends of $C$ asymptotic to $\gamma_1$, and the floor and ceiling terms come from ends of $C$ asymptotic to $\gamma_2$ and the fact that $C$ has ECH index one.
From the partition conditions we see that $\sum_{i=1}^{k_m^+}\floor{-m_+^i f'(0)} = \floor{-m_+ f'(0)}$. 
Likewise we can show 
\[
B_+ +m_+ +\sum_{i=1}^{k_n^+} 2\floor{- n_+^i /f'(a)} - B_--m_- -\sum_{i=1}^{k_n^-} 2\floor{- n_-^i /f'(a)} \geq 0
\]
Hence we conclude from the Fredholm index formula that if $C$ has ends at $\gamma_+$ or $\gamma_-$, then it must have genus 0.

\textbf{Step 3}
Next we consider the case where $C$ has no ends at $\gamma_+$ or $\gamma_-$. We assume $C$ has genus one. Then $A_+=A_-, B_+=B_-$ from Fredholm index considerations.  Let $\Lambda_\pm$ denote the polygonal paths associated to generators $\alpha_\pm$.
We first show $\Lambda_+$ lies outside $\Lambda_-$. By the above we already know they agree at end points. 

As a preamble, we consider the homology classes $F_x\cap C$. First for $x$ very close to zero, say equal to $\ep>0$, let $[F_\ep] = (p,q)$. Then we have $p+ f'(0)q\geq 0$. Similarly consider $[F_{1-\ep}] = (-p,-q)$. We have $-p -f'(a)q\geq 0$. Adding these inequalities to get $(f'(0)-f'(a))q\geq 0 $ from which we deduce $q\leq 0$. Then we have $-f'(a)q \geq p \geq -f'(0)q$, which implies $p=q=0$. Incidentally this implies a kind of maximal principle for holomorphic curves. Note $p+f'(x)q=0$ only if the curve is a branched cover of a trivial cylinder. This implies for our curves they are confined to have $x\in (0,1)$.

Next we compute $[F_x]$ for any $x$ irrational and $\ep>0$ sufficiently small. We have
\begin{align*}
&[F_x] - [F_\ep] + \text{homology class of Reeb orbits in $[\ep,x]$ approached by positive ends of $C$} \\
&-\text{homology class of Reeb orbits in $[\ep,x]$ approached by negative ends of $C$} = 0.
\end{align*}

Next we consider the no crossing of polygonal paths. 

Suppose the no crossing result does not hold, since we know $\Lambda_\pm$ have the same beginning and end points, there must exists two intersection points which we call $(a,b)$ and $(c,d)$, with $a<c$. Then on the interval $(a,c)$ the path $\Lambda_-$ is strictly above $\Lambda_+$ except at end points where they overlap. Form the line connecting $(a,c)$ and $(b,d)$, we can find $x_0\in (a,c)$ such that $f'(x_0) = \frac{d-b}{c-a}$. We compute $[F_{x_0-\ep}]$ and apply the local energy inequality to it. We use $x_0-\ep$ to avoid the case where $x_0$ is the $x$ coordinate of lattice points in $\Lambda_\pm$, practically this will not make a difference. 

Let the lattice point $(p,q)$ have the following property: it is a vertex on $\Lambda_+$, the edge to the left of this lattice point has slope less than $f'(x_0)$, and the edge to the right of this vertex has slope greater than equal to $f'(x_0)$. Then the contribution to $[F_{x_0-\ep}]$ from $\Lambda_+$ is simply $(-(B-q),-p)$. We also consider the contribution of $F_{x_0-\ep}$ from $\Lambda_-$, which takes the form $(B-q',p')$. The lattice point $(p',q')$ on $\Lambda_-$ is chosen the same way as $(p,q)$. If no such vertex exists, then $\Lambda_-$ must overlap with the line segment connecting $(a,b)$ and $(c,d)$. Then the point $(p',q')$ is still the lattice point on $\Lambda_-$ which corresponds to the left most end point of where $\Lambda_-$ overlaps with the line connecting $(a,b)$ to $(c,d)$. In either case the local energy inequality says that 
\[
(q-q') + \frac{d-b}{c-a} (p'-p) \geq 0
\]
We first assume $(p',q')$ is not on the line connecting $(a,b)$ and $(c,d)$, then this means that the point $(p,q)$ is further away from the line connecting $(a,b)$ to $(c,d)$ than $(p',q')$. Geometrically this is described by
\[
(b-d)(p-p') +(c-a) (q-q') < 0.
\]
which is impossible.
Now assume $(p',q')$ is on the line connecting $(a,b)$ to $(c,d)$, then since we have chosen $[F_{x_0-\epsilon}]$, we must have $p'<p$. The energy inequality implies 
\[
\frac{q-q'}{p-p'} > \frac{d-b}{c-a}
\]
contradicting the geometric picture.

\textbf{Step 4}. After we proved no-crossing in the previous step, we show there cannot be a genus one curve satisfying the assumptions of the previous step. The Fredholm index formula tells us that (recall we are assuming $g=1$)
\[
1 = h_+ + h_- + 2e_-
\]
which means $e_-=0$ and at most one of $h_+$ and $h_-$ is one. If $h_+=1$, and $h_-=0$, then $\alpha_- = \emptyset$. By inspection $C$ cannot have ECH index one.

On the other hand, if $h_+=1$ and $h_-=1$, then $\Lambda_-$ consists of a single line segment. $\Lambda_+$ has the same end points as $\Lambda_-$ and is concave, hence must also agree with $\Lambda_-$ as polygonal paths. One checks easily that in this case the ECH index cannot be one.

This concludes the proof that all ECH index one curves have genus zero.
\end{proof}

After we have proved all ECH index one curves have genus zero, we can then use the tree like compatification to describe the moduli space of cascades. However there is the complication that there are two nondegenerate orbits, $\gamma_+$ and $\gamma_-$. So in the tree like compactification, we allow the ends of $J$-holomorphic curves to land on nondegenerate orbits. Furthermore,  connecting between two nontrivial curves, instead of a gradient trajectory, it could be that adjacent ends of $J$-holomorphic curves land on the same non-degenerate orbits and no gradient trajectories connect between them. See figure \ref{fig:tree like with nondeg}.
\begin{figure}[!ht]
    \centering
    \includegraphics[width=.4\linewidth]{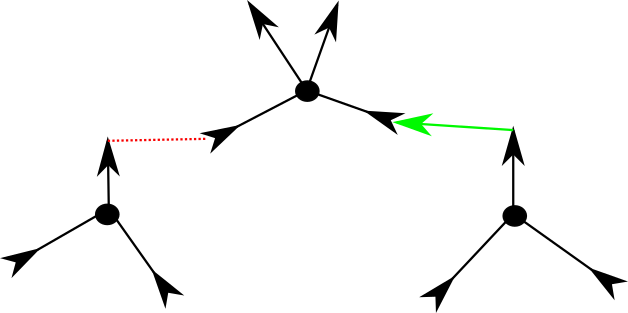}
    \caption{Cascade with tree like compactification for concave toric domains. The unconnected ends of holomorphic curves can land on either Morse-Bott tori or nondegenerate Reeb orbits. The green arrow denotes a finite gradient flow line connecting between two adjacent ends that land on Morse-Bott tori. The dashed line is used to indicate the adjacent ends land on non-degenerate Reeb orbits, and there is no need for gradient trajectories to connect between them.}
    \label{fig:tree like with nondeg}
    \end{figure}

Given such a cascade of ECH index one, we can cut it into subtrees along each matching pair of nondegenerate orbits, see figure \ref{cut tree like cascade}.

\begin{figure}[!ht]
    \centering
    \includegraphics[width=.4\linewidth]{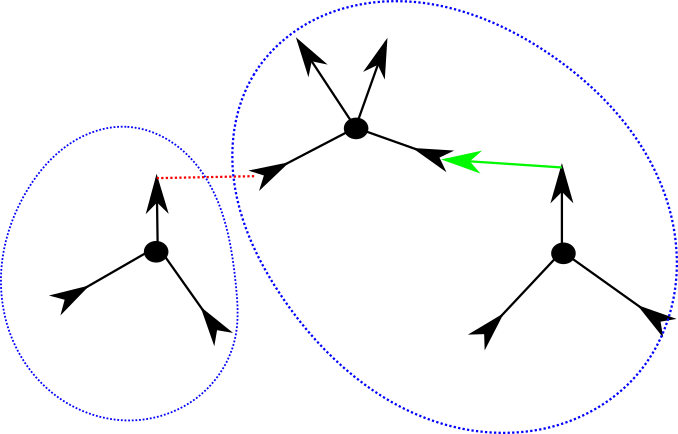}
    \caption{We cut along the red dashed lines to sub trees of cascades. For this figure each subtree is circled by dashed blue lines. The ECH index is additive along concatenation of such sub trees.}
    \label{cut tree like cascade}
    \end{figure}

The ECH index is additive with respect to concatenation of sub-trees. So the ECH index one conditions implies there are no matching along nondegenerate orbits, and we can use the correspondence theorem \ref{theorem:cobordism for genus 0} as before.

\section{Convex Toric Domains}\label{section convex}
In this section we show we can compute the ECH chain complex of convex toric domains via enumeration of $J$-holomorphic cascades. As there are many similarities with the case of concave toric domains, we will be brief in its treatment. 

Suppose $\Omega$ is a domain bounded by the horizontal segment from $(0,0)$ to $(a,0)$, the vertical segment from $(0,0)$ to $(0,b)$ and the graph of a concave function $f:[0,a] \rightarrow [0,b]$ so that $f(0)=b$ and $f(a)=0$. We further assume $f$ is smooth, $f'(0)$ and $f'(a)$ are irrational, $f'(x)$ is constant near $0$ and $a$, and $f''(x)<0$ whenever $f'(x)$ is rational, then we say $X_\Omega$ is a \textbf{convex toric domain}.

As in the case of a concave toric domain, the boundary of $X_\Omega$, written as $\partial X_\Omega$, is a contact 3-manifold diffeomorphic to $S^3$.
We now describe the Reeb orbits that appear in $\partial X_\Omega$. We also note their Conley Zehnder indices, having chosen the same trivializations as in \cite{beyondech}
\begin{enumerate}
    \item $\gamma_1 = \{ (z_1,0) \in \partial X_\Omega \}$. The orbit $\gamma_1$ is elliptic with rotation angle $-1/f'(a)$, hence $CZ(\gamma_1^k) = 2\floor{-k/f'(a)}+1$
    \item $\gamma_2 = \{ (0,z_2) \in \partial X_\Omega \}$. The orbit $\gamma_2$ has rotation angle $-f'(0)$, hence $CZ(\gamma_2^k)=2\floor{-kf'(0)}+1$.
    \item Let $x\in (0,a)$ be such that $f'(x)$ is rational. Then the torus described by $\{(z_1,z_2) | \mu(z_1,z_2) = (x,f(x))\}$ is a (positive) Morse-Bott torus. Each Reeb orbit has Robbin-Salamon index $+1/2$.
\end{enumerate}

\begin{definition}
A combinatorial generator is a quadruple $\tilde{\Lambda} = (\Lambda,\rho,m,n)$ where
\begin{enumerate}
    \item $\Lambda$ is a convex integral path from $(0,B)$ to $(A,0)$ such that the slope of each edge is in the interval $[f'(0),f'(a)]$.
    \item $\rho$ is a labeling of each edge of $\Lambda$ by $e$ or $h$.
    \item $m$ and $n$ are nonnegative integers.
\end{enumerate}
\end{definition}

Let $\Lambda_{m,n}$ denote the concatenation of the following sequence of paths:
\begin{enumerate}
    \item The highest polygonal path with vertices at lattice points from $(0,B+n+\floor{-mf'(0)})$ to $(m,B+n)$ which is below the line through $(m,B+n)$ with slope $f'(0)$.
    \item The image of $\Lambda$ under the translation $(x,y)\mapsto (x+m,y+n)$.
    \item The highest polygonal path with vertices at lattice points from $(A+m,n)$ to $(A+m+\floor{-n/f'(a)},0)$ which is below the line through $(A+m,n)$ with slope $f'(a)$.
\end{enumerate}

Let $\mcal{L}(\Lambda_{m,n})$ denote the number of lattice points bounded by the axes and $\Lambda_{m,n}$, including the lattice points on the edges of $\Lambda_{m,n}$. 
 We then define 
\[
I^{comb}(\Lambda_{m,n}) =2( \mcal{L}(\Lambda_{m,n})-1) - h(\Lambda)
\]
And the Chern class of $\Lambda_{m,n}$ is given by 

\[
c_\tau(\Lambda_{m,n}) = A+B+m+n. 
\]
\begin{theorem}
The ECH index of a holomorphic curve between two ECH generators is the difference of the $I^{comb}$ we associate to their corresponding combinatorial ECH generators.
\end{theorem}
\begin{proof}
The proof is a generalization of the computation in \cite{beyondech,intoconcave}. We briefly summarize this below. Let $\alpha$ denote a ECH orbit set. We consider $I(\alpha,\emptyset,Z)$ where $Z$ is the unique relative homology class that is represented by discs with boundary $\alpha$. Let $m,n$ denote the multiplicity of $\gamma_2,\gamma_1$ respectively in $\alpha$, and let $\Lambda$ be the resulting convex integral path defined by associating Reeb orbit sets to integral paths as in \cite{beyondech}. Then it suffices to show $I(\alpha,\emptyset,Z) = I^{comb}(\Lambda_{m,n})$. The computation is the same as the one in \cite{beyondech}, except the Conley-Zehnder index terms arising from $\gamma_1$ and $\gamma_2$ may not just be $1$ due to the fact their rotation angles $\theta$ need not be very close to zero. This is accounted for by the polygonal paths we append to image of $\Lambda$ under the translation $(x,y)\mapsto (x+m,y+n)$.

\end{proof}
\begin{theorem}
A nontrival $J_\dt$-holomorphic curve in a convex toric domain of ECH index one has genus zero. Here we use $J_\dt$ to mean we have perturbed away all Morse-Bott degeneracies.
\end{theorem}
\begin{proof}
We borrow the notation of the previous section, except here $e_+$ denotes the total multiplicity of elliptic Reeb orbits in $\alpha_+$ arising from Morse-Bott tori and $e_-$ denotes the total number of distinct elliptic Reeb orbits in $\alpha_-$ arising from perturbations of Morse-Bott tori. The Fredholm index of a connected $J$-holomorphic curve $C$ between two orbit sets $\alpha_+$ and $\alpha_-$ is given by
\begin{align*}
    Ind(C) =& 2g-2 +(e_+ + h_+ + k_m^++k_n^+)+(e_- + h_- + k_m^-+k_n^-)\\
    &+2(A_++B_+ +m_+ +n_+ -A_--B_--m_--n_-)\\
    & +e_+  -e_- \\
    & + (k_n^+ + k_m^+ +k_m^- + k_n^-) \\
    & + \sum_{i=1}^{k_n^+} 2\floor{- n_+^i /f'(a)} + \sum_{i=1}^{k_m^+} 2\floor{-m_+^i f'(0)} - \sum_{i=1}^{k_n^-} 2\ceil{- n_-^i /f'(a)} - \sum_{i=1}^{k_m^-} 2\ceil{-m_-^i f'(0)}.
\end{align*}
The same linking number relations as in \ref{thm:concave genus zero} holds in the case of convex toric domains; so similarly by considering the intersections of $C$ with the trivial cylinders at $\gamma_1$ and $\gamma_2$, we conclude 
\begin{equation*} 
A_+ + n_+ +\sum_{i=1}^{k_m^+}\floor{-m_+^i f'(0)} -A_--n_- -  \sum_{i=1}^{k_m^-}\ceil{-m_-^i f'(0)} \geq 0
\end{equation*}
and 
\[
B_+ +m_+ +\sum_{i=1}^{k_n^+} 2\floor{- n_+^i /f'(a)} - B_--m_- -\sum_{i=1}^{k_n^-} 2\floor{- n_-^i /f'(a)} \geq 0.
\]
Hence for $C$ to have genus nonzero it must not have any ends at $\gamma_1$ and $\gamma_2$.

The local energy inequality holds as before, to prove the no-crossing lemma, we can associate two polygonal paths $\Lambda_+$ and $\Lambda_-$ to ECH generators  $\alpha_+$ and $\alpha_-$ respectively. As before from index considerations the $x$ and $y$ intercepts of $\Lambda_+$ and $\Lambda_-$ agree. Hence as before we can choose points $(a,b)$ and $(c,d)$ where $\Lambda_+$ and $\Lambda_-$ intersect, and between these two points $\Lambda_-$ is strictly above $\Lambda_+$. As before we may choose $x_0 \in (a,c)$ so that $f'(x_0) = \frac{d-b}{c-a}$. Let the lattice point $(p',q')$ have the following property: it is a vertex on $\Lambda_-$, the edge to the left of this lattice point has slope greater than or equal to $f'(x_0)$, and the edge to the right of this vertex has slope less than $f'(x_0)$. Let $(p,q)$ denote a vertex of $\Lambda_+$ with the same property. We assume such a vertex $(p,q)$ exists and leave the case where such a vertex does not exist to later. Then consider $[F_{x_0+\ep}] = (q-q',p'-p)$. Now again the energy inequality says
\[
(q-q') + \frac{d-b}{c-a} (p'-p) \geq 0
\]
In this case, the point $(p,q)$ is closer to the line connecting $(a,b)$ and $(c,d)$ than $(p',q')$, but this time on the other side of the line. This means that 
\[
(p-p')(b-d)  + (c-a)  (q-q') < 0
\]
Comparing with the energy inequality we see a contradiction. Now if $(p,q)$ is in fact on the line connecting $(a,b)$ and $(c,d)$, then since we are computing $[F_{x_0+\ep}]$, we must have $p>p'$, from which we have
\[
\frac{d-b}{c-a} > \frac{q-q'}{p-p'}
\]
which is a contradiction.

With the no-crossing result at hand, we turn to the index formula. If $C$ had genus one, then 
\[
1 = 2e_+ + h_++h_-.
\]
As before we break this into cases. We must have $e_+=0$. If $h_+=1$ then $\Lambda_+ $ consists of a single edge, by no-crossing $\Lambda_-$ is either an identical edge or empty. We check either case cannot produce an ECH index 1 curve. $h_+$ cannot equal zero because then $\Lambda_+ =\emptyset$.

\end{proof}

Hence we concluded all ECH index one curves are index zero, a similar description of tree-like cascades shows we can use them to compute the ECH chain complex.
\appendix
\section{Appendix: Transversality Issues}
In this Appendix we describe some the transversality difficulties in the moduli space of cascades, even if all the appearing curves are somewhere injective. Note we are not claiming transversality is impossible, we are simply saying there are issues with the standard universal moduli space approach of transversality. We give some simple examples below to illustrate this.

Consider the universal moduli space of somewhere injective cascades, written as 
\[
\mcal{B} : = \{ (\cas{u},J) | \, \cas{u} \, \textup{is a} \, J \textup{-holomorphic cascade, and that all curves appearing in $\cas{u}$ are simple} \}.
\]
We explain why the standard proof that $\mcal{B}$ is a Banach manifold does not necessarily work. Given a cascade $\cas{u} \in \mcal{B}$, there are two evaluation maps $EV^+$ and $EV^-$ that map into a product of $S^1$, as in Definition \ref{def:transversality conditions}. The usual procedure to show that $\mcal{B}$ is a Banach manifold is to show the maps $EV^\pm$ are transverse to each other. 
However in complicated enough cascades, the same curve can appear in multiple different levels. An illustration is given in the figure below. Here we have a cascade of 5 levels. The red curve is a map $u: \Sigma \rightarrow  \bb{R} \times Y^3$, and the blue curve is a map $v: \Sigma' \rightarrow \bb{R} \times Y^3 $. Green horizontal arrows denote the upwards gradient flow, and the black horizontal lines denote Morse-Bott tori. Diamonds denote the critical points of $f$ on the Morse-Bott tori. For instance, one of the positive ends of the black curve ends on a critical point of $f$, and there is a chain of fixed trivial cylinders atop this end.
\begin{figure}[!ht] \label{cascade with 5 levels}
    \centering
    \includegraphics[width=.5\linewidth]{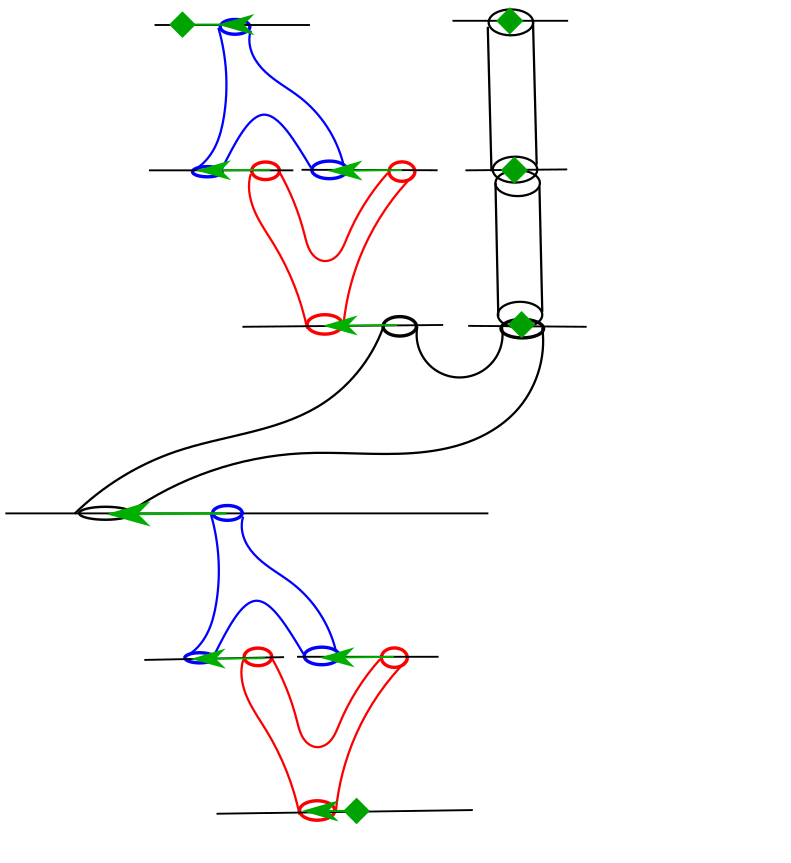}
    \caption{Cascade with 5 levels}
    \label{fig:5level}
    \end{figure}
This is an illustration of how the same curves can happen in the same cascade. To illustrate the transversality issue, we assume that the configuration consisting the red and blue curves (which we labelled $u$ and $v$) in figure \ref{repeat_red_and_blue} happens $n$ times in a cascade $\cas{u}$.
\begin{figure}[!ht]
    \centering
    \includegraphics[width=.5\linewidth]{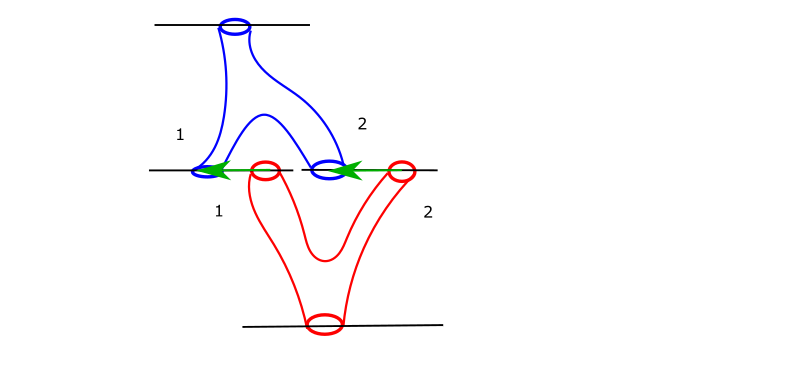}
    \caption{A repetitive pattern that can appear multiple times in a cascade.}
    \label{repeat_red_and_blue}
\end{figure}
We assume both $u$ and $v$ are rigid (we are allowed since we are working in the universal moduli space, in general more complicated things can still happen but the principle is the same). We label the $n$ identical copies of $u$ and $v$ as $u_i,v_i$ with $ i=1,...,n$. The two negative ends of $u_i$ and the two positive ends of $v_i$ are labelled by $1,2$, as shown in the figure. The remaining end of $u_i$ and $v_i$ is labelled $3$. We denote their evaluation maps by $ev(u_i,k)$ and $ev(v_i,k)$ where $k=1,2,3$. As a necessary condition for the $EV^+$ and $EV^-$ to be transverse, we must have 
\begin{equation} \label{equation:transverse}
\bigoplus (dev(u_i,1)+dev(v_i,1)+t_i, dev(u_i,2)+dev(v_i,1)+t_i): T\mcal{W}_u \oplus T\mcal{W}_v \bigoplus_{i=1,..,n} \bb{R} \longrightarrow  \bigoplus_{i=1,...,n} (TS^1\oplus TS^1)
\end{equation}
is surjective. Note $(t_1,...,t_n)\in \bigoplus_{i=1,..,n} \bb{R}$. The vector space $T\mcal{W}_u$ has the following description. Recall a neighborhood of (not necessarily $J$ holomorphic) curves near $u$ can be represented by $W^{2,p,d}(u^*TM) \oplus T\mcal{J} \oplus V_1 \oplus V_2 \oplus V_3$. Here $W^{2,p,d}(u^*TM)$ is the Sobolev space of vector fields on $u$ with exponential weight $e^{d|s|}$ near the cylindrical ends. $T\mcal{J}$ is a finite dimensional Teichmuller slice, and the vector spaces $V_i$ consist of asymptotically constant vectors near each of the cylindrical ends, which we labelled $1,2,3$ (see \cite{Yaocas,wendlauto}). Recalling the coordinate choices of Section \ref{degenerations} near Morse-Bott tori, the $V_i$ is spanned by vector fields of the form
\[
\beta \partial_z, \quad \beta \partial_a, \quad \beta \partial_x.
\]
$\beta$ here is a cutoff function that is one near a cylindrical neighborhood of a puncture and zero elsewhere.
We denote a triple of these vector fields in $V_i$ as $(r,a,p)_i$.

Then the vector space $\mcal{W}_u$ is given by
\[
\{(\xi,(r,a,p)_i,Y) \in W^{2,p,d}(u^*TM) \oplus T\mcal{J} \oplus V_1 \oplus V_2 \oplus V_3 \oplus T \mcal{I} |D\db_J (\xi + \sum_i(r,a,p)_i) +Y\circ Tu \circ j =0\}
\]
$D\db_J$ is the linearization of Cauchy Riemann operator along $u$ that includes deformation of the domain complex structure of $u$.
Here $T\mcal{I}$ denotes the Sobolev space that is the tangent space of all $\lambda$ compatible almost complex structures (we should choose a Sobolev space for this but that is unimportant for now). A similar expression holds for $T\mcal{W}_v$. We note the same $Y\in T\mcal{I}$ appears in the definition of $T\mcal{W}_v$ as well.
Now since $u$ is rigid for given $Y$ there exists a unique tuple $(\xi,(r,a,p)_i)$ (up to translation in the symplectization direction) so that $(\xi,(r,a,p)_i,Y)\in T\mcal{W}_u$. A similar statement holds for $\mcal{W}_v$. Conversely, given two tuples $(p_1(u),p_2(u),p_3(u))$ and $(p_1(v),p_2(v),p_3(v))$ (we use brackets to denote whether the vector field is living over $u$ or $v$, we can find $Y \in T\mcal{I}$ and $(\xi(u),(r(u),a(u))_i)$ and $(\xi(v),(r(v),a(v))_i)$ so that the tuples $(\xi(u),(r(u),a(u),p(u))_i,Y )\in T\mcal{W}_u$, and similarly for $T\mcal{W}_v$.
Hence we can think of the map described in Equation \ref{equation:transverse} as the following. Its imagine is spanned by vector fields of the form  
\[
\bigoplus_i (x_1+y_1+t_i,x_2+y_2+t_i)
\]
where $(x_1,y_1)$ and $(x_2,y_2)$ are arbitrary real numbers. We think of $x_1$ as $p_1(u)$, $x_2$ corresponding to $p_2(u)$, and likewise for $y$ and $p(v)$.
For given $n$ the domain has $2+n$ independent variables, but the target is $2n$ dimensional. Hence for large values of $n$ this space cannot be transverse.
\begin{proof}[Proof of Theorem \ref{thm:list of transversality conditions}]
We note if the above situation does not happen, then the usual proof that $\mcal{B}$ is a Banach manifold follows through. To be precise, if we let $\tilde{B}$ denote the universal moduli space so that

\begin{equation}
\tilde{B}  : =  \left\lbrace  (\cas{u},J)\;\middle|\;
  \begin{tabular}{@{}l@{}}
 $\cas{u}$ \, \textup{is a reduced} $J$ \textup{-holomorphic cascade as in Definition \ref{def:transversality conditions}};\\ \textup{in addition,  either all nontrivial curves}\\ \textup{ are distinct, or the cascade has less than or equal to 3 levels} 
   \end{tabular}
  \right\rbrace
 \end{equation}

Then $\tilde{B}$ is a Banach manifold, and for generic $J$, cascades satisfying the extra hypothesis of $\tilde{B}$ are transversely cut out living in moduli spaces given by the virtual dimension. 
In particular if we take as assumption after we perturb away the Morse-Bott degeneracy, all ECH index one curves degenerate (as reduced cascades) to reduced cascades of the form specified in $\tilde{B}$, then we can choose a $J$ so that the conditions \ref{def:transversality conditions} are satisfied for these cascades. A straightforward modification of the proofs in Sections \ref{Finite}, \ref{section:computing using cascades} shows the Morse-Bott chain complex $(C_*^{MB},\p_{MB})$ when we further restrict the differential to only consider cascades whose reduced versions can appear in $\tilde{B}$ is well defined and computes $ECH(Y,\xi)$. The only different part is showing the cascades counted by $\p_{MB}$ is finite. Consider the following. Suppose $\cas{u}_n$ is a sequence of cascades of the form allowed in $\tilde{B}$ and $\cas{u}_n \rightarrow \cas{u}$. Then for each $\cas{u}_n$ there is a sequence of $J_{\dt_n^m}$-holomorphic curves $v_n^m$ of ECH index one that converges to $\cas{u_n}$ as $m\rightarrow \infty$. We pass to a diagonal subsequence, which we denote by $v_n$, of ECH index one $J_{\dt_n}$-holomorphic curves that degenerate into $u$. By assumption, then the reduced version of $\cas{u}$ must be of the form allowed in $\tilde{B}$, and this concludes the proof of finiteness.
\end{proof}
\printbibliography
\end{document}